\documentclass[10pt]{article}

\usepackage[left=1in,top=1in,right=1in,bottom=1in,letterpaper]{geometry}

\usepackage[colorlinks]{hyperref}
\usepackage{url}
\usepackage{amsmath,amssymb,amsthm,amscd}
\usepackage{mathrsfs}
\usepackage{amsfonts}
\usepackage[utf8]{inputenc}
\usepackage{graphicx}
\usepackage{booktabs}
\everymath{\displaystyle}
\usepackage{indentfirst}
\usepackage{enumerate}
\usepackage{bm}
\usepackage{enumitem}
\usepackage{algorithm}
\usepackage[noend]{algpseudocode}
\usepackage{dsfont}
\usepackage{multirow}
\usepackage{threeparttable}
\usepackage{makecell}
\usepackage{colortbl}
\usepackage{comment}
\usepackage{xcolor}
\usepackage[font={small,it}]{caption}
\usepackage{xspace}
\usepackage{apptools}
\usepackage{chngcntr}
\usepackage{mathtools}

\usepackage{amsmath,bm,dsfont,amssymb,xspace}

\newtheorem{thm}{Theorem}
\newtheorem{remark}{Remark}
\newtheorem{assumption}{Assumption}
\newtheorem{example}{Example}
\usepackage{apptools}
\usepackage{booktabs}
\usepackage{threeparttable}
\AtAppendix{\counterwithin{thm}{section}}

\newcommand{\one}{\mathds{1}_n}

\newcommand{\vg}{{\mathbf{g}}}

\newcommand{\vx}{{\mathbf{x}}}
\newcommand{\vy}{{\mathbf{y}}}

\newcommand{\cF}{{\mathcal{F}}}
\newcommand{\cG}{{\mathcal{G}}}

\newcommand{\cO}{{\mathcal{O}}}

\newcommand{\EE}{\mathbb{E}}
\newcommand{\RR}{\mathbb{R}}

\newcommand{\pushsum}{{\sc Push-Sum}\xspace}
\newcommand{\pushonly}{{\sc Push-Only}\xspace}
\newcommand{\pushpull}{{\sc Push-Pull}\xspace}
\newcommand{\pushdg}{{\sc Push-DIGing}\xspace}
\newcommand{\pullonly}{{\sc Pull-Only}\xspace}

\newcommand{\extrapush}{{\sc EXTRA-Push}\xspace}
\newcommand{\dextra}{{\sc DEXTRA}\xspace}
\newcommand{\addopt}{{\sc ADD-OPT}\xspace}
\newcommand{\saddopt}{{\sc S-ADDOPT}\xspace}
\newcommand{\ab}{{\sc AB}\xspace}
\newcommand{\subpush}{{\sc Subgradient-Push}\xspace}
\newcommand{\sgdpush}{{\sc SGD-Push}\xspace}

\newcommand{\norm}[1]{\| #1 \|}
\newcommand{\ip}[1]{\left\langle#1\right\rangle}

\newcommand{\ie}{{\em i.e.\xspace}}

\definecolor{gan}{RGB}{128,128,255}

\newcommand{\pulldiag}{{\sc Pull-Diag}\xspace}

\newcommand{\revision}[1]{\textcolor{black}{#1}}

\newtheorem{theorem}[thm]{Theorem}
\newtheorem{lemma}[thm]{Lemma}
\newtheorem{proposition}[thm]{Proposition}

\newcommand{\email}[1]{\href{mailto:#1}{#1}}

\title{
On the Linear Speedup of the Push-Pull Method for Decentralized Optimization over Digraphs
}
\author{Liyuan Liang\thanks{Equal contribution. Liyuan Liang and Gan Luo are with School of Mathematics Science, Peking University (\email{liangliyuangg@gmail.com, 2200013522@stu.pku.edu.cn}), }
\and Gan Luo{$^*$}
\and Kun Yuan\thanks{Corresponding author. Kun Yuan is with Center for Machine Learning Research, Peking University (\email{kunyuan@pku.edu.cn})}}

\date{}

\begin{document}
\maketitle

\begin{abstract}
The linear speedup property is essential for demonstrating the advantage of distributed algorithms over their single-node counterparts. This paper studies the stochastic \pushpull method, a widely adopted decentralized optimization algorithm for directed graphs (digraphs). Unlike methods that rely on \pushonly or \pullonly communications, \pushpull avoids nonlinear correction and has shown superior empirical performance across various settings. Despite its success, whether \pushpull achieves linear speedup has remained an open question, revealing a significant gap between its empirical performance and theoretical understanding. To bridge this gap, we propose a novel multi-step descent analysis framework and prove that \pushpull indeed achieves linear speedup over arbitrary strongly connected digraphs. Our results provide a comprehensive theoretical understanding of the stochastic \pushpull method, finally aligning its theoretical convergence with empirical performance.
% The linear speedup property is essential for demonstrating the advantage of distributed algorithms over their single-node counterparts. In this paper, we study the stochastic \pushpull method, a widely adopted decentralized optimization algorithm over directed graphs (digraphs). Unlike methods that rely solely on row-stochastic or column-stochastic mixing matrices, \pushpull avoids nonlinear correction and has shown superior empirical performance across a variety of settings. However, the linear speedup property has not been generally established for \pushpull—revealing a significant gap between empirical success and limited theoretical understanding. To bridge this gap, we propose a novel analysis framework and prove that \pushpull achieves linear speedup over arbitrary strongly connected digraphs. Our results provide the comprehensive theoretical understanding for stochastic \pushpull, aligning its theory with empirical performance. 
Code: \href{https://github.com/pkumelon/PushPull}{https://github.com/pkumelon/PushPull}. 

\end{abstract}
\vspace{-4pt}

% \begin{keywords}
%   decentralized optimization, directed graph, push-pull, linear speedup
% \end{keywords}

\section{Introduction}
Scaling modern machine learning tasks to massive datasets and increasingly large models requires efficient distributed computation across multiple nodes. In distributed optimization, a network of $n$ nodes cooperatively solves  
\begin{align}\label{prob-general}
\min_{x\in \mathbb{R}^d}\quad f(x) := \frac{1}{n}\sum_{i=1}^n f_i(x) \quad \mbox{where} \quad f_i(x) = \mathbb{E}_{\xi_i \sim \mathcal{D}_i}[F_i(x; \xi_i)].
\end{align}
where each local function $f_i$ is defined by a data distribution $\mathcal{D}_i$. Each node has access only to stochastic gradients $\nabla F_i(x;\xi_i)$ from its local data, while communication is required to aggregate information across the network. In practice, the distributions $\{\mathcal{D}_i\}_{i=1}^n$ are heterogeneous, so the local losses $f_i$ generally differ. This heterogeneity poses significant challenges to both the design and analysis of distributed algorithms.  

A natural baseline for solving~\eqref{prob-general} is centralized stochastic gradient descent (SGD), where gradients are aggregated globally either through a central server~\cite{li2014scaling} or collective communication primitives such as Ring All-Reduce~\cite{patarasuk2009bandwidth}. While centralized SGD enjoys strong theoretical guarantees---including the celebrated linear speedup property in iteration complexity---its reliance on global averaging leads to significant communication overhead and creates a single point of failure at the server. These drawbacks limit scalability in large, heterogeneous systems.  

Decentralized SGD has emerged as a compelling alternative~\cite{chen2012diffusion,nedic2009distributed}. In this setting, each node maintains its own local model and communicates only with its neighbors. This localized communication substantially reduces bandwidth usage and alleviates latency bottlenecks. Moreover, decentralized methods are inherently robust to failures: as long as the communication graph remains connected, the network can continue to train without dependence on a single master node.  

In this paper, we consider decentralized optimization over \emph{directed graphs (digraphs)}. Directed communication naturally models many real-world systems, including heterogeneous robotic swarms with asymmetric links, sensor networks with one-sided message passing, and distributed deep learning environments where communication is constrained by bandwidth asymmetry. Directed networks are also attractive because certain classes of sparse digraphs, such as exponential graphs~\cite{assran2019stochastic}, EquiTopo~\cite{song2022communication}, and B-ary trees~\cite{you2024bary}, achieve both scalability and strong connectivity.  

\subsection{Decentralized algorithms over digraphs} In a connected network, the topology can be characterized by a mixing matrix. This matrix reflects how nodes are connected and is important for understanding how the network affects algorithm performance. For undirected networks, it is straightforward to construct symmetric and \textit{doubly stochastic} mixing matrices. For directed networks, achieving a doubly stochastic matrix is generally infeasible. Consequently, mixing matrices in such cases are typically either \textit{column-stochastic}~\cite{Olshevsky2015Dist, Nedic2017pushdiging} or \textit{row-stochastic}~\cite{sayed2014adaptive, mai2016distributed}, but not both.

Building on the connectivity pattern of the underlying network topology, decentralized optimization algorithms over directed graphs can be broadly categorized into three families in the literature. The first family relies on the \pushsum technique \cite{kempe2003gossip, tsianos2012push}, which employs column-stochastic mixing matrices for all communications. Accordingly, we refer to this class as \pushonly algorithms in this paper. When exact gradients $\nabla f_i$ are available, the well-known \subpush algorithm \cite{Olshevsky2015Dist, tsianos2012push} converges to the desired solution, albeit with relatively slow sublinear rates even under strong convexity. More advanced methods such as \extrapush \cite{zeng2017extrapush}, \dextra \cite{xi2017dextra}, \addopt \cite{xi2017add}, and \pushdg \cite{Nedic2017pushdiging} improve convergence rate by mitigating the effects of data heterogeneity. Stochastic variants, including \sgdpush \cite{assran2019stochastic} and \saddopt \cite{qureshi2020s}, extend these approaches to settings with stochastic gradients.

The second family consists of algorithms that use only row-stochastic matrices for communication, referred to as the \pullonly methods. Just as \pushsum underpins the design of \pushonly algorithms, the \pulldiag gossip protocol~\cite{mai2016distributed} serves as the foundation for \pullonly methods. Reference \cite{mai2016distributed} adapted the distributed gradient descent algorithm to the \pullonly setting. Subsequently, gradient tracking techniques were extended to this framework by~\cite{li2019row, FROST-Xinran}, while momentum-based \pullonly\ gradient tracking methods were developed in~\cite{ghaderyan2023fast, lu2020nesterov}.

The third family comprises algorithms that alternate between row-stochastic and column-stochastic mixing matrices, commonly referred to as the \pushpull method~\cite{pu2020push} or the \ab method~\cite{xin2018linear}. In this framework, a row-stochastic matrix is used to exchange model parameters, while a column-stochastic matrix is employed to share gradient information. As demonstrated in~\cite{xin2018linear, pu2020push, nedic2023ab}, the \pushpull family typically achieves significantly faster convergence rates than algorithms relying solely on column- or row-stochastic matrices. Reference \cite{xin2019SAB,zhao2023asymptotic} extended the \pushpull method to stochastic optimization settings, while \cite{you2024bary} improved its convergence rate on B-ary tree network topologies. Accelerated or constrained variant of the \pushpull method have also been explored in~\cite{nguyen2023accelerated, gong2023push}.

\subsection{Linear speedup in iteration complexity} The \emph{linear speedup} is a central property for distributed optimization. \revision{In the smooth nonconvex setting considered in this paper, $\epsilon$-accuracy for an iterate sequence $\{x^{(t)}\}_{t=0}^{T-1}$ is measured by the stationarity criterion $\min_{0\le t\le T-1}\EE[\norm{\nabla f(x^{(t)})}^2]\le \epsilon$. Linear speedup then means that the number of iterations $T$ required to achieve this criterion decreases linearly with the number of nodes $n$.} Centralized SGD achieves this scaling, converging in $\mathcal{O}(1/(n\epsilon^2))$ iterations. Without this property, the benefit of parallelization is limited, as increasing the number of nodes would not yield proportional gains in iteration complexity.  

For decentralized algorithms on \emph{undirected networks}, linear speedup is well established. The first rigorous analysis for decentralized SGD was given in~\cite{lian2017can}. Building on this work, a range of advanced algorithms---including Exact Diffusion/NIDS~\cite{yuan2018exact,li2019decentralized} and gradient tracking methods~\cite{xu2015augmented,di2016next,qu2017harnessing,Nedic2017pushdiging}---have been proven to achieve linear speedup~\cite{pu2021distributed}. Time-varying network extensions are also addressed in~\cite{koloskova2020unified}.  

For \emph{digraphs}, theoretical understandings are less  explored. For \pushonly algorithms with column-stochastic mixing, linear speedup has been established in \cite{assran2019stochastic,kungurtsev2023decentralized}, with subsequent refinements achieving optimal iteration complexity in \cite{liang2025understanding}. For \pullonly algorithms with row-stochastic mixing, convergence has been extensively analyzed under various settings \cite{mai2016distributed,li2019row,FROST-Xinran,ghaderyan2023fast,lu2020nesterov}, but these studies did not establish linear speedup because row-stochastic matrices inherently introduce a deviation from the globally averaged gradient. Only recently, linear speedup for \pullonly algorithms was proven in \cite{liang2025achieving}. These advances reveal the  effectiveness of both \pushonly and \pullonly methods over digraphs.

Despite this progress, establishing the linear speedup property for \pushpull\ algorithms—those combining both row- and column-stochastic matrices—has remained an open problem. While \pushpull\ methods consistently outperform \pushonly\ and \pullonly\ variants in practice, their theoretical understanding remains limited. Existing analyses focus mostly on strongly convex settings~\cite{xin2018linear, xin2019SAB, pu2020push, you2024bary}, and provide no guarantees for nonconvex or stochastic optimization. 
A recent work~\cite{ying2025exact} establishes guarantees for \pushpull\ in a specific non-convex federated learning setting, but does not prove the linear speedup property. 
\revision{
The most relevant work~\cite{you2025distributed} establishes linear speedup for the Spanning Tree Push-Pull~(STPP) method, which is a variant of Push-Pull based on specialized spanning-tree topology and mixing designs. Its analysis provides a clear topology-dependent polynomial transient-time characterization. This is an important advance in sharpening the finite-time guarantees of Push-Pull type methods. However, it does not resolve the question of whether the original \pushpull\ algorithm, with general row-/column-stochastic mixing matrices, enjoys linear speedup under general directed topologies.}
Table~\ref{table::comparision} provides a summary of existing results on decentralized algorithms over digraphs.

\subsection{Motivation and main results} The above discussion highlights a striking gap between \emph{theory and practice}: \pushpull algorithms consistently outperform \pushonly and \pullonly methods in empirical studies, yet their scalability theory is far less developed. This gap motivates two fundamental questions:  

\vspace{2mm}
\begin{itemize}
\item[Q1.] {What is the intrinsic difference between the \pushpull method and the other two families, \ie, \pushonly and \pullonly? Why is the theoretical analysis of \pushpull more challenging?}

\vspace{1mm}
\item[Q2.] Can push--pull achieve linear speedup on general strongly connected digraphs under standard assumptions?  
\end{itemize}

\vspace{2mm}
In this work, we address both questions. Regarding Question~Q1, we show that \pushonly\ and \pullonly\ can be reformulated to align closely with centralized SGD, making their linear speedup analysis relatively straightforward. In contrast, \pushpull\ exhibits a persistent, non-vanishing deviation from centralized SGD dynamics, rendering the conventional single-step descent analysis framework ineffective.  

For Question~Q2, we introduce a novel \emph{multi-step descent analysis framework} that aggregates errors over blocks of iterations. This framework uncovers a telescoping structure in the gradient-tracking recursion: intermediate gradient noise cancels out, while boundary terms are exponentially attenuated by mixing dynamics. Consequently, the residual error remains bounded at a constant scale, enabling us to establish the desired $\mathcal{O}(1/(n\epsilon^2))$ iteration complexity for \pushpull.  

To the best of our knowledge, this paper presents the \emph{first general proof} that \pushpull\ achieves linear speedup on arbitrary strongly connected digraphs in the stochastic nonconvex setting. In particular, our multi-step descent analysis framework not only bridges the gap between theory and practice for \pushpull\ but also introduces techniques that may be of independent interest for the broader study of decentralized algorithms. A summary of existing results is provided in Table~\ref{table::comparision}.  

\begin{table}[t]
\centering
\setlength{\tabcolsep}{9pt} %change here
\caption{Convergence rate comparison of representative decentralized algorithms in non-convex and stochastic settings over directed graphs. ``Conv. Rate'' refers to the asymptotic rate as \( T \to \infty \); ``Iter. Complex.'' denotes the number of iteration requried to achieve $\epsilon$-accuracy; ``S.C.'' denotes ``Strongly Connected''; ``R-Sto.'', ``C-Sto.'' and ``R/C Alt.'' refer to ``Row-Stochastic'', ``Column-Stochastic'' and ``Row/Column Alternate'', respectively. Note that References~\cite{xin2019SAB,zhao2023asymptotic} study the Push-Pull algorithm under strongly convex objectives, rather than the non-convex setting. $\sigma$ denotes gradient noise.} 
\label{table::comparision}
\begin{tabular}{lcclc}
\toprule
\multicolumn{1}{c}{\textbf{Algorithms}} & \multicolumn{1}{c}{\textbf{\begin{tabular}[c]{@{}c@{}}Conv.\\ Rate\end{tabular}}} & \multicolumn{1}{c}{\textbf{\begin{tabular}[c]{@{}c@{}}Iter.\\ Complex.\end{tabular}}} & \multicolumn{1}{c}{\textbf{Topology}} & \textbf{\begin{tabular}[c]{@{}c@{}}Mixing\\ Matrix\end{tabular}} \\ \midrule
\sc{\small Centralized SGD} & $\mathcal{O}(\frac{\sigma}{\sqrt{nT}})$ & $\mathcal{O}(\frac{\sigma^2}{n\epsilon^2})$ & Star graph & $-$ \\ \midrule
\sc {\small Pull-Diag-GT}~\cite{liang2025achieving} & $\mathcal{O}(\frac{\sigma}{\sqrt{nT}})$ & $\mathcal{O}(\frac{\sigma^2}{n\epsilon^2})$ & S.C. digraph & R-Sto. \vspace{1mm} \\
\sc {\small Push-Diging}~\cite{liang2025understanding} & $\mathcal{O}(\frac{\sigma}{\sqrt{nT}})$ & $\mathcal{O}(\frac{\sigma^2}{n\epsilon^2})$ & S.C. digraph & C-Sto. \\ \midrule
{\small \pushpull}~\cite{xin2019SAB} & $-$ & $\;\;\;-\;\;$ & S.C. digraph & R/C Alt. \\
{\small \pushpull}~\cite{zhao2023asymptotic} & $-$ & $\;\;\;-\;\;$ & S.C. digraph & R/C Alt. \\ 
{\small \pushpull}~\cite{ying2025exact} & $\mathcal{O}(\frac{\sigma}{\sqrt{T}})$ & $\mathcal{O}(\frac{\sigma^2}{\epsilon^2})$ & S.C. digraph & R/C Alt. \\ 
{\small \pushpull}~\cite{you2024bary} & $\mathcal{O}(\frac{\sigma}{\sqrt{nT}})$ & $\mathcal{O}(\frac{\sigma^2}{n\epsilon^2})$ & B-ary Tree & R/C Alt. \\ 
{\small \pushpull} \textbf{(Ours)} & $\mathcal{O}(\frac{\sigma}{\sqrt{nT}})$ & $\mathcal{O}(\frac{\sigma^2}{n\epsilon^2})$ & S.C. digraph & R/C Alt. \\ 
\bottomrule
\end{tabular}
\vspace{-4mm}
\end{table}

The remainder of the paper is organized as follows. Section~\ref{sec:notation} introduces the notations, assumptions, and performance metrics used throughout the paper. Section~\ref{sec:Preliminary} reviews the \pushonly, \pullonly, and \pushpull\ methods and provides insights into the derivation of their linear speedup properties. Section~\ref{sec:linear speedup} presents our novel multi-step descent analysis framework and explains the intuition behind its role in enabling linear speedup analysis. Section~\ref{sec:linear speedup 2} formally develops the supporting lemmas and states the main convergence theorem. Section~\ref{sec:experiment} reports extensive numerical experiments that validate the ability of the \pushpull\ method to achieve linear speedup. Finally, Section~\ref{sec:conclusion} concludes the paper.

\section{Notations, Algorithms, Assumptions and Metrics}\label{sec:notation}
This section introduces notations, assumptions, and metrics that will be used throughout the paper.
% as well as the \pushpull algorithm.

\textbf{Notations.} The directed network is denoted by $\mathcal{G} = (\mathcal{N}, \mathcal{E})$, where $\mathcal{N} = \{1,2,\ldots,n\}$ is the set of nodes and $\mathcal{E}$ the set of directed edges. An edge $(j,i)\in\mathcal{E}$ indicates that node $j$ can send information to node $i$. The graph is \textit{strongly connected} if a directed path exists between every pair of nodes. We say $\mathcal{G}$ is the underlying topology of a nonnegative weight matrix $A=[a_{ij}]$ if $a_{ij}>0$ precisely when $(j,i)\in\mathcal{E}$. Let $\one\in\mathbb{R}^n$ denote the all-ones vector and $[n]:=\{1,2,\ldots,n\}$. For a vector $v$, $\mathrm{Diag}(v)$ denotes the diagonal matrix with $v$ on its diagonal, while $\min v$ and $\max v$ denote its smallest and largest entries, respectively.

We let $A$ denote a row-stochastic matrix ($A\one=\one$) and $B$ a column-stochastic matrix ($B^\top\one=\one$), with Perron vectors~\cite{Perron1907} $\pi_A$ and $\pi_B$, respectively. Define $A_\infty := \one \pi_A^\top$ and $B_\infty := \pi_B \one^\top$. 

Additionally, we define \(n \times d\) matrices
\begin{align*}
                \vx^{(t)} &\hspace{-0.5mm}:=\hspace{-0.5mm} [( {x}_1^{(t)})^\top\hspace{-0.5mm}; ( {x}_2^{(t)})^\top\hspace{-0.5mm}; \cdots; ( {x}_n^{(t)})^\top\hspace{-0.5mm}]  \\
\hspace{-0.5mm} \vg^{(t)} &\hspace{-0.5mm}:=\hspace{-0.5mm} [\nabla F_1( {x}_1^{(t)};\hspace{-0.5mm}\xi_1^{(t)})^\top\hspace{-0.5mm};\hspace{-0.5mm} \cdots\hspace{-0.5mm};\hspace{-0.5mm} \nabla F_n( {x}_n^{(t)};\xi_n^{(t)})^\top\hspace{-0.5mm}] \\
\nabla f(\vx^{(t)}) &\hspace{-0.5mm}:=\hspace{-0.5mm} [\nabla f_1({x}^{(t)}_1)^\top;
\cdots;\nabla f_n({  x}^{(t)}_n)^\top] 
\end{align*}
where subscripts index nodes and superscripts index iterations. Let $\cF_t$ be the filtration containing all information up to $\vx^{(t)}$, excluding the stochastic gradients at iteration $t$. We use $\EE_t[\cdot]=\EE[\cdot\mid \cF_t]$ for conditional expectation. Norms are denoted as follows: $\|\cdot\|$ for the vector $\ell_2$-norm, $\|\cdot\|_F$ for the Frobenius norm, and $\|\cdot\|_2 := \max_{\|v\|=1}\|Av\|$ for the matrix spectral norm.

\vspace{0.5mm}
\textbf{Push-Pull algorithm.} The \pushpull algorithm, initially proposed for deterministic decentralized optimization~\cite{pu2020push,xin2018linear}, can be naturally extended to the stochastic gradient setting~\cite{xin2019SAB}:
\begin{subequations}
\label{eq:wenads}
\begin{align}
    \vx^{(t+1)}&=A\vx^{(t)}-\alpha \vy^{(t)}\label{eq:pushpull-1},\\
\vy^{(t+1)}&=B\vy^{(t)}+\vg^{(t+1)}-\vg^{(t)},\label{eq:pushpull-2}
\end{align}
\end{subequations}
where \revision{$\vg^{(t)}$ denotes the stacked stochastic gradients, and} $A$, $B$ are row-stochastic and column-stochastic matrices determined \revision{respectively} by the underlying digraph. For initialization, we let  $x_1^{(0)}=x_2^{(0)}=\cdots=x_n^{(0)}=x^{(0)}$ and $\vy^{(0)}=\vg^{(0)}$. The row-stochastic mixing matrix \( A \) in~\eqref{eq:pushpull-1} facilitates consensus among the local variables maintained by each node, whereas the column-stochastic mixing matrix \( B \) in~\eqref{eq:pushpull-2} enables the computation of a globally averaged gradient. In the absence of gradient noise, \ie, when 
\( \vg^{(t)} = \nabla f(\vx^{(t)})
% [\nabla f_1(x_1^{(t)})^\top;\, \cdots;\, \nabla f_n(x_n^{(t)})^\top] 
\),
it has been shown in~\cite{pu2020push,xin2018linear} that \( \vx^{(t+1)} \) converges to 
\( \vx^\star = [(x^\star)^\top;\cdots;(x^\star)^\top] \in \mathbb{R}^{n \times d} \) for strongly convex problems, where \( x^\star \) is the solution to problem~\eqref{prob-general}.

\vspace{0.5mm}
\textbf{Assumptions.} 
We first specify the underlying graphs (\(\mathcal{G}_A\), \(\mathcal{G}_B\)) and the corresponding weight matrices (\(A\), \(B\)) used in the stochastic \pushpull method.
\begin{assumption}[\sc \small Communication Graph]\label{ass:graph}
     Assume $\cG_A$ and $\cG_B$ to be strongly connected digraphs (see the definition in notations) with $n$ nodes.
\end{assumption}

\begin{assumption}[\sc \small Weight Matrix]\label{ass:matrix}
    The weight matrix \(A\) is row-stochastic with underlying digraph \(\mathcal{G}_A\), and the weight matrix \(B\) is column-stochastic with underlying digraph \(\mathcal{G}_B\). Both \(A\) and \(B\) have positive trace.
\end{assumption}

Assumptions~\ref{ass:graph} and~\ref{ass:matrix} are standard in the analysis of the \pushpull method~\cite{nedic2023ab,xin2019SAB,nguyen2023accelerated}. In our paper, their primary role is to facilitate the derivation of Proposition~\ref{prop:perron}. Strictly speaking, these assumptions could be replaced as long as Proposition~\ref{prop:perron} remains valid, without affecting the main convergence results.

\begin{proposition}[\sc \small Exponential Decay]\label{prop:perron}
For \(A\) and \(B\) satisfying Assumptions~\ref{ass:graph} and~\ref{ass:matrix}, there exist entrywise positive Perron vectors \(\pi_A, \pi_B \in \mathbb{R}^n\) such that
\[
\pi_A^\top A = \pi_A^\top, \quad B \pi_B = \pi_B, \quad \norm{\pi_A}_1=\norm{\pi_B}_1= 1, \quad \pi_A^\top\pi_B >0.
\]
Besides, there exist constants $p_A, p_B>0$, $\lambda_A,\lambda_B\in [0,1)$ such that 
\begin{equation}\label{eq:exp decay}
    \norm{A^t-A_\infty}_2\le p_A\lambda_A^t,\quad\norm{B^t-B_\infty}_2\le p_B\lambda_B^t,\quad \forall k\ge 0.
\end{equation}

\end{proposition}
\begin{proof}
    The proof follows directly from the Perron-Frobenius theorem~\cite{Perron1907}.
\end{proof}
\revision{Here, $\lambda_A$ and $\lambda_B$ characterize the exponential mixing rates of $A$ and $B$, respectively. Under Assumption~2, Perron--Frobenius theory implies that the convergence rates are governed by the subdominant eigenvalue moduli. In particular, letting
\(
\rho_2(A)=\max_{i\ge 2}|\lambda_i(A)|,\,
\rho_2(B)=\max_{i\ge 2}|\lambda_i(B)|,
\)
one may take $\lambda_A\in(\rho_2(A),1)$ and $\lambda_B\in(\rho_2(B),1)$ in the corresponding mixing bounds. }
We next state assumptions regarding stochastic gradients and function smoothness.
\begin{assumption}[\sc {\small Gradient Oracle}]\label{ass:gradient oracle}
    There exists a constant $\sigma\ge 0$ such that
\begin{align*}
\EE_t[\nabla F_i(x_i^{(t)};\xi_i^{(t)})]&=\nabla f_i(x_i^{(t)}) \quad \mbox{and} \quad
\EE_t[\norm{\nabla F_i(x_i^{(t)};\xi_i^{(t)})-\nabla f_i(x_i^{(t)})}_F^2]\le \sigma^2
\end{align*}
for any $i\in [n]$ and $t\ge 0$. In addition, we assume that stochastic gradients at different nodes are independent for any $t\ge 0$ and $1\le i<j\le n$.
\end{assumption}

\begin{assumption}[\sc {\small smoothness}]\label{ass:smooth}
% There exist positive constants $L, \Delta$ such that
% \begin{align*}
% \norm{\nabla f_i(x)-\nabla f_i(y)} \le L\norm{x-y} \quad \mbox{and} \quad f_i(x)-\inf_{x}f_i(x) \le \Delta, \quad \forall i\in [n].
% \end{align*}
\revision{There exist positive constants $L, \Delta$ such that}
\begin{align*}
\revision{\norm{\nabla f_i(x)-\nabla f_i(y)} \le L\norm{x-y}, \quad \forall x,y\in \RR^d,\ \forall i\in [n],}
\end{align*}
\revision{and the initialization point $x^{(0)}$ satisfies}
\begin{align*}
\revision{f_i(x^{(0)})-\inf_{x}f_i(x) \le \Delta,\quad \forall i\in [n].}
\end{align*}
\end{assumption}

\textbf{Metrics.} Various effective metrics have been proposed in literautre to capture the impact of digraphs on decentralized stochastic optimization. Let \(\Pi_A := \mathrm{Diag}(\pi_A)\) and \(\Pi_B := \mathrm{Diag}(\pi_B)\) where $\pi_A$ and $\pi_B$ are defined in Proposition~\ref{prop:perron}. Given a column-stochastic mixing matrix \(B\), references~\cite{xin2019SAB,liang2025understanding} introduce \(\beta_B := \norm{\Pi_B^{-1/2}(B - B_\infty)\Pi_B^{1/2}}_2 \in [0, 1)\) as a measure of graph connectivity, and define \(\kappa_B := \max \pi_B / \min \pi_B\) to reflect the deviation of the equilibrium distribution \(\pi_B\) from the uniform distribution \(n^{-1}\one\). Similarly, reference~\cite{liang2025understanding} introduces \(\beta_A := \norm{\Pi_A^{1/2}(A - A_\infty)\Pi_A^{-1/2}}_2 \in [0, 1)\) and \(\kappa_A := \max \pi_A / \min \pi_A\) to characterize the influence of the row-stochastic mixing matrix \(A\) on decentralized algorithms. 

This paper introduces a new set of metrics to facilitate the convergence analysis of the \pushpull method. When Proposition~\ref{prop:perron} holds, we define:
\begin{align}
    s_A &:= \max\big\{\sum_{i=0}^\infty \norm{A^i - A_\infty}_2, \sum_{i=0}^\infty \norm{A^i - A_\infty}_2^2\big\},\label{eq:s_A definition}\\
    s_B &:= \max\big\{\sum_{i=0}^\infty \norm{B^i - B_\infty}_2,\sum_{i=0}^\infty \norm{B^i - B_\infty}_2^2\big\},\label{eq:s_B definition} \vspace{2mm}\\ 
    c &:= n\, \pi_A^\top \pi_B.\label{eq:c definition}
\end{align}
By the exponential decay property~\eqref{eq:exp decay}, both $s_A$ and $s_B$ are finite. In this paper, the influence of the mixing matrices is described solely through $s_A$, $s_B$, $c$, and $n$. Their connection to previously introduced metrics is given in the next proposition.  

\begin{proposition}[\sc\small Relation with existing metrics]\label{prop:s_A,s_B}
Suppose Assumption~\ref{ass:graph} and \ref{ass:matrix} hold.
The following inequalities on $s_A$, $s_B$ and $c$ hold:
    \begin{align*}
    1\le s_A\le \frac{M_A^2(1+\frac{1}{2}\ln(\kappa_A))}{1-\beta_A},&\quad 1\le s_B \le \frac{M_B^2(1+\frac{1}{2}\ln(\kappa_B))}{1-\beta_B},\\
    n\max\{\min \pi_A, \min \pi_B\}&\le  \,c \le n\min\{\max \pi_A, \max \pi_B\}. 
\end{align*}
where \( M_A := \max_{t \ge 0} \|A^t - A_\infty\|_2 \) and \( M_B := \max_{t \ge 0} \|B^t - B_\infty\|_2 \) are finite numbers which can be proved to be no larger than $\sqrt{n}$. 
\end{proposition}
\begin{proof}
    See Appendix~\ref{app:sec:rolling sum}.
\end{proof} 

\section{Intrinsic Differences between Push-Only, Pull-Only, and Push-Pull}\label{sec:Preliminary}
This section provides a concise overview of the linear speedup analysis for the \pushonly\ and \pullonly\ algorithms, clarifying their key differences from \pushpull. We also highlight the limitations of existing analyses and explain why they cannot establish the linear speedup property of \pushpull.

\vspace{1mm}
\subsection{Push-Only and Pull-Only align with centralized SGD}
Many decentralized algorithms---such as decentralized SGD~\cite{koloskova2020unified}, gradient tracking~\cite{xu2015augmented,di2016next,qu2017harnessing}, 
Push-Diging~\cite{Nedic2017pushdiging}, Gradient-Pull~\cite{mai2016distributed}, and FROST~\cite{FROST-Xinran}---are designed to emulate centralized SGD in a decentralized setting. These algorithms can be reformulated as
\begin{align}
    \hat{x}^{(t+1)} = \underbrace{\hat{x}^{(t)} - \alpha \bar{g}^{(t)}}_{\text{centralized SGD}} + \epsilon^{(t)}, \label{eq:coordinate}
\end{align}
where $\hat{x}^{(t)} = \sum_{i=1}^n v_i^{(t)} (x_i^{(t)})^\top \in \mathbb{R}^{1 \times d}$ is a convex combination of the local decision variables $\{x_i^{(t)}\}_{i=1}^n$, with $v_i^{(t)}$ denoting the combination weights. The term $\bar{g}^{(t)} = \tfrac{1}{n}\one^\top \vg^{(t)}$ is the globally averaged stochastic gradient, and $\epsilon^{(t)}$ captures the residual error due to decentralization. When $\epsilon^{(t)}=0$, the recursion in~\eqref{eq:coordinate} reduces to centralized SGD, whose linear speedup property is well established (see Table~\ref{table::comparision}).

For \pushonly\ and \pullonly\ algorithms, the residual  $\epsilon^{(t)}$ is either exactly zero or decays exponentially to zero. This ensures that the trajectory of the weighted average of local iterates closely tracks that of centralized SGD. Once the residual is sufficiently small, the linear speedup property follows directly from the standard analysis of centralized SGD. We illustrate this alignment with two examples.

\begin{example}[\sc {\small Push-DiGing}~\cite{Nedic2017pushdiging}]\label{exp1}
Push-DiGing is a \pushonly\ method, whose iterates are given by
\begin{subequations}
    \begin{align}
    \vx^{(t+1)} &= A_t\vx^{(t)} - \alpha V_{t+1}^{-1}\vy^{(t)}, \label{eq:pushdg-1}\\
    \vy^{(t+1)} &= B\vy^{(t)} + \vg^{(t+1)} - \vg^{(t)},
    \end{align}
\end{subequations}
where $B$ is column-stochastic, $v_t = B^t\one$, $V_t=\mathrm{Diag}(v_t)$, $A_t=V_{t+1}^{-1}BV_t$, and $\vg^{(t)}$ stacks local stochastic gradients. Left-multiplying~\eqref{eq:pushdg-1} by $n^{-1} v_{t+1}^\top$ yields
\[
n^{-1}v_{t+1}^\top\vx^{(t+1)} = n^{-1}v_t^\top\vx^{(t)} - \alpha \bar{g}^{(t)}.
\]
Defining $\hat{x}^{(t)} = n^{-1} v_t^\top \vx^{(t)}$, the recursion matches~\eqref{eq:coordinate} with $\epsilon^{(t)}=0$.
\end{example}

\vspace{1mm}
\begin{example}[\sc {\small FROST}~\cite{li2019row, FROST-Xinran}] \label{exp2}
FROST/PullDiag-GT is a \pullonly\ method based on a row-stochastic matrix $A$. Its updates are
\begin{subequations}
    \begin{align}
    \vx^{(t+1)} &= A\vx^{(t)} - \alpha \vy^{(t)}, \label{eq:pulldiag-gt-1}\\
    \vy^{(t+1)} &= A\vy^{(t)} + D_{t+1}^{-1}\vg^{(t+1)} - D_t^{-1}\vg^{(t)},
    \end{align}
\end{subequations}
where $D_t=\mathrm{Diag}(A^t)$, $\vg^{(t)}$ is the stacked stochastic gradient, and $\vy^{(0)}=\vg^{(0)}$. Left-multiplying~\eqref{eq:pulldiag-gt-1} by $n^{-1}\pi_A^\top$ gives
\[
n^{-1}\pi_A^\top\vx^{(t+1)} = n^{-1}\pi_A^\top\vx^{(t)} - \alpha \bar{g}^{(t)} + \epsilon^{(t)},
\]
where $\epsilon^{(t)} = \alpha n^{-1}(\one^\top - \pi_A^\top D_t^{-1})\vg^{(t)}$ decays exponentially fast~\cite[Lemma 12]{liang2025achieving}. Setting $\hat{x}^{(t)} = n^{-1}\pi_A^\top\vx^{(t)}$ recovers the form~\eqref{eq:coordinate}.
\end{example}

\vspace{1mm}
\subsection{Push--Pull does not explicitly align with centralized SGD}
Unlike \pushonly\ and \pullonly, the \pushpull\ method cannot be perfectly aligned with centralized SGD. Although it can be written in the form of \eqref{eq:coordinate}, it suffers from a non-vanishing error term. Recall the \pushpull\ recursion~\eqref{eq:wenads}. Left-multiplying~\eqref{eq:pushpull-2} by $\one^\top$ yields
\begin{align}\label{2hasdas}
\one^\top \vy^{(t+1)} = \one^\top \vy^{(t)} + \one^\top \vg^{(t+1)} - \one^\top \vg^{(t)} = \one^\top \vg^{(t+1)},
\end{align}
where the last equality follows from the initialization $\vy^{(0)}=\vg^{(0)}$. Left-multiplying step~\eqref{eq:pushpull-1} by $\pi_A^\top$, we obtain
\begin{align*}
    \pi_A^\top \vx^{(t+1)} &= \pi_A^\top A\vx^{(t)} - \alpha \pi_A^\top B_\infty \vy^{(t)} - \alpha \pi_A^\top(I-B_\infty)\vy^{(t)} \\
    &= \pi_A^\top \vx^{(t)} - \alpha \pi_A^\top \pi_B \one^\top \vy^{(t)} - \alpha \pi_A^\top(I-B_\infty)\vy^{(t)}.
\end{align*}
Letting $\hat{x}^{(t)} = \pi_A^\top\vx^{(t)}$ and using~\eqref{2hasdas}, we obtain
\begin{align}\label{eq:push-pull-centralized}
\hat{x}^{(t+1)} = \hat{x}^{(t)} - c\alpha \bar{g}^{(t)} + \epsilon^{(t)}, 
\quad \text{where} \quad \epsilon^{(t)} = -\alpha \pi_A^\top(I-B_\infty)\vy^{(t)},
\end{align}
with $c=n\pi_A^\top\pi_B$ (see Definition \eqref{eq:c definition}) and $\bar{g}^{(t)} = n^{-1}\one^\top\vg^{(t)}$. Unlike Push-DiGing (Example~\ref{exp1}), where $\epsilon^{(t)}=0$, or FROST (Example~\ref{exp2}), where $\epsilon^{(t)}$ decays exponentially, the error term $\pi_A^\top(I-B_\infty)\vy^{(t)}$ in \pushpull\ does not vanish but instead stabilizes at a non-negligible constant. As shown in Figure~\ref{fig::synthetic_error_example}, this persistent error remains regardless of network size or topology, making the linear speedup analysis of \pushpull\ substantially more challenging.

\begin{figure}[t]
    \centering
    
    \vspace{0.02\linewidth}
    
    \begin{minipage}{0.3\linewidth}
        \centering
        \includegraphics[width=\linewidth]{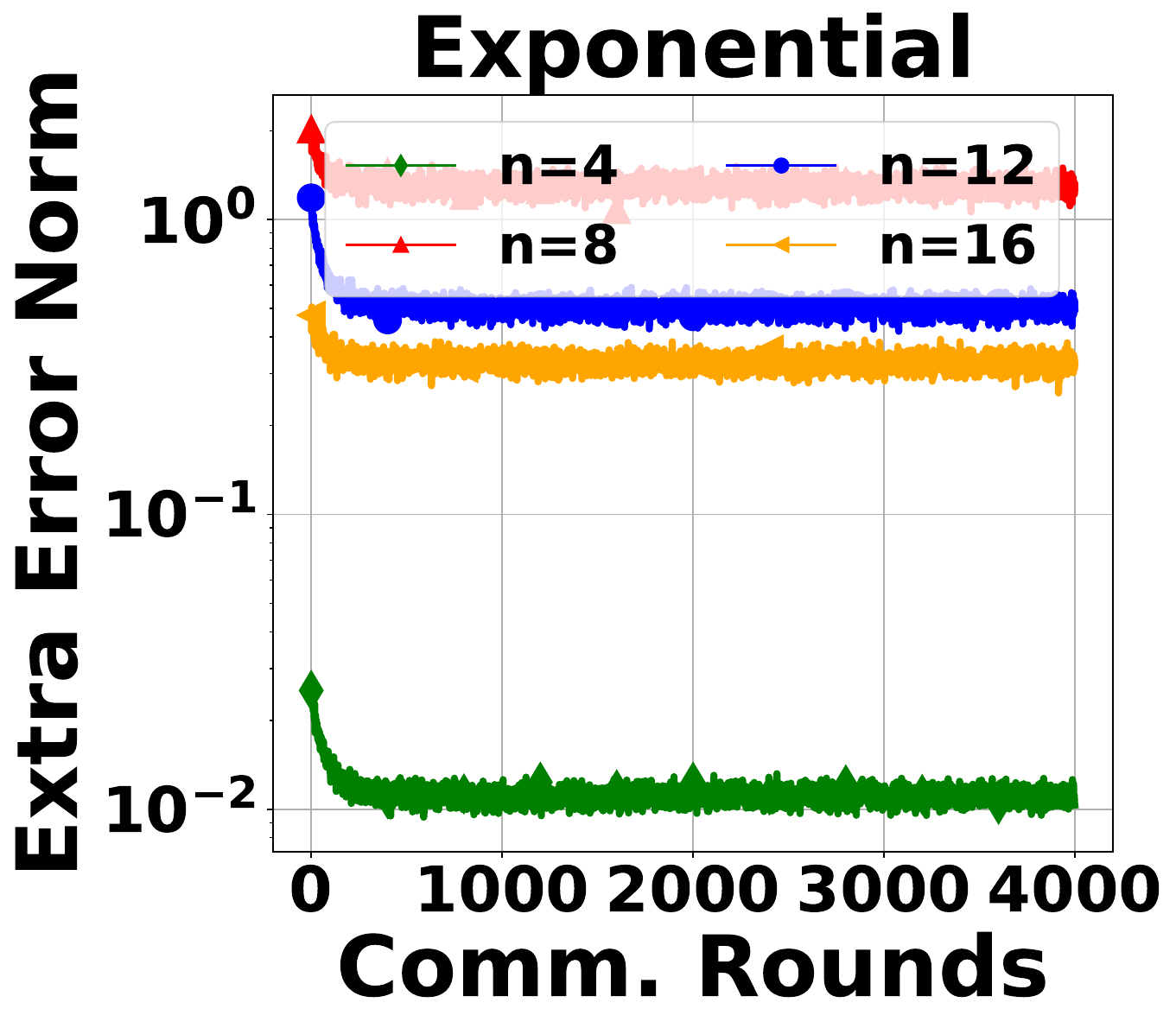}
    \end{minipage}
    \hspace{0.02\linewidth}
    \begin{minipage}{0.3\linewidth}
        \centering
        \includegraphics[width=\linewidth]{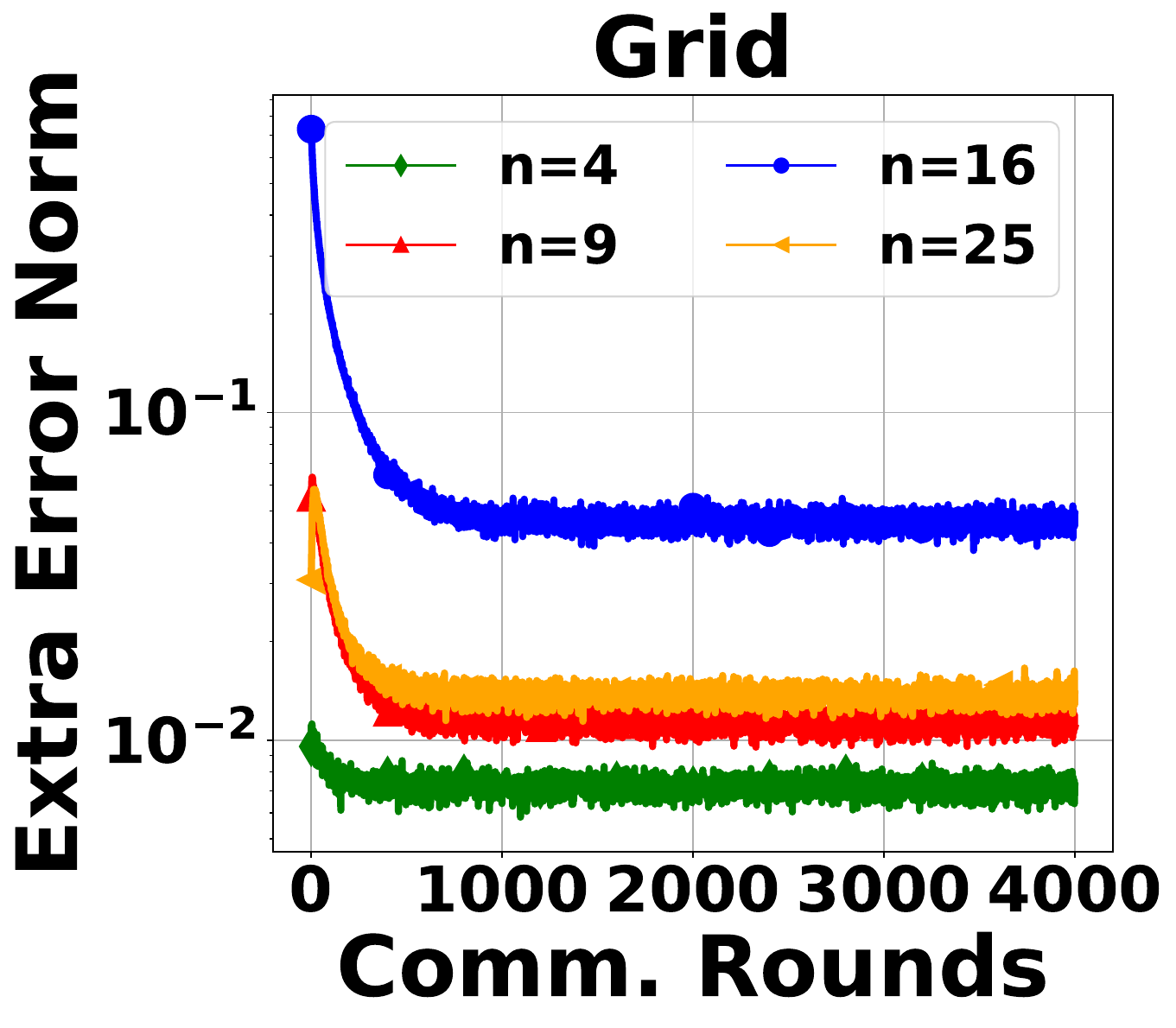}
    \end{minipage}
    \hspace{0.02\linewidth}
    \begin{minipage}{0.3\linewidth}
        \centering
        \includegraphics[width=\linewidth]{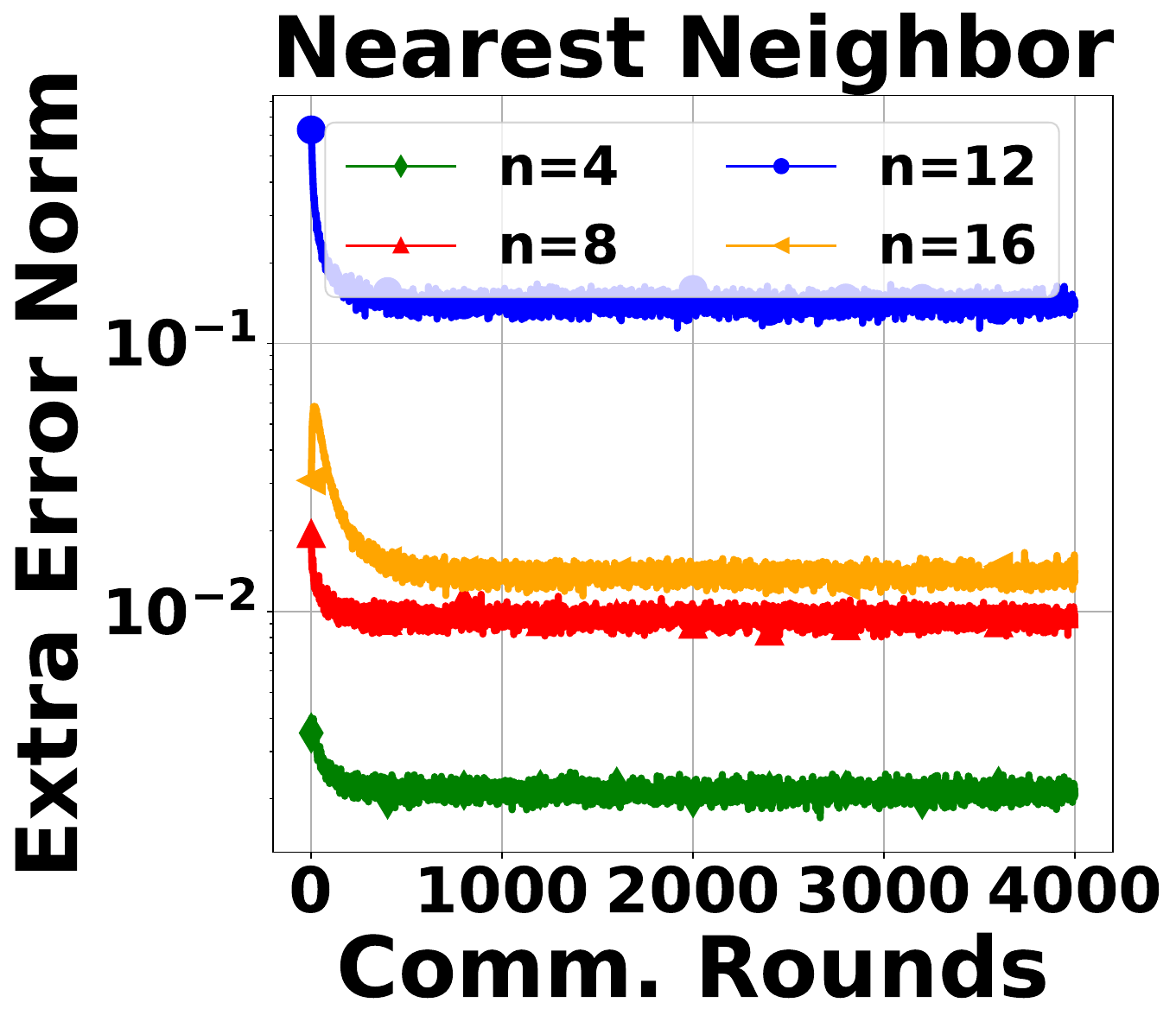}
    \end{minipage}
    \vspace{-2mm}
% top three plots illustrate the convergence behavior of \pushpull across different network topologies and sizes, while the bottom 
    \caption{
    With a fixed stepsize $\alpha$, the norm of \revision{extra error} term $\epsilon^{(t)}=-\alpha\,\pi_A^\top(I-B_\infty)\vy^{(t)}$ in \pushpull\ does not vanish but stabilizes at a non-negligible level. \revision{The experiment is nonconvex logistic regression on a synthetic dataset. The three panels correspond, from left to right, to directed exponential, grid, and nearest-neighbor graphs. Curves show different node numbers: $n\in\{4,8,12,16\}$ for the exponential and nearest-neighbor graphs, and $n\in\{4,9,16,25\}$ for the grid graph.}
    }
    \label{fig::synthetic_error_example}
    \vspace{-4mm}
\end{figure}

\section{A Novel Multiple-Step Descent Analysis Framework}\label{sec:linear speedup}
Establishing linear speedup for \pushpull\ is challenging due to the presence of a non-vanishing error in~\eqref{eq:push-pull-centralized}. We begin by revisiting the classical single-step descent framework, highlighting why it suffices for \pushonly\ and \pullonly\ methods but fails for \pushpull, thereby motivating the development of a new multi-step descent analysis.

\subsection{Limitations of single-step descent analysis}\label{sec-limitation}
Following conventional single-step descent analysis, we establish the descent inequality based on update~\eqref{eq:push-pull-centralized}:
\begin{align}\label{ieq:1-step descent}
    \norm{\nabla f(\hat{x}^{(t)})}^2 
    % + \norm{\overline{\nabla f}^{(t)}}^2 
    \hspace{-1mm}\lesssim\hspace{-1mm} \frac{f(\hat{x}^{(t)}) \hspace{-1mm}-\hspace{-1mm}\EE_t[f(\hat{x}^{(t+1)})]}{c\alpha}\hspace{-0.8mm}+\hspace{-0.8mm}\frac{c\alpha L}{n}\sigma^2 \hspace{-0.8mm}+\hspace{-0.8mm} \alpha L \norm{\pi_A^\top(I\hspace{-1mm}-\hspace{-1mm}B_\infty)\vy^{(t)}}^2 \hspace{-1mm}+\hspace{-1mm} \Delta_{\text{consent}}^{(t)}
\end{align}
where $\Delta_{\text{consent}}^{(t)}$ denotes consensus errors among local variables, typically of order $\mathcal{O}(\alpha^2)$ and hence negligible when $\alpha$ is sufficiently small. We 
% define $\overline{\nabla f}^{(t)} := n^{-1}\one^\top \nabla f(\vx^{(t)})$, and 
use $\lesssim$ to denote inequalities that hold up to absolute numerical constants. Taking full expectation and summing over $t=0,\dots,T-1$ yields
\begin{align}\label{ieq:1-step corrolary}
     \frac{1}{T}\sum_{t=0}^{T-1}\EE[\norm{\nabla f(\hat{x}^{(t)})}^2] &\lesssim \frac{f(\hat{x}^{(0)})\hspace{-0.5mm}-\hspace{-0.5mm}f(\hat{x}^{(T)})}{c\alpha T} \hspace{-0.5mm}+\hspace{-0.5mm} \frac{c\alpha L}{n}\sigma^2 \hspace{-0.5mm}\hspace{-0.5mm} \\
    &\quad +\frac{\alpha L}{T}\sum_{t=0}^{T-1}\EE[\norm{\pi_A^\top(I\hspace{-0.5mm}-\hspace{-0.5mm}B_\infty)\vy^{(t)}}^2] +\frac{1}{T}\sum_{t=1}^{T-1}\mathbb{E}[\Delta^{(t)}_{\rm consent}]. \label{eq-classical-approach}
\end{align}
Since $\tfrac{1}{T}\sum_{t=1}^{T-1}\EE[\Delta^{(t)}_{\rm consent}]$ is negligible for sufficiently small $\alpha$, the convergence rate is dominated by the two main terms in~\eqref{ieq:1-step corrolary}, provided the first term in~\eqref{eq-classical-approach} is either zero (Example~\ref{exp1}) or decays exponentially (Example~\ref{exp2}). With $\alpha = \sqrt{n/(c^2\sigma^2 T)}$, we achieve
$\tfrac{1}{T}\sum_{t=0}^{T-1}\EE[\|\nabla f(\hat{x}^{(t)})\|^2] = \mathcal{O}(\sigma/\sqrt{nT})$, 
implying $\mathcal{O}(1/(n\epsilon^2))$ iteration complexity and hence establish the linear speedup property.  

However, unlike \pushonly\ and \pullonly, the term $\EE[\|\pi_A^\top(I-B_\infty)\vy^{(t)}\|^2]$ in \pushpull\ persists (see Figure~\ref{fig::synthetic_error_example}), \revision{thus establishing linear speedup further requires the following inequality:}
\begin{equation}\label{ieq:strong ieq}
    \frac{1}{T}\sum_{t=0}^{T-1}\mathbb{E}[\|\pi_A^\top(I-B_\infty)\mathbf{y}^{(t)}\|^2]\lesssim\frac{c}{n}\sigma^2 + \text{negligible terms}.
\end{equation}
If this inequality holds, the first term in \eqref{eq-classical-approach} can be absorbed into the second term of \eqref{ieq:1-step corrolary}. Choosing $\alpha = \sqrt{n/(c^2\sigma^2 T)}$ will yield the same linear speedup rate. \revision{Unfortunately, this inequality is unlikely to be true because the term
\(
\pi_A^\top(I-B_\infty)\mathbf{y}^{(t)}
\)
does not enjoy the same averaging effect as the network average, and every term may contribute a noise term of constant order. To see this, we can left-multiply~\eqref{eq:pushpull-2} by $(I-B_\infty)$, which gives  
\begin{align}
\label{2dhbae00}
(I-B_\infty)\vy^{(t+1)} = (B-B_\infty)\vy^{(t)} + (I-B_\infty)\big(\vg^{(t+1)}-\vg^{(t)}\big),
\end{align}
where $\vg^{(t)}$ stacks the $n$ stochastic gradients $\nabla F_i(x_i^{(t)};\xi_i^{(t)})$, each with variance $\sigma^2$.  Consequently, one should not expect
\(
\EE\!\left[\|\pi_A^\top (I-B_\infty)\vy^{(t)}\|^2\right]
= O(\sigma^2/n);
\)
indeed, this quantity can be of order $n\sigma^2$ depending on $\pi_A$ and $B_\infty$. This explains why the desired bound cannot be obtained by the classical approach and motivates the need for a different analysis.}

\subsection{A novel multi-step descent analysis} 
In contrast to analyzing each iteration in isolation, we now introduce a novel $m$-step recursion that aggregates information across a block of $m$ iterations. This blockwise perspective reveals cancellation effects in the gradient-tracking recursion and enables the gradient noise to be controlled at the desired level. Specifically, for a positive integer $m$ and any index $k$ that is a multiple of $m$, the recursion~\eqref{eq:push-pull-centralized} can be written as
\begin{equation}\label{eq:m-step recursion}
\hat{x}^{(k+m)}=\hat{x}^{(k)} - c\alpha \sum_{i=0}^{m-1}\bar{g}^{(k+i)} - \alpha \pi_A^\top \sum_{i=0}^{m-1}\Delta_y^{(k+i)},
\end{equation}
where $\Delta_y^{(t)} = (I - B_\infty)\mathbf{y}^{(t)}$. Analogous to the derivation of the single-step descent inequality~\eqref{ieq:1-step descent}, we can establish the following $m$-step descent inequality:
\begin{align}\label{ieq:m-step descent}
    \norm{\nabla f(\hat{x}^{(k)})}^2 
    % + \frac{1}{m}\sum_{i=0}^{m-1}\EE_k[\norm{\overline{\nabla f}^{(k+i)}}^2] 
    \hspace{-0.8mm}\lesssim\hspace{-0.8mm} \frac{f(\hat{x}^{(k)}) \hspace{-0.8mm}-\hspace{-0.8mm} \EE_k[f(\hat{x}^{(k+m)})]}{c\alpha m} \hspace{-0.8mm}+\hspace{-0.8mm} \frac{c\alpha L}{n}\sigma^2 \hspace{-0.8mm}+\hspace{-0.8mm} \frac{\alpha L}{m} \norm{\pi_A^\top \hspace{-0.8mm} \sum_{i=0}^{m-1}\hspace{-1mm}\Delta_y^{(k+i)}\hspace{-0.2mm} }^2 \hspace{-0.8mm}+\hspace{-0.8mm} \Delta_{\text{consent}}^{(k,m)}
\end{align}
where $\Delta_{\text{consent}}^{(k,m)}$ denotes consensus errors within the $m$-step block, which are typically $\mathcal{O}(\alpha^2)$ and negligible for small $\alpha$. A formal statement of~\eqref{ieq:m-step descent} is provided in Lemma~\ref{lem:m-descent}. Let $T=mK$ denote the total number of iterations and define $[m^*] := \{0,m,2m,\ldots,(K-1)m\}$. Taking the full expectation of~\eqref{ieq:m-step descent} and summing over all $k \in [m^*]$ yields
\begin{align}\label{ieq:m-step corrolary}
    \quad \min_{0\le t\le T-1} \EE[\norm{\nabla f(\hat{x}^{(t)})}^2] &\le \frac{1}{K}\sum_{k\in[m^*]}\EE[\norm{\nabla f(\hat{x}^{(k)})}^2] \lesssim \frac{f(\hat{x}^{(0)})-f(\hat{x}^{(T)})}{c\alpha T} + \frac{c\alpha L}{n}\sigma^2\nonumber\\
    &\hspace{-0.8cm} +  \frac{\alpha L}{T}\sum_{k\in[m^*]}\EE[\norm{\pi_A^\top \sum_{i=0}^{m-1}\Delta_y^{(k+i)}}^2] + \frac{\alpha L}{T}\sum_{k\in[m^*]}\EE[\Delta_{\text{consent}}^{(k,m)}].
\end{align}
Similar to the single-step descent inequality \eqref{eq-classical-approach}, the central component in the $m$-step analysis is to control $\sum_{i=0}^{m-1}\Delta_y^{(k+i)}$. Summing the single-step recursion~\eqref{2dhbae00} from $i=0$ to $m-1$ yields
\[\sum_{i=0}^{m-1}\Delta_y^{(k+i)} =\Delta_y^{(k)}  +(B-B_\infty)\sum_{i=0}^{m-2}\Delta_y^{(k+i)} +(I-B_\infty) (\vg^{(k+m-1)}-\vg^{(k)}).\]
\revision{Because of the $\vg^{(k+1)}-\vg^{(k)}$ structure in~\eqref{2dhbae00}, this direct summation exhibits a telescoping cancellation: the intermediate stochastic gradients $\vg^{(k+i)},\, 1\le i\le m-2$, cancel out. Repeatedly applying this equality yields} 
\begin{align}\label{eq:sum of deltay}
    \sum_{i=0}^{m-1}\Delta_y^{(k+i)} 
    = \Big(\sum_{j=0}^{m-1}(B^j-B_\infty)\Big)\Delta_y^{(k)} 
    + \sum_{j=1}^{m-1}(B^{m-j-1}-B_\infty)(\vg^{(k+j)}-\vg^{(k)}),
\end{align}
\revision{where all intermediate gradients are multiplied by mixing residuals such as \(B^{m-j-1}-B_\infty\), whose norms decay exponentially~(Proposition~\ref{prop:s_A,s_B}).} 
In particular, each boundary difference contributes gradient noise with variance $n\sigma^2$, but its effect is scaled by $\lambda_B^{m-j-1}$, yielding $\sum_{j=1}^{m-1}\lambda_B^{m-j-1}n\sigma^2 \le (n\sigma^2)/(1-\lambda_B)$ which is constant in $m$. Therefore,
\begin{align}
    &\ \frac{1}{T}\sum_{k\in[m^*]}\EE[\norm{\pi_A^\top \Big(\sum_{i=0}^{m-1}\Delta_y^{(k+i)}\Big)}^2]\label{ieq:strong ieq-new}\\
    =&\ \frac{1}{T}\sum_{k\in[m^*]}\mathrm{Var}_k[\pi_A^\top \Big(\sum_{i=0}^{m-1}\Delta_y^{(k+i)}\Big)] + \frac{1}{T}\sum_{k\in[m^*]}\EE[\norm{\EE_k[\pi_A^\top \Big(\sum_{i=0}^{m-1}\Delta_y^{(k+i)}\Big)]}^2]   \nonumber \\
    \lesssim &\ \frac{n\sigma^2}{m} + \text{negligible terms}  \lesssim \frac{\sigma^2}{n} + \text{negligible terms}.  \nonumber
\end{align}
The second inequality holds when $m$ is chosen sufficiently large\footnote{In the rigorous proof, the choice of $m$ is more involved than the simplified condition stated here and may also depend on the number of iterations $T$; see Section~\ref{sec:linear speedup 2} for details.}, e.g., $m \ge \text{Const}\cdot n^2$. Recalling that $T = mK$, substituting~\eqref{ieq:strong ieq-new} into~\eqref{ieq:m-step descent} and choosing $\alpha = \sqrt{n/(c^2\sigma^2 T)}$ yield the desired linear speedup rate. 

\revision{The fundamental difference between prior single-step analyses and 
our multi-step framework lies in \emph{what must be controlled to obtain 
linear speedup}. Single-step analyses require the decentralization 
residual \(\epsilon^{(t)}\) to vanish, or to be exponentially small, at 
\emph{every iteration}, so that each update preserves the 
\(\mathcal{O}(\sigma^2/n)\) noise scaling. This pointwise criterion is 
unattainable for Push--Pull, whose residual does not vanish but 
stabilizes at a non-negligible constant 
(Figure~\ref{fig::synthetic_error_example}). Our framework instead 
imposes only a \emph{block-level} condition: that the accumulated effect 
of the residual over an \(m\)-step window stay bounded. This holds 
because the telescoping structure of the gradient-tracking recursion 
damps the block-level residual by exponentially decaying mixing factors 
rather than summing it at full strength, keeping its variance bounded 
independently of \(m\). In short, the multi-step reformulation changes the very object of control, replacing a single-step vanishing residual with a controlled block-level residual, which is what enables the first linear-speedup proof for \pushpull. }

\begin{remark}[Relation to large-batch and variance-reduction methods]
The decay of the normalized block-level stochastic contribution with the 
window length $m$ is reminiscent of large-batch averaging, but the 
resemblance is only analytical. The algorithm remains \pushpull, with 
one stochastic gradient per node per iteration and no extra samples, 
snapshots, control variates, or enlarged batches; $m$ is merely an 
analytical window. Unlike variance-reduction or large-batch methods, 
which lower the variance of the gradient estimator, our analysis 
controls the decentralized tracking residual induced by the row/column 
mixing mismatch, exploiting the telescoping temporal-difference 
structure of gradient tracking. The framework therefore does not reduce 
the stochastic gradient noise below the centralized mini-batch SGD 
level, but prevents the \pushpull residual from dominating it.
\end{remark}

\section{Achieving Linear Speedup for Push-Pull}
\label{sec:linear speedup 2}
In this section, we formalize the aforementioned arguments through supporting lemmas and rigorously establish the linear speedup property for \pushpull. 

\subsection{Supporting lemmas}
\begin{lemma}[\sc \small m-step Descent Lemma]\label{lem:m-descent}
Suppose Assumption~\ref{ass:gradient oracle},~\ref{ass:smooth} and Proposition~\ref{prop:perron} hold. For any integers $k\ge 0$, we define $\hat{x}^{(k)}=\pi_A^\top \vx^{(k)}$. When $\alpha \le \frac{1}{6cmL}$, the following inequality holds for integers $k\ge 0, m\ge 6c^{-2}s_B^2$: 
    \begin{align}\label{eq:m-descent}
        &\quad\norm{\nabla f(\hat{x}^{(k)})}^2+\frac{1}{m}\sum_{i=0}^{m-1}\EE_k[\norm{\overline{\nabla f}^{(k+i)}}^2]\\
    &\le \frac{4(f(\hat{x}^{(k)})-\EE_k[f(\hat{x}^{(k+m)})])}{\alpha c m}+\frac{8\alpha c L}{n}\sigma^2+\frac{8\alpha Lns_B^2}{cm}\sigma^2\label{eq:m-descent-2}\\
    &\quad+\frac{2 L^2}{m}\Delta_1^{(k)}+\frac{2 L^2}{m}\Delta_2^{(k)}+\frac{12 s_B^2}{c^2m^2}\norm{\Delta_y^{(k)}}_F^2,\label{eq:m-descent-3}
    \end{align}
    where $\overline{\nabla f}^{(k+i)} := n^{-1}\mathbf{1}^\top \nabla f(\vx^{(k+i)})$ denotes the average of the local deterministic gradients evaluated at $\vx^{(k+i)}$, and  
    \begin{align*}
        \Delta_1^{(k)}:=\sum_{j=0}^{m-1}\EE_k[\|\vx^{(k+j)} - \hat{\vx}^{(k)}\|_F^2], 
        \quad
        \Delta_2^{(k)}:= \sum_{j=1}^{m}\EE_k[\|\vx^{(k+j)} - \vx^{(k)}\|_F^2].
    \end{align*}
\end{lemma}
\begin{proof}
    See Appendix~\ref{app:subsec:lem 4.3}.
\end{proof}

In \eqref{eq:m-descent-2}, when $m \ge c^{-2} n^{2} s_{B}^{2}$, \revision{the dominant noise term is $\tfrac{8 \alpha c L}{n} \sigma^{2}$. This term is referred to as the linear-speedup term because, under the stepsize choice \(\alpha=\mathcal{O}(\sqrt{n/T})\), it contributes the rate \(\mathcal{O}(\sigma/\sqrt{nT})\), provided that the remaining error terms are of lower order.}  The remaining error terms in~\eqref{eq:m-descent-3} are estimated in the next lemma.

\begin{lemma}\label{lem:consensus error}
    Suppose Assumption~\ref{ass:gradient oracle},~\ref{ass:smooth} and Proposition~\ref{prop:perron} hold. With $\Delta_1^{(k)}$, $\Delta_2^{(k)}$ defined above, when $\alpha\le \min\{\frac{1}{10cm\sqrt{s_Bs_{B^m}^2}L},\frac{1}{6L}\}$ and \revision{$m\ge \tilde{C}$}, we have
    \begin{align*}
     &\quad\frac{2 L^2}{m}\sum_{k\in [m^*]}\EE[\Delta_1^{(k)}+\Delta_2^{(k)}]+\frac{12 s_B^2}{c^2m^2}\sum_{k\in [m^*]}\EE[\norm{\Delta_y^{(k)}}_F^2] \\
     &\le \frac{1}{m}\sum_{t=0}^{T-1}\EE_k[\norm{\overline{\nabla f}^{(t)}}^2]+(\frac{73c\alpha L}{n}+\frac{384s_B^3s_{B^m}^2}{m^2c^2})\sigma^2 + C_{\Delta,1}\frac{ L\Delta}{4m^2}.
\end{align*}
Here, \revision{$\tilde{C}$} and $C_{\Delta,1}$ denote matrix-related constants (see~\eqref{eq::tilde_c_m_simplified} and \eqref{eq:coefficients}).  
The quantity $s_{B^m}$ represents the matrix measure of $B^m$, as defined in~\eqref{eq:s_B definition} and ~\eqref{eq:s_B^m_def_in_appendix}.  
In general, we have $s_{B^m} \leq s_{B}$, and moreover, $s_{B^m}$ converges to $\lVert I - B_\infty \rVert_2^2$ as $m \to \infty$.
\end{lemma}
\begin{proof}
    See Appendix~\ref{app:subsec:lem 4.4}.
\end{proof}

With the above lemmas, we are ready to state the main result of this paper.
\begin{theorem}\label{thm:main}
     Suppose Assumption~\ref{ass:gradient oracle},~\ref{ass:smooth} and Proposition~\ref{prop:perron} hold. For given positive integers $m$ and $K$, we define $T:=mK$.  When $\alpha \le \frac{1}{10cm\sqrt{s_Bs_{B^m}^2}L}$ and \revision{$m\ge \tilde{C}$}, we have
    \begin{align}\label{ieq:thm:main-1}
        \quad\frac{1}{K}\sum_{k\in [m^*]}\EE[\norm{\nabla f(\hat{x}^{(k)})}^2]\le \frac{4\Delta}{c\alpha T}+\frac{100c\alpha L}{n}\sigma^2+\frac{384s_B^3s_{B^m}^2}{m^2c^2}\sigma^2 + C_{\Delta,1}\frac{L\Delta}{mT}.
    \end{align}
    %where $C_m$ and $C_{\Delta,1}$ are matrix-related constants defined in \eqref{eq::c_m_simplified} and \eqref{eq:coefficients}. 
    Furthermore, if we choose 
    % \gan{when T is large, m, alpha, so $\frac{\sqrt{n\Delta}\sigma}{10c\sqrt{LT}}$, $c=n\pi_A^T\pi_B$, pushi: $c=n$, but for}
\revision{
\begin{align}
    m&=\lceil \max\{4c^{-1} s_B^{2.5} \left(\frac{nT\sigma^2}{L\Delta}\right)^{0.25}, \tilde{C}\}  \rceil, \label{xbasbd-1}\\
    \alpha&=\min\left\{\frac{\sqrt{n\Delta}\sigma}{10c\sqrt{LT}}, \frac{1}{10cm\sqrt{s_Bs_{B^m}^2}L}, \frac{1}{6L}\right\}, \label{xbasbd-2}
\end{align}
}
it holds that 
\begin{align}\label{ieq:thm:main-2}
    \min_{t\in\{0,1,\ldots,T-1\}}\EE[\norm{\nabla f(\hat{x}^{(t)})}^2]
    &\le \frac{44\sigma\sqrt{L\Delta}}{\sqrt{nT}}
        +\frac{L\Delta (m^{-1}C_{\Delta,1}+10m\sqrt{s_Bs_{B^m}^2}+c^{-1}) }{T}\\
    &\revision{= \frac{44\sigma\sqrt{L\Delta}}{\sqrt{nT}}
        +\frac{40\,s_B^{3}s_{B^m}(n\sigma^2)^{1/4}}{c}\Big(\frac{L\Delta}{T}\Big)^{\frac34}}\nonumber
        \\\qquad \revision{ +\frac{(c^{-1}+10\sqrt{s_Bs_{B^m}^2}\Tilde{C})L\Delta}{T}}
        & \revision{+\frac{10s_As_{A^m}^2n^2+312s_B^2s_{B^m}n}{4c\,s_B^{2.5}(n\sigma^2)^{1/4}}\Big(\frac{L\Delta}{T}\Big)^{\frac54}}\nonumber\\
    &= \frac{44\sigma\sqrt{L\Delta}}{\sqrt{nT}}
        + \cO\left(\frac{L\Delta}{T}\right)^{\frac{3}{4}}
        + \cO\left(\frac{L\Delta}{T}\right)\nonumber,
\end{align}
\revision{where $\Tilde{C}= \max \Big\{ 
    n^{-1}s_A^2, 20c^{-2}ns_A^2,\;c^{-2}n^2s_B^2,  30c^{-2}s_B^3,
    c^{-2}n^2(8s_A^2s_B+2ns_B+2s_As_B\\+\sqrt{n}s_B), 20c^{-2}n^3s_A^2s_B^2
\Big\}.
$
}
\end{theorem}
\begin{proof}
    See Appendix~\ref{app:subsec:main}.
\end{proof}

\noindent\textbf{Comparison with existing analyses.} Theorem~\ref{thm:main} provides the first convergence guarantee for stochastic \pushpull on nonconvex problems. In contrast, prior analyses focus either on the strongly convex setting~\cite{saadatniaki2020decentralized,xin2019SAB} or on \pushpull variants restricted to specific network topologies~\cite{you2024bary}. The leading term $\tfrac{44\sigma\sqrt{L\Delta}}{\sqrt{nT}}$ demonstrates a linear speedup with respect to the number of nodes $n$ for any topology satisfying Proposition~\ref{prop:s_A,s_B}. Consequently, \pushpull requires $\mathcal{O}(1/(n\epsilon^2))$ iterations to achieve an $\epsilon$-accurate solution, with the complexity decreasing linearly as $n$ increases.

\vspace{1mm}
\noindent \revision{\textbf{The impact of $m$-step:} We would like to clarify that $m$ is only an auxiliary analytical window rather than an algorithmic parameter of \pushpull: it appears solely in the analysis and is not a hyperparameter of the algorithm. Although $m$ enters one branch of the  upper bound of the learning rate \eqref{xbasbd-2}, this branch is inactive for large $T$. To see it, since $\tilde{C}$ is independent of $T$, we have $m = \Theta\big((nT)^{1/4}\big)$ when $T$ is sufficiently large, and hence the $m$-dependent branch scales as $\Theta\big((nT)^{-1/4}\big)$, which is dominated by the standard branch $\Theta\big(\frac{1}{c}\sqrt{n/T}\big)$ once $T$ is sufficiently large. Consequently, the effective learning rate satisfies $c\alpha = \Theta\big(\sqrt{n/T}\big)$, matching the order of the learning rates in standard decentralized SGD~\cite{koloskova2020unified} and gradient tracking~\cite[Theorem 1]{alghunaim2022unified}. Therefore, $m$ affects only a branch of the learning-rate bound \eqref{xbasbd-2} that is inactive in the regime of sufficiently large $T$. This also explains why the multi-step parameter $m$ does not affect the linear-speedup term (i.e., the first term) in \eqref{ieq:thm:main-2}.
}

\section{Experiment}\label{sec:experiment}
In this section, we empirically examine the linear speedup properties of the \pushpull algorithm across a variety of network topologies.  
Each node in the stochastic gradient oracle operates on a local dataset, computing gradients from randomly sampled data points at every iteration.  
The experimental study consists of two phases: (i) nonconvex regression on synthetic data, and (ii) training a fully connected neural network on the heterogeneously distributed MNIST dataset.  

To further evaluate practical performance on standard benchmarks, we extend our experiments to training ResNet-18 on CIFAR-10.  
The distributed training is performed across multiple nodes ($n > 1$) under different network topologies, including directed exponential graphs, grid graphs, and random graphs.  
For reference, all results are systematically compared against the single-node baseline ($n = 1$) to quantify the speedup obtained through distributed optimization.  

Complete implementation details are provided in Appendix~\ref{appendix:experiment}, and the source code with experimental scripts is publicly available at \href{https://github.com/pkumelon/PushPull}{our GitHub repository}.

\begin{figure}[ht]
    \centering
    \begin{minipage}{0.3\linewidth}
        \centering
        \includegraphics[width=\linewidth]{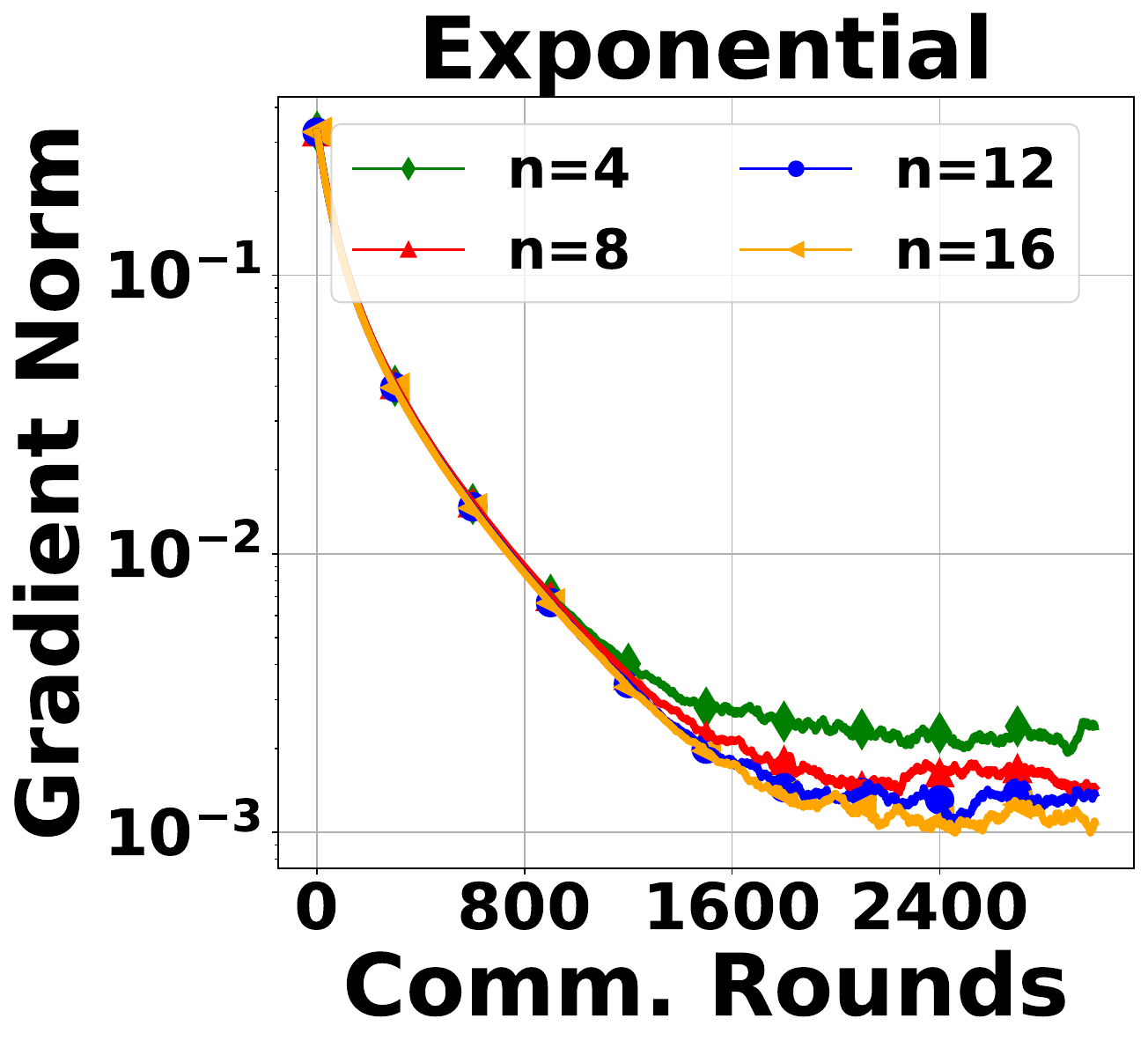}
    \end{minipage}
    \hspace{0.02\linewidth}
    \begin{minipage}{0.3\linewidth}
        \centering
        \includegraphics[width=\linewidth]{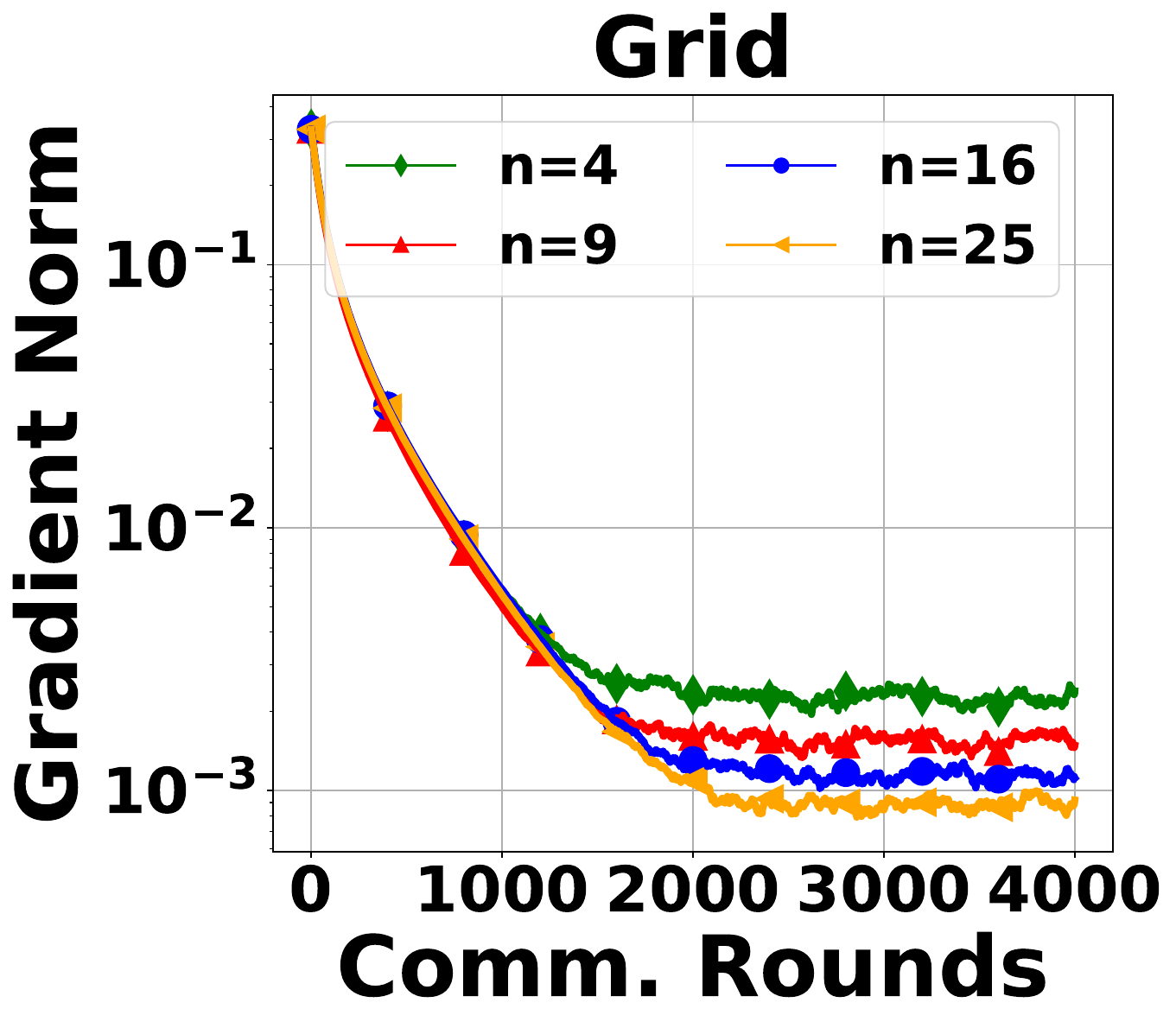}
    \end{minipage}
    \hspace{0.02\linewidth}
    \begin{minipage}{0.3\linewidth}
        \centering
        \includegraphics[width=\linewidth]{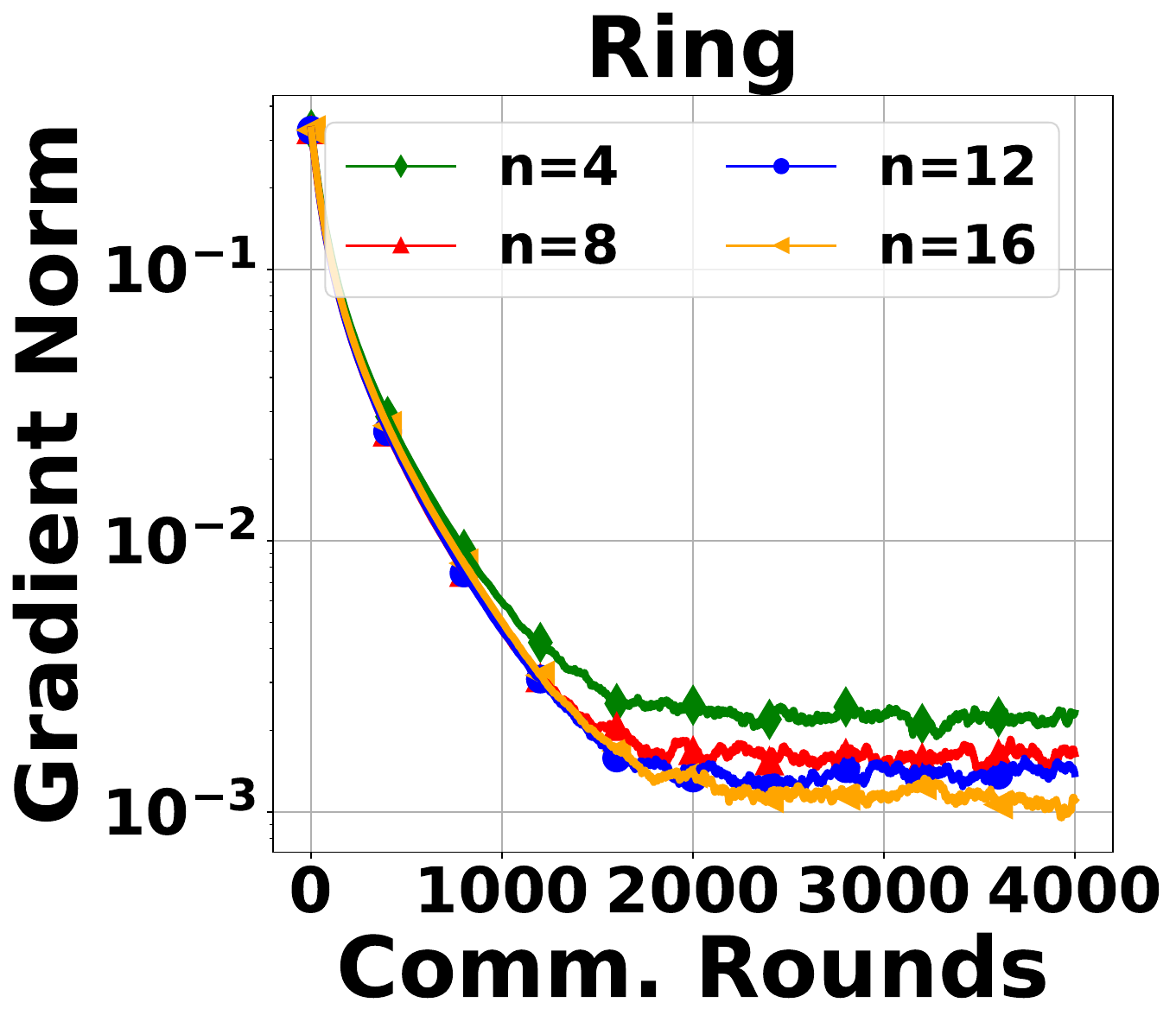}
    \end{minipage}
    
    \vspace{0.02\linewidth}
    
    \begin{minipage}{0.3\linewidth}
        \centering
        \includegraphics[width=\linewidth]{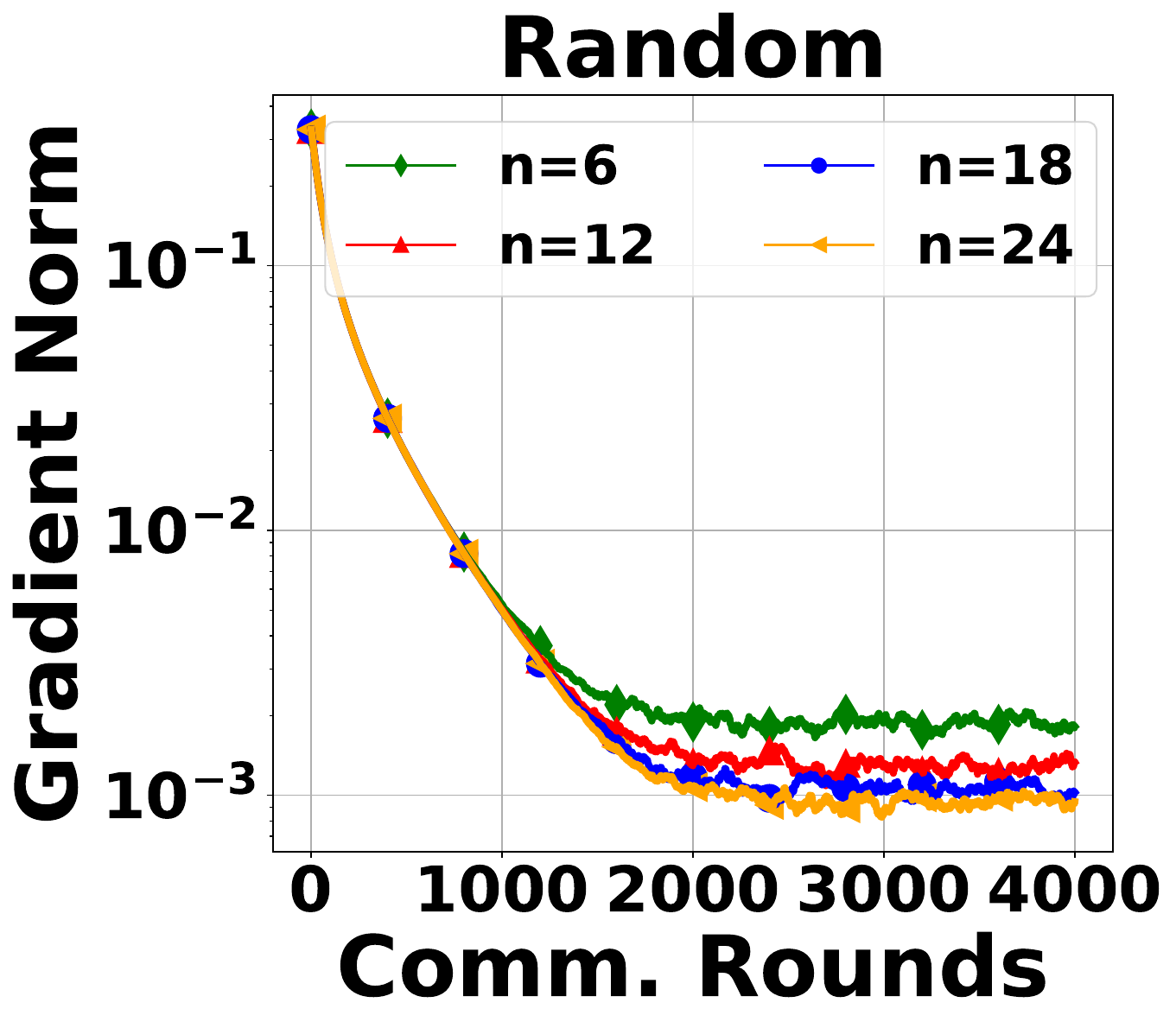}
    \end{minipage}
    \hspace{0.02\linewidth}
    \begin{minipage}{0.3\linewidth}
        \centering
        \includegraphics[width=\linewidth]{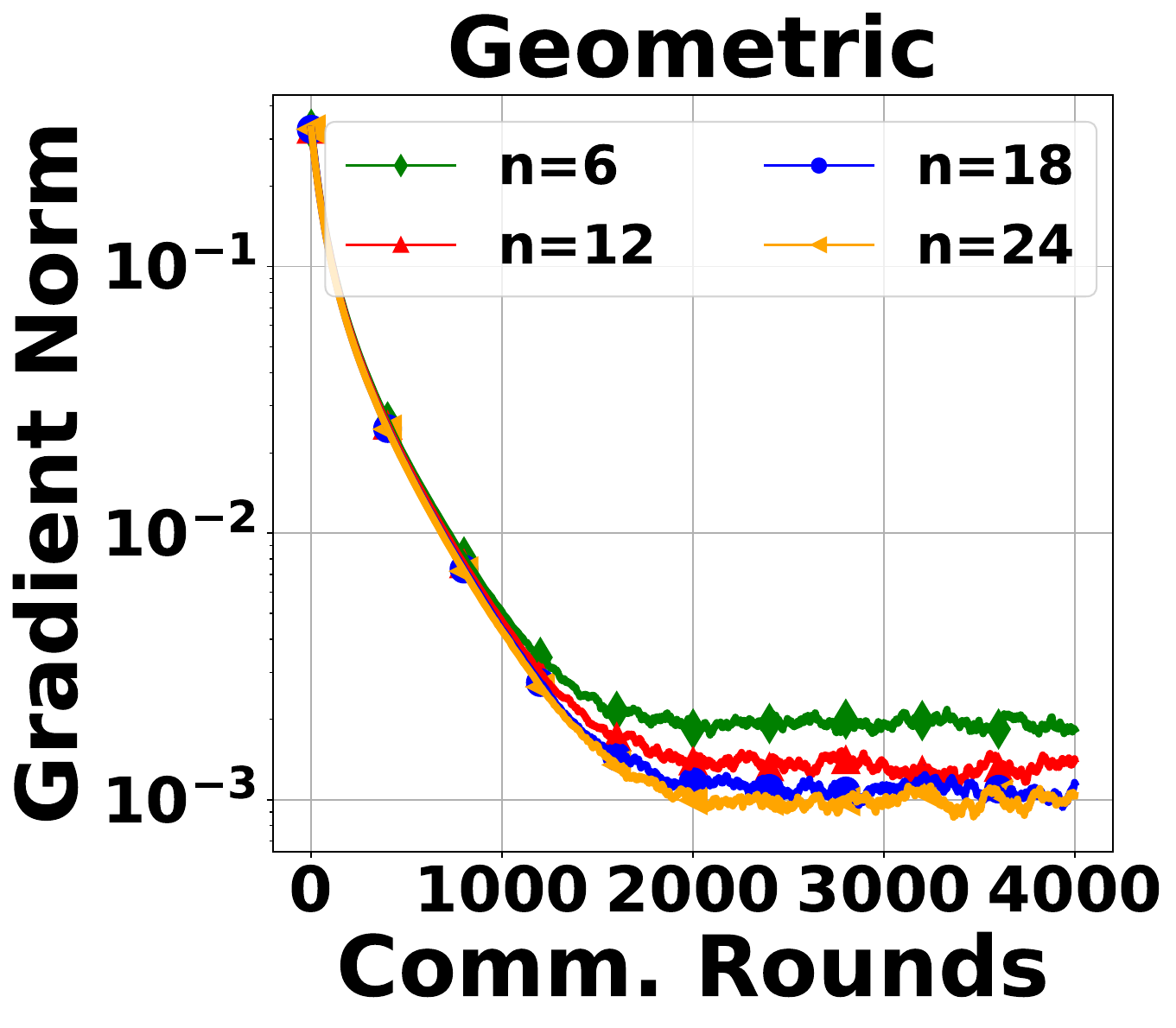}
    \end{minipage}
    \hspace{0.02\linewidth}
    \begin{minipage}{0.3\linewidth}
        \centering
        \includegraphics[width=\linewidth]{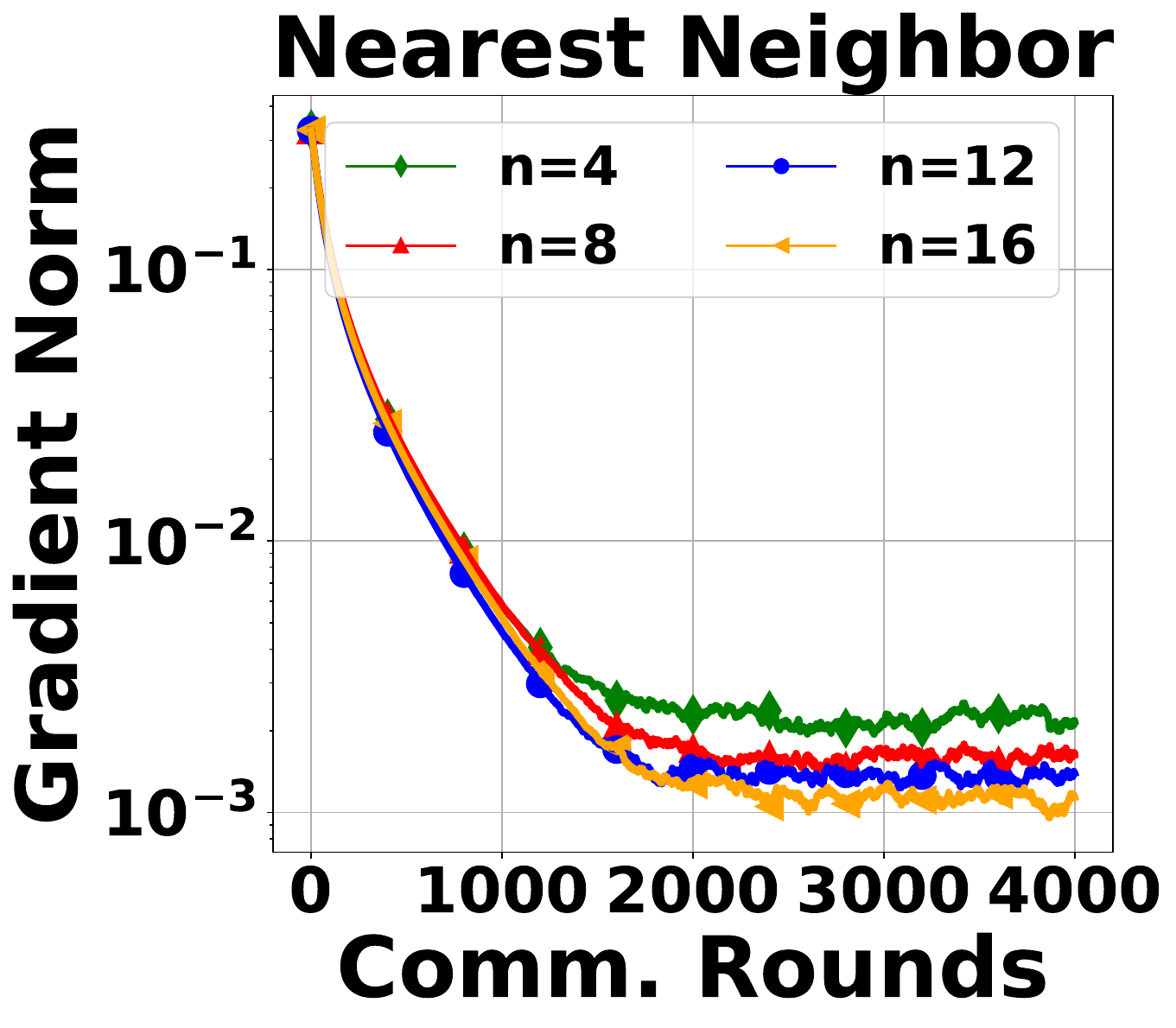}
    \end{minipage}

    \caption{\revision{Nonconvex logistic regression on the synthetic dataset. Each panel plots the squared gradient norm $\|\nabla f(\hat{x}^{(t)})\|^2$ versus iteration $t$, with the $y$-axis in log scale. The panels are ordered as directed exponential, directed grid, directed ring, undirected random, undirected geometric, and undirected nearest-neighbor graphs. Curves correspond to $n\in\{4,8,12,16\}$ nodes, except for the grid graph where $n\in\{4,9,16,25\}$; all runs use learning rate $\mathrm{lr}=0.005$. The lower noise floor for larger $n$ illustrates the linear-speedup trend.}}
    \label{fig::linear_speedup_synthetic}
\end{figure}

\subsection{Nonconvex Logistic Regression on Synthetic Dataset}\label{subsec::synthetic}

In the first experiment, the loss function at each node is logistic regression with nonconvex regularization terms.  
We evaluate six network topologies: exponential, grid, and ring graphs (directed), and random, geometric, and nearest-neighbor graphs (undirected).  
The mixing matrices $A$ and $B$ are constructed using a specialized procedure, which yields row-stochastic or column-stochastic matrices, but not doubly stochastic ones. 
% Details are provided in Appendix~\ref{app:subsec:weight matrix} and Appendix~\ref{app:subsec:synthetic}.

As shown in Figure~\ref{fig::linear_speedup_synthetic}, the results exhibit clear linear speedup.  
Across all topologies, the gradient noise level decreases proportionally to the square root of the number of nodes, given the same number of communication rounds.  
The learning rate is fixed at $lr = 0.005$ for all experiments.

\subsection{Training Neural Network for MNIST Classification}

In this part, we trained a three-layer neural network for MNIST classification.  
The same network topologies as in the synthetic dataset experiment were used.  
The MNIST data are distributed heterogeneously according to a Dirichlet distribution. 
% (see Appendix~\ref{appendix::exp_mnist}).  
As shown in Figure~\ref{fig::linear_speedup_mnist}, \pushpull achieves linear speedup on both directed and undirected topologies.

\begin{figure}[ht]
    \centering
    \begin{minipage}{0.3\linewidth}
        \centering
        \includegraphics[width=\linewidth]{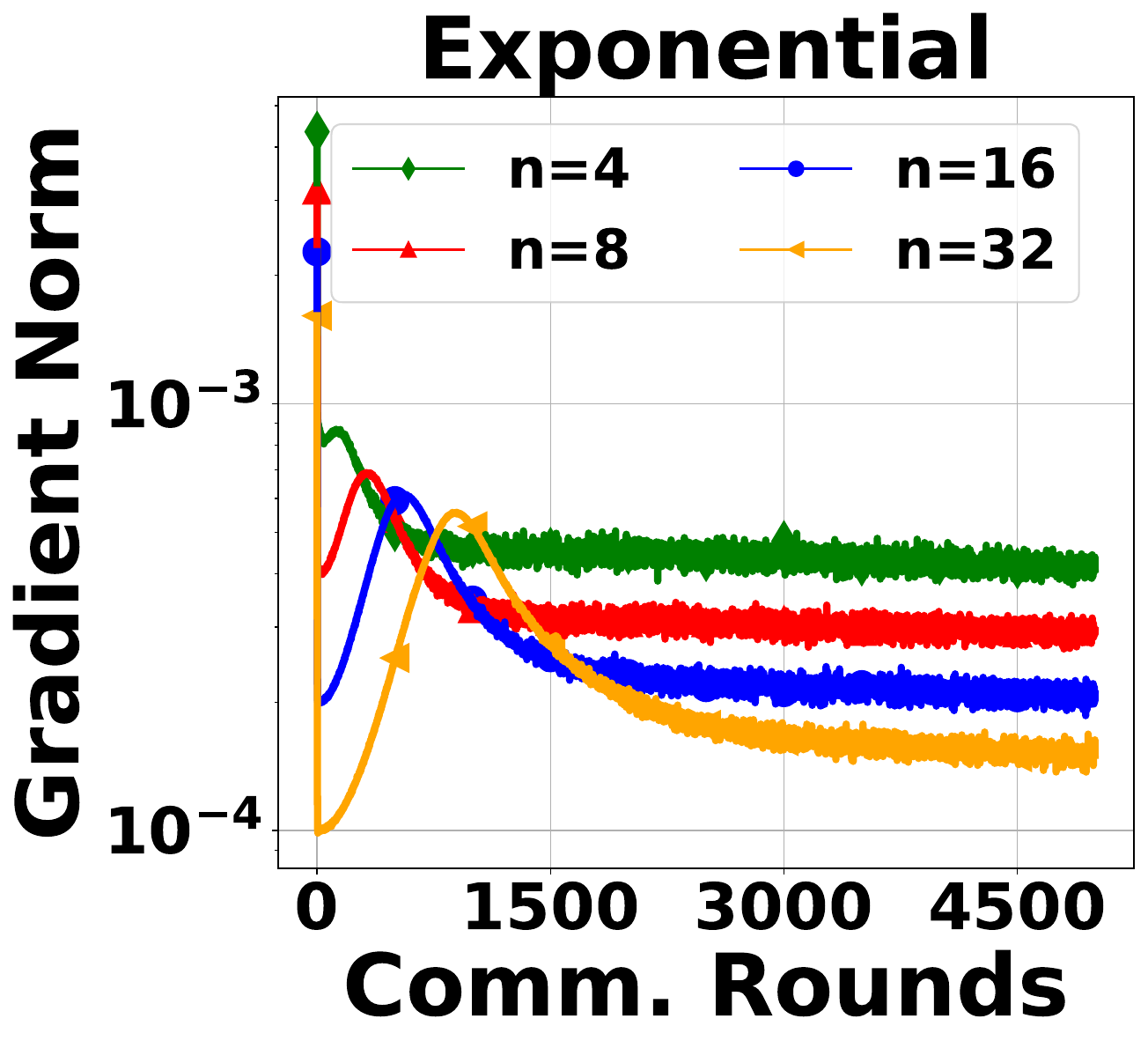}
    \end{minipage}
    \hspace{0.02\linewidth}
    \begin{minipage}{0.3\linewidth}
        \centering
        \includegraphics[width=\linewidth]{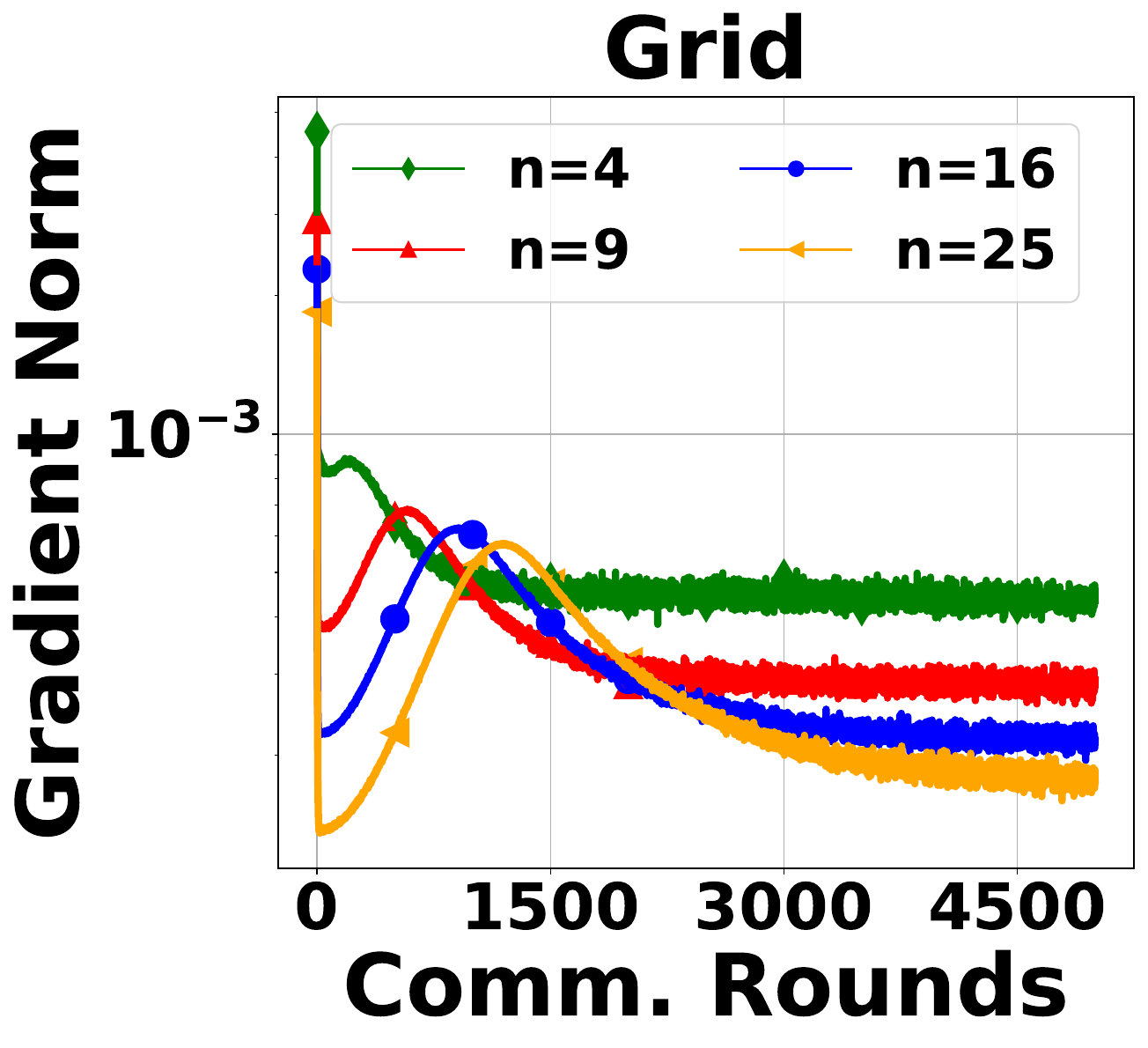}
    \end{minipage}
    \hspace{0.02\linewidth}
    \begin{minipage}{0.3\linewidth}
        \centering
        \includegraphics[width=\linewidth]{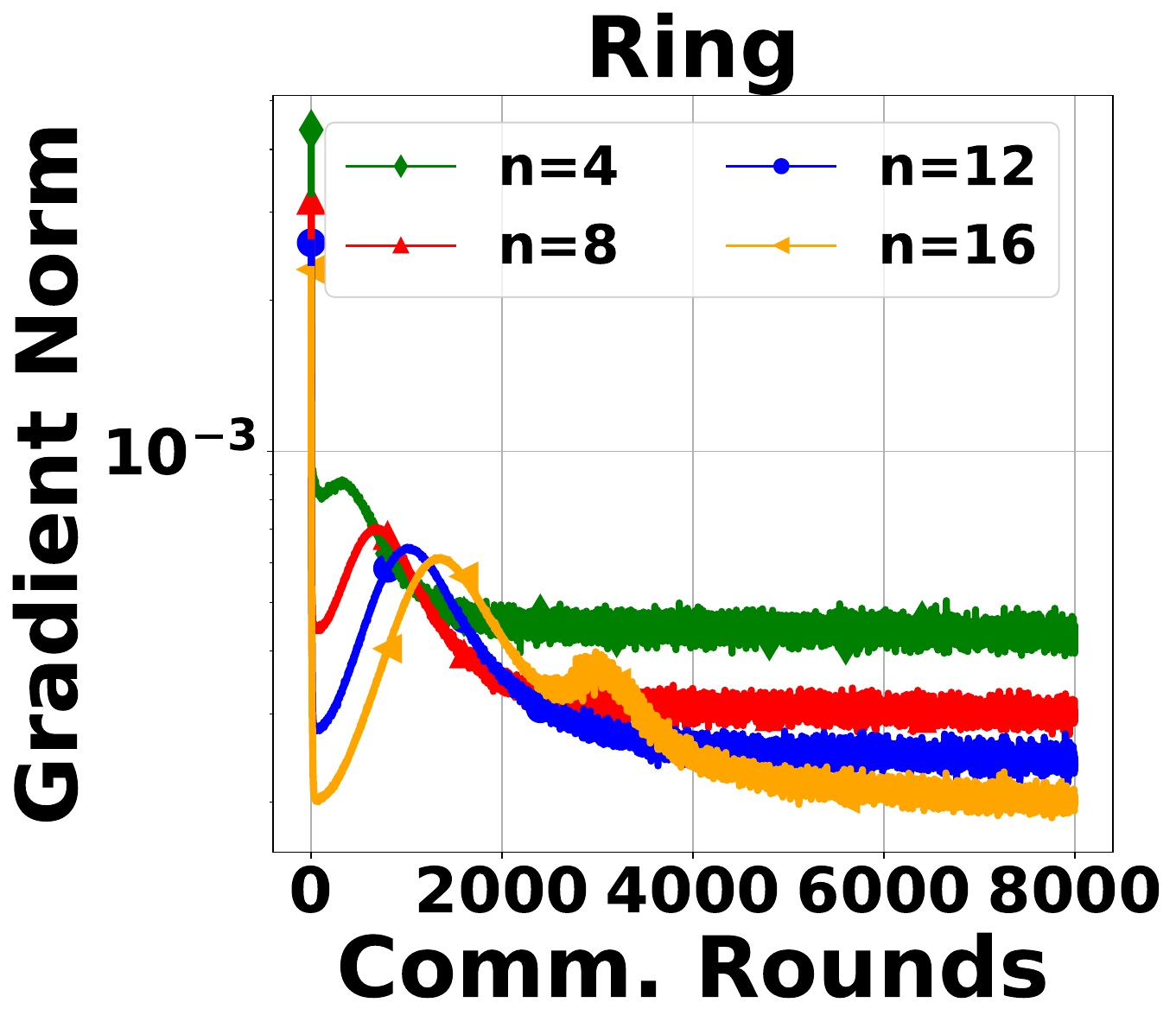}
    \end{minipage}
    
    \vspace{0.02\linewidth}
    
    \begin{minipage}{0.3\linewidth}
        \centering
        \includegraphics[width=\linewidth]{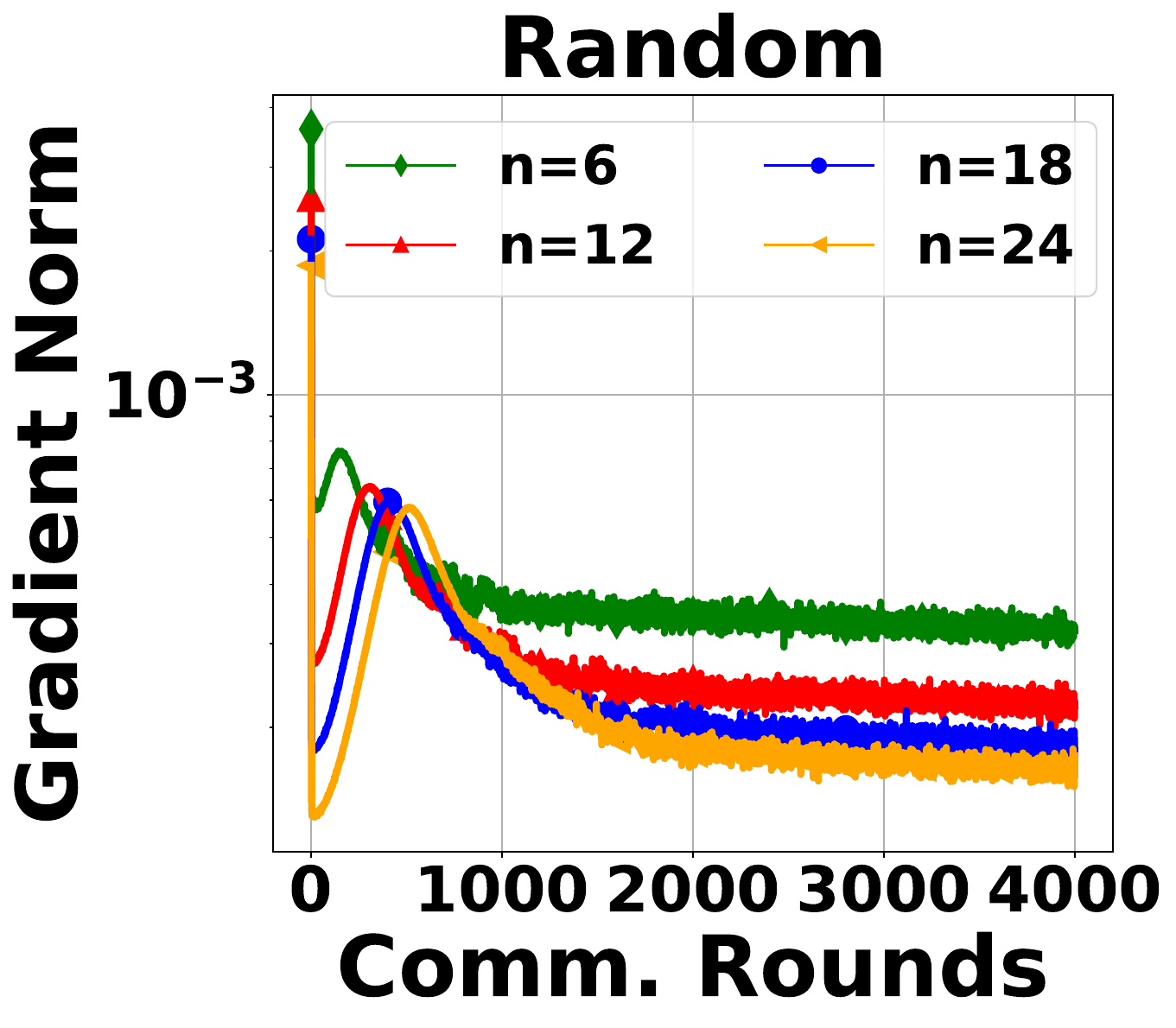}
    \end{minipage}
    \hspace{0.02\linewidth}
    \begin{minipage}{0.3\linewidth}
        \centering
        \includegraphics[width=\linewidth]{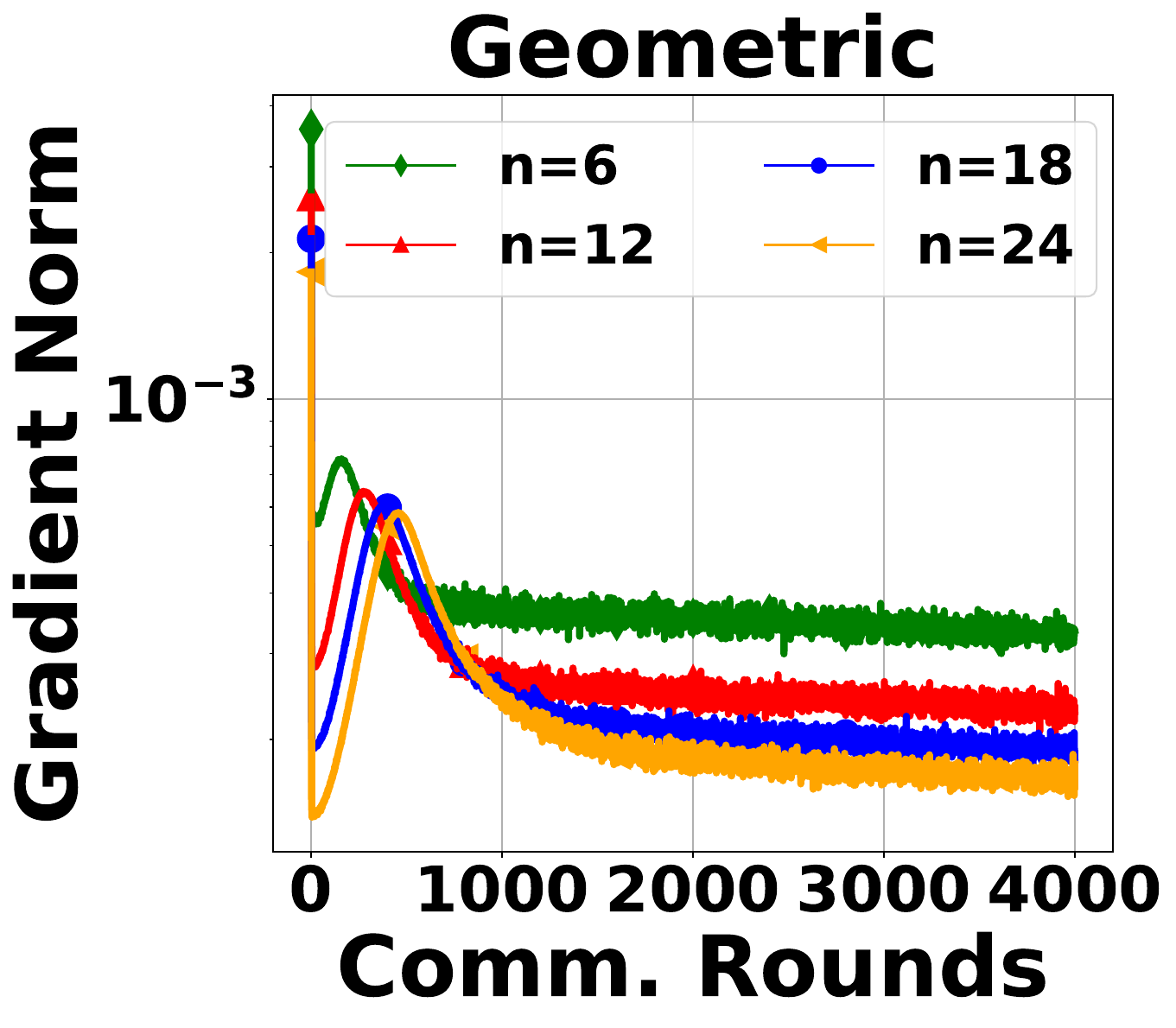}
    \end{minipage}
    \hspace{0.02\linewidth}
    \begin{minipage}{0.3\linewidth}
        \centering
        \includegraphics[width=\linewidth]{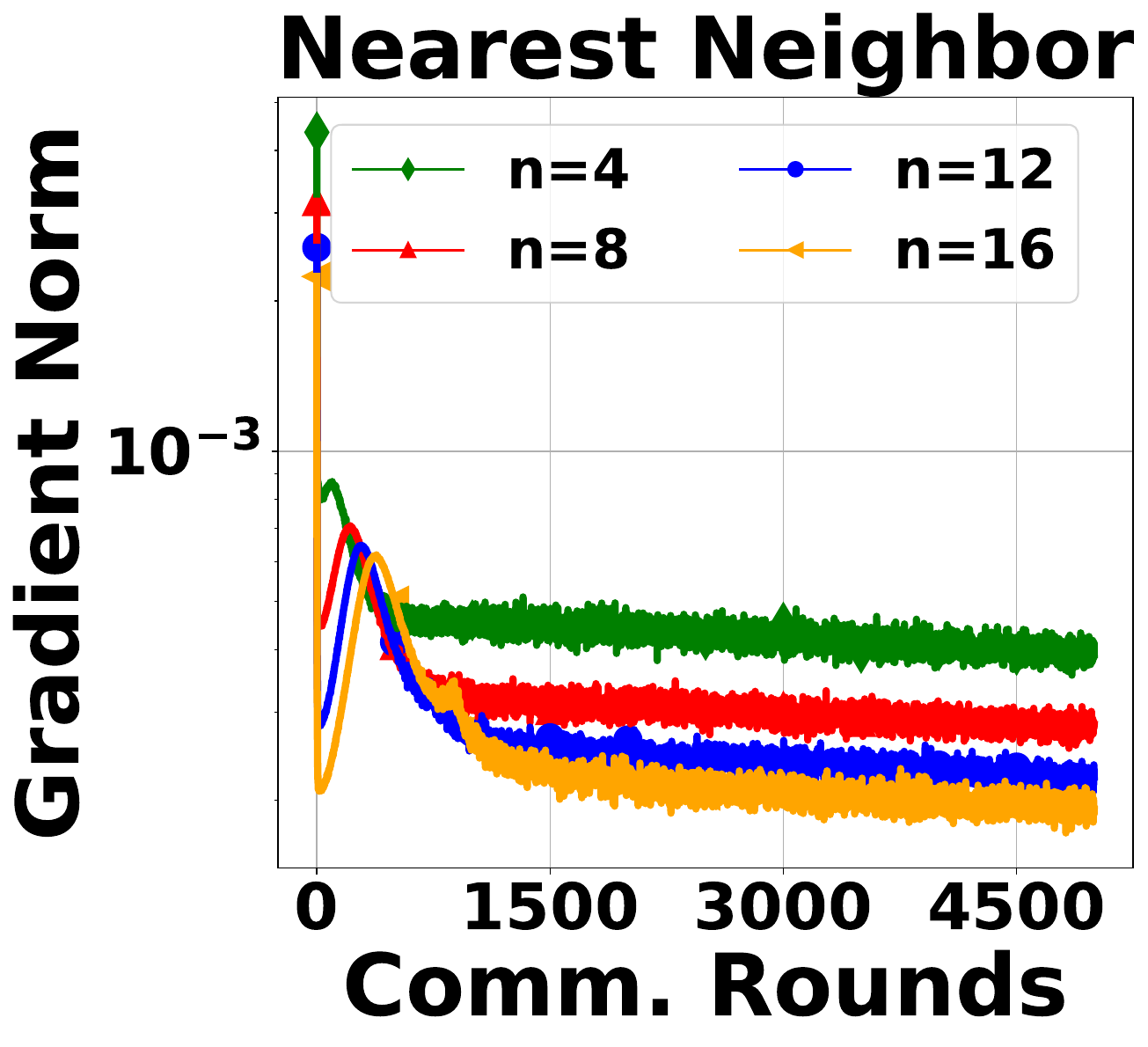}
    \end{minipage}

    \caption{\revision{MNIST classification with a three-layer fully-connected neural network and a heterogeneous Dirichlet data partition across nodes. Each panel plots training loss versus iteration count. The panels are ordered as directed exponential, directed grid, directed ring, undirected random, undirected geometric, and undirected nearest-neighbor graphs. Curves correspond to different node numbers $n$ as shown in the legend, and the learning rate is fixed within each panel.}}
    \label{fig::linear_speedup_mnist}
\end{figure}

\subsection{Training ResNet-18 for CIFAR-10 Image Classification}

In the third experiment, we trained ResNet-18 for CIFAR-10 classification across varying numbers of nodes and network topologies.  
We considered three representative topologies: directed exponential graphs, grid graphs, and random graphs.  
For reference, $n=1$ corresponds to training ResNet-18 on a single node without distributed optimization.  

Figure~\ref{fig::acc_cifar10} illustrates the performance of the \pushpull method.
We observe that its accuracy is initially lower than that of single-node training, due to the overhead introduced by network topology.
\revision{However, after a sufficient number of epochs, \pushpull\ matches and, on certain topologies, slightly surpasses the single-node baseline in test accuracy. All curves use the same learning rate and the same effective batch size, so this gap is not due to a larger batch size or a different learning-rate schedule. The result suggests that decentralized dynamics may affect generalization, although the effect can depend on topology and hyperparameter tuning. Since the single-node baseline was not tuned separately, we view this as an empirical observation rather than a universal advantage of \pushpull.}

\begin{figure}[ht]
    \centering
    \begin{minipage}{0.3\linewidth}
        \centering
        \includegraphics[width=\linewidth]{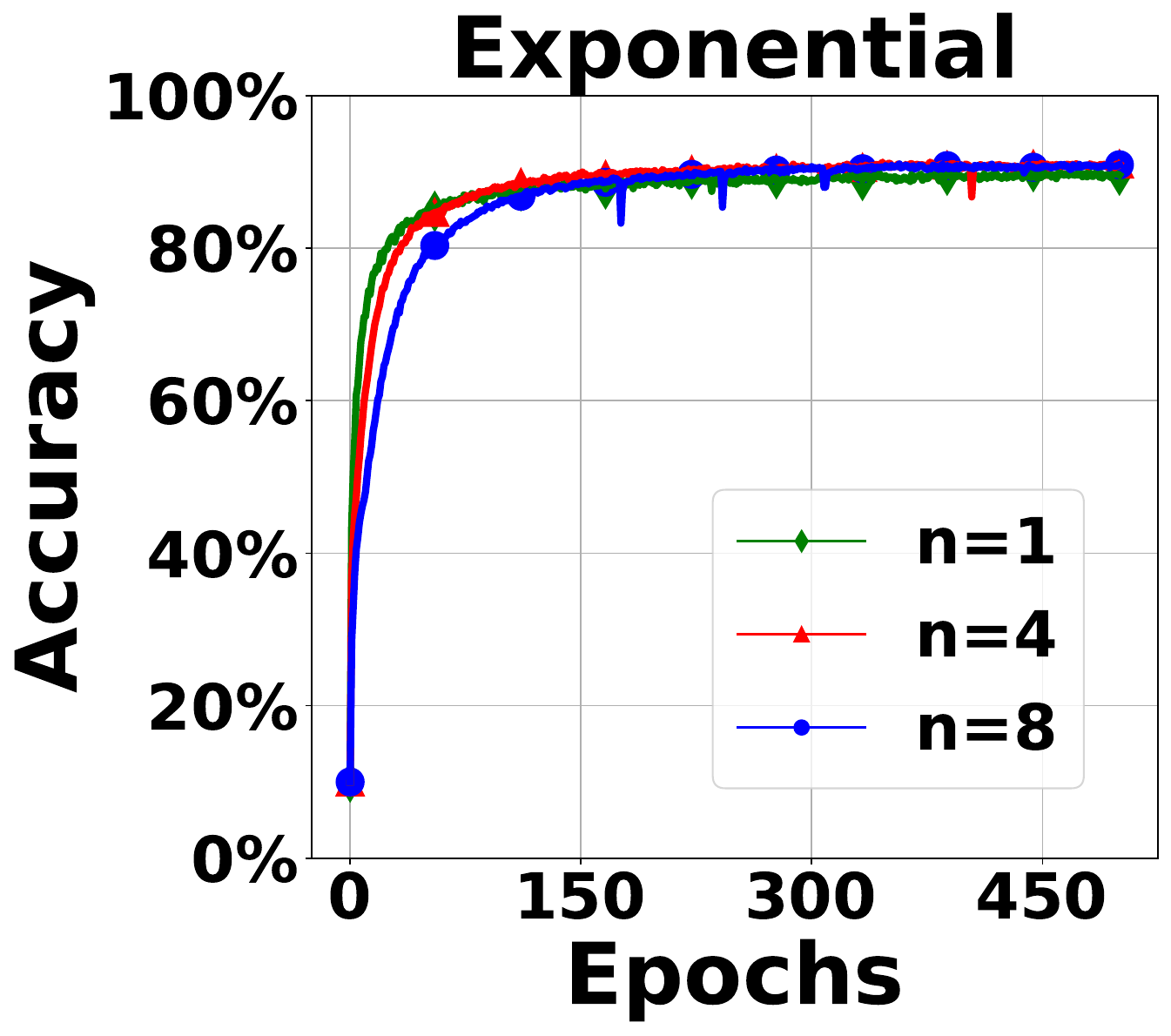}
    \end{minipage}
    \hspace{0.02\linewidth}
    \begin{minipage}{0.3\linewidth}
        \centering
        \includegraphics[width=\linewidth]{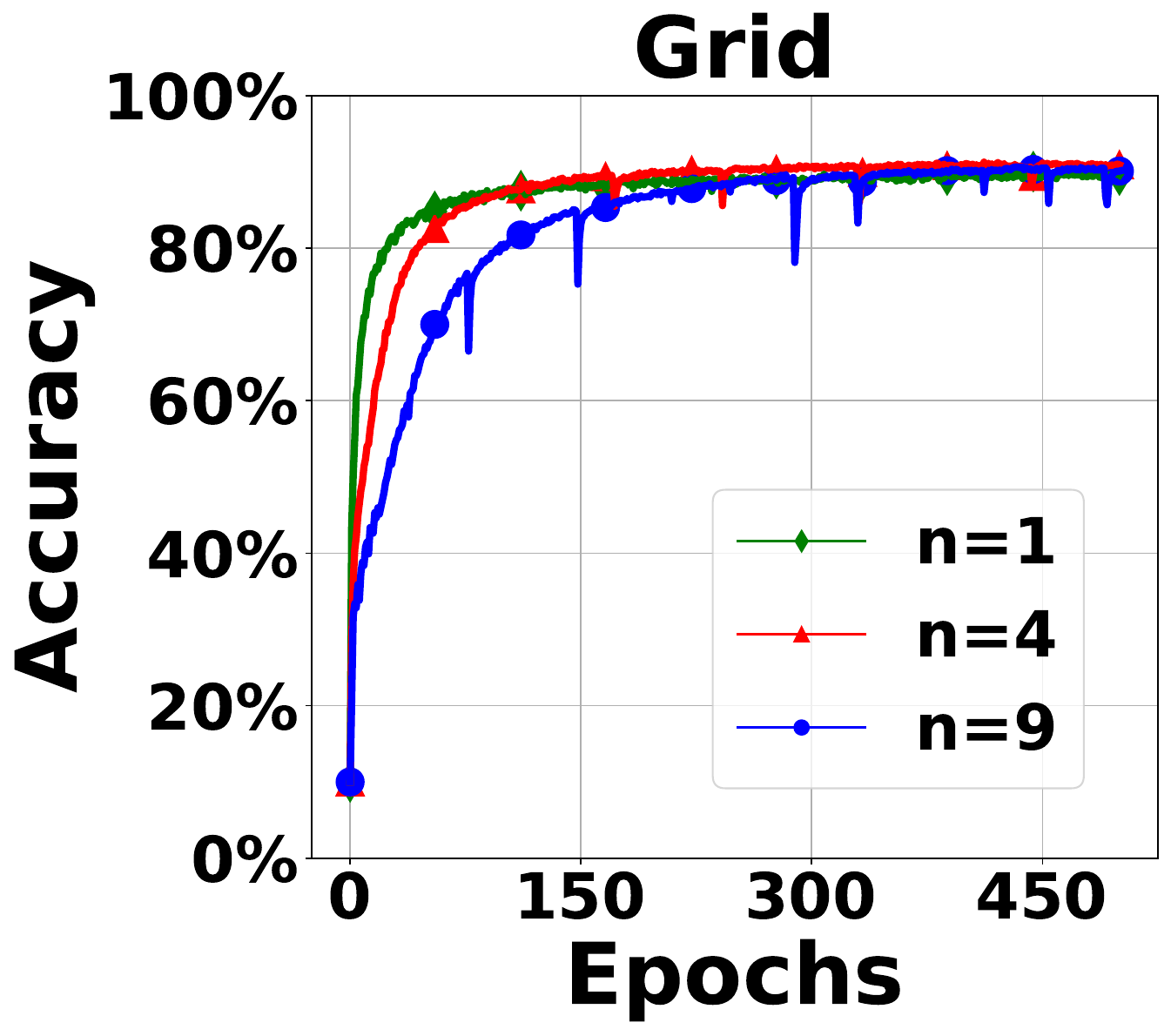}
    \end{minipage}
    \hspace{0.02\linewidth}
    \begin{minipage}{0.3\linewidth}
        \centering
        \includegraphics[width=\linewidth]{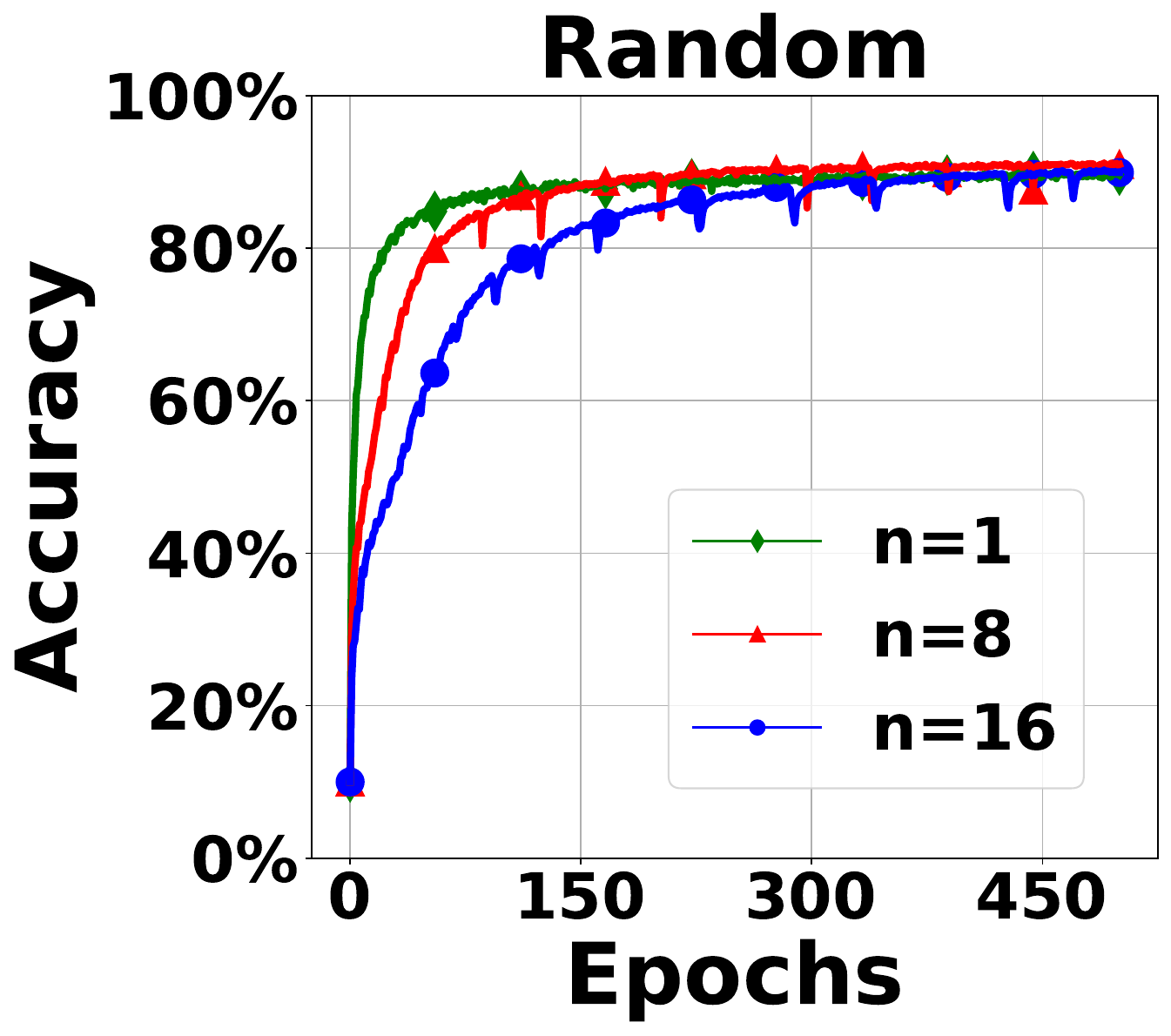}
    \end{minipage}
    
    \vspace{0.02\linewidth}
    
    \begin{minipage}{0.3\linewidth}
        \centering
        \includegraphics[width=\linewidth]{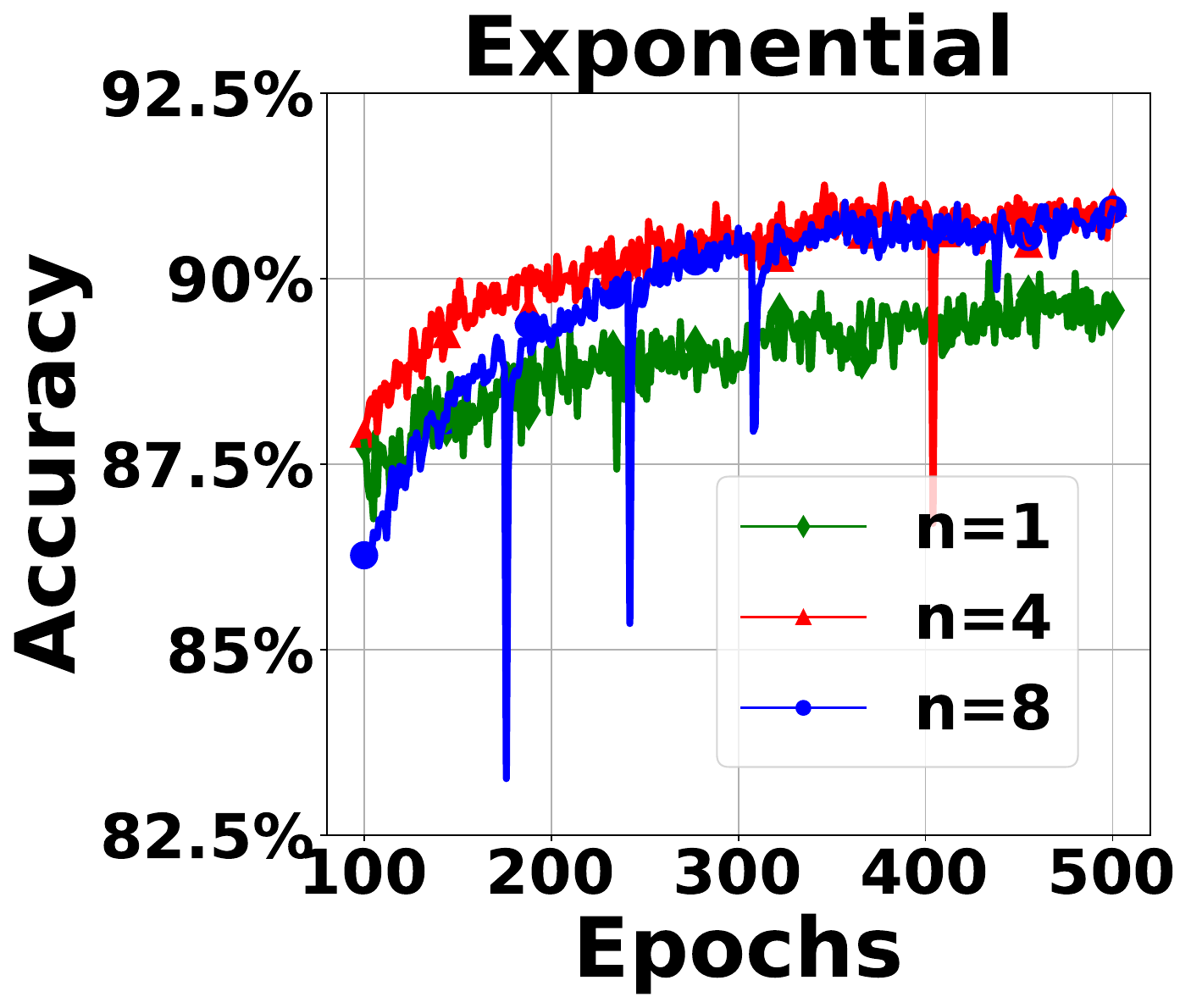}
    \end{minipage}
    \hspace{0.02\linewidth}
    \begin{minipage}{0.3\linewidth}
        \centering
        \includegraphics[width=\linewidth]{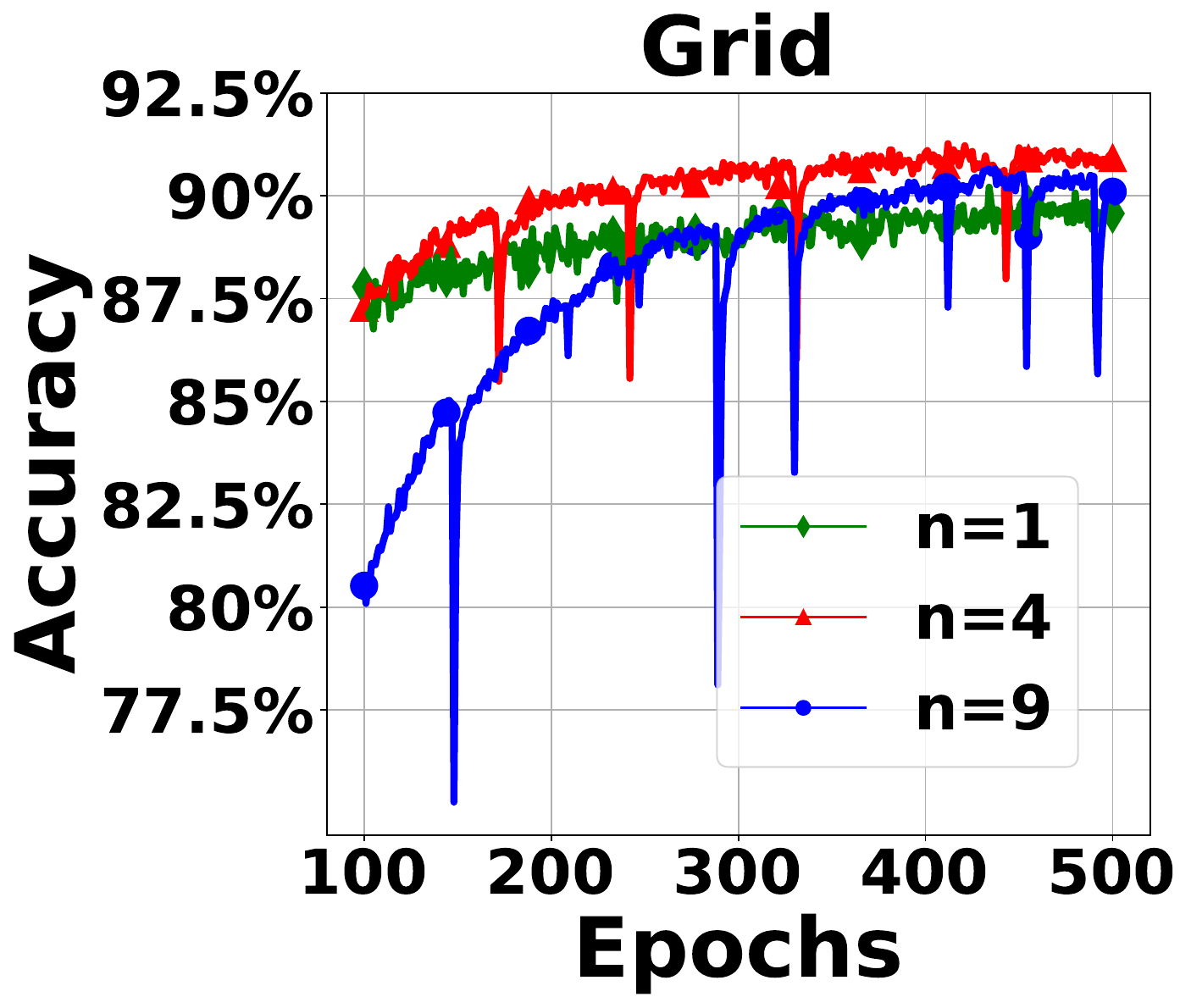}
    \end{minipage}
    \hspace{0.02\linewidth}
    \begin{minipage}{0.3\linewidth}
        \centering
        \includegraphics[width=\linewidth]{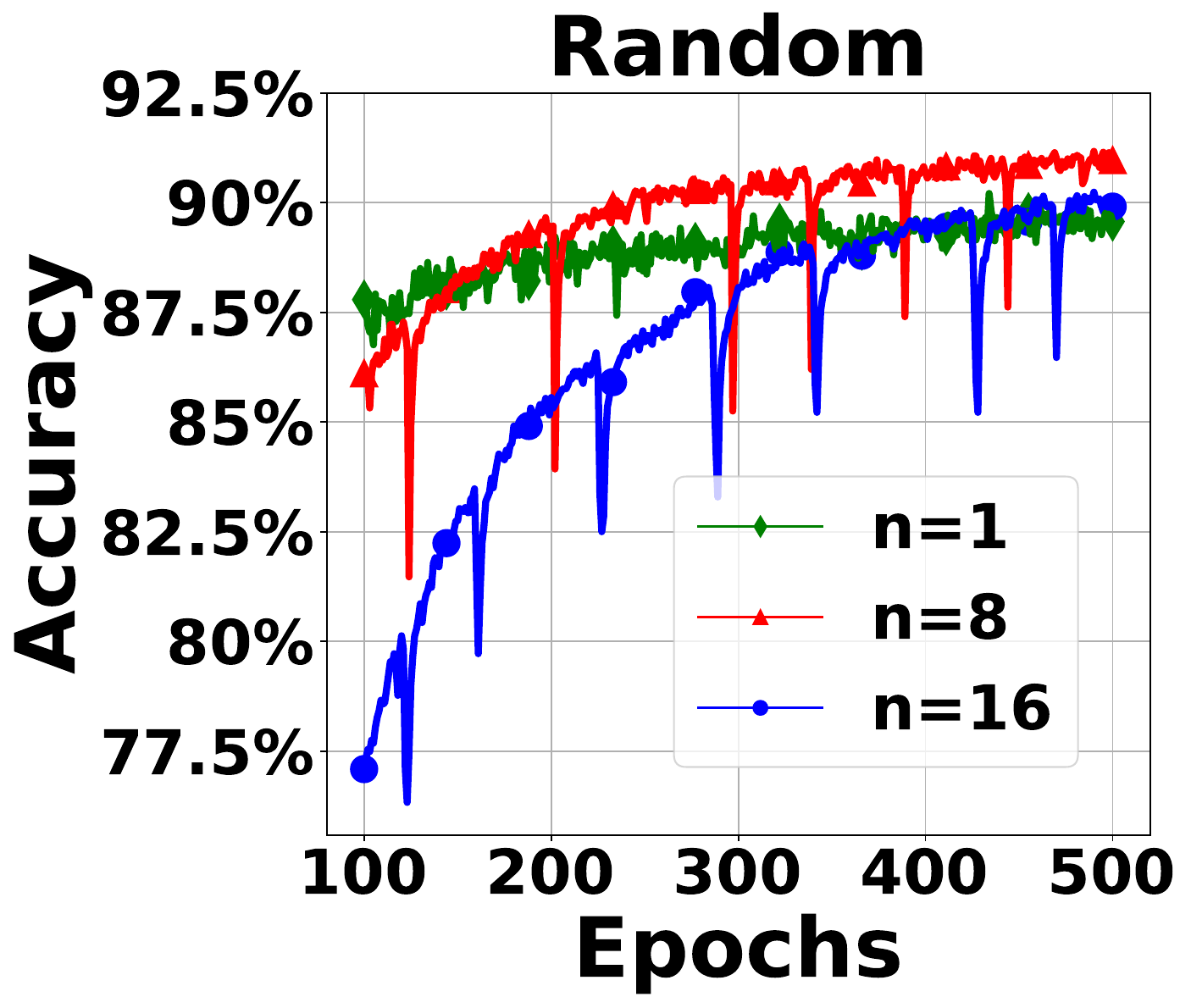}
    \end{minipage}

    \caption{\revision{CIFAR-10 test accuracy of ResNet-18 trained by \pushpull. Curves correspond to different node numbers $n$ as shown in the legend; $n=1$ is the single-node baseline without distributed optimization. The three columns correspond, from left to right, to directed exponential, grid, and random graphs. The top row gives the full test-accuracy trajectory over epochs, and the bottom row zooms in on the same runs after epoch $100$. All runs use learning rate $0.005$.}}
    \label{fig::acc_cifar10}
\end{figure}

\section{Conclusion}\label{sec:conclusion}
In this paper, we develop a novel analysis and establish, for the first time, the linear-speedup property of the stochastic \pushpull algorithm over arbitrary strongly connected topologies. This result advances the understanding of \pushpull and confirms its effectiveness. Our theoretical findings are further validated by numerical experiments.

\bibliographystyle{unsrt}
\bibliography{references}

\appendix
\setcounter{table}{0}
\setcounter{figure}{0}
\renewcommand{\thetable}{A\arabic{table}}
\renewcommand{\thefigure}{A\arabic{figure}}
\section{Linear Algebra Inequalities}\label{app:sec:rolling sum}

\begin{lemma}[\sc \small Rolling Sum Lemma]\label{lem:rolling sum}
For a row-stochastic weight matrix satisfying Proposition~\ref{prop:perron}, the following inequality holds for any $ T\ge 0$:
    \begin{align}\label{eq:rolling sum A}
    \sum_{k=0}^T\norm{\sum_{i=0}^k(A^{k-i}-A_\infty)\Delta^{(i)}}_F^2\leq s_A^2\sum_{i=0}^T\norm{\Delta^{(i)}}_F^2,
    \end{align}
    where $\Delta^{(i)}\in \RR^{n\times d}$ are arbitrary matrices, $s_A$ is defined as in~\eqref{eq:s_A definition}. Furthermore, if Assumption~\ref{ass:graph} and~\ref{ass:matrix} hold, we have
    \begin{align*}
        s_A:= \max\{\sum_{i= 0}^\infty\norm{A^i-A_\infty}_2, \sum_{i= 0}^\infty\norm{A^i-A_\infty}_2^2\}\le \frac{M_A^2(1+\frac{1}{2}\ln(\kappa_A))}{1-\beta_A}.
    \end{align*} 
     where $\beta_A$, $\kappa_A, M_A$ are defined in Proposition~\ref{prop:s_A,s_B}. The row-stochastic matrix \(A\) and the subscript $A$ can be replaced by a column-stochastic matrix \(B\) without changing the lemma.
\end{lemma}
\begin{proof}
First, we prove that
\begin{align}\label{eq: rolling sum lemma-1}
    \norm{A^i-A_\infty}_2\le \sqrt{\kappa_A}\beta_A^{i}, \forall i \ge 0.
 \end{align}
Notice that $\beta_A:=\norm{A-A_\infty}_{\pi_A}$ and \[\norm{A^i-A_\infty}_{\pi_A}= \norm{(A-A_\infty)^i}_{\pi_A}\le \norm{A-A_\infty}^i_{\pi_A}=\beta_A^i,\]
we have
\[
\norm{(A^{i}-A_\infty)v}=\norm{\Pi_A^{-1/2}(A^i-A_\infty)v}_{\pi_A}\le \sqrt{\underline{\pi_A}}\beta_A^i\norm{v}_{\pi_A}\le \sqrt{\kappa_A}\beta_A^i\norm{v},
\]
where $\underline{\pi_A}:=\min \pi_A$, $\overline{\pi_A}:=\max \pi_A$.
The last inequality comes from $\norm{v}_{\pi_A}\le \overline{\pi_A} \norm{v}$. Therefore, \eqref{eq: rolling sum lemma-1} holds.

Second, we want to prove that
\begin{align}\label{eq: rolling sum lemma-2}
    s_A\le \frac{M_A^2(1+\frac{1}{2}\ln(\kappa_A))}{1-\beta_A}.
\end{align}
Towards this end, we define $M_A:=\max\{1, \max_{k\ge 0}\norm{A^k-A_\infty}_2 \}$. According to \eqref{eq: rolling sum lemma-1}, $M_A$ is well-defined. We also define $p=\max\left\{\frac{\ln(\sqrt{\kappa_A})-\ln(M_A)}{-\ln(\beta_A)}, 0\right\}$,
then we can verify that $\norm{A^i-A_\infty}_2\le \min\{M_A, M_A\beta_A^{i-p}\},\forall i\ge 0.$ With this inequality, we can bound $\sum_{i=0}^k\norm{A^{k+1-i}-A_\infty}_2$ as follows:
\begin{align}\label{eq: rolling sum lemma-3}
    &\sum_{i=0}^k\norm{A^{k-i}-A_\infty}_2=\sum_{i=0}^{\min\{\lfloor p \rfloor, k\}}\norm{A^i-A_\infty}_2+ \sum_{i=\min\{\lfloor p \rfloor, k\}+1}^{k}\norm{A^{i}-A_\infty}_2\\
    \le &\sum_{i=0}^{\min\{\lfloor p \rfloor, k\}} M_A+\sum_{i=\min\{\lfloor p \rfloor, k\}+1}^{k} M_A\beta_A^{i-p}\nonumber\\
    \le & M_A\cdot (1+\min\{\lfloor p \rfloor, k\})+M_A\cdot \frac{1}{1-\beta_A}\beta_A^{\min\{\lfloor p \rfloor, k\}+1-p}.\nonumber
\end{align}
If $p=0$, \eqref{eq: rolling sum lemma-3} is simplified to $\sum_{i=0}^k\norm{A^{k-i}-A_\infty}_2\le M_A\cdot\frac{1}{1-\beta_A}$ and \eqref{eq: rolling sum lemma-2} is naturally satisfied.
If $p>0$, let $x=\min\{\lfloor p \rfloor, k\}+1-p\in [0,1)$, \eqref{eq: rolling sum lemma-2} is simplified to
\[
\sum_{i=0}^k\norm{A^{k-i}-A_\infty}_2\le M_A\!\left(x+p+\frac{\beta_A^x}{1-\beta_A}\right)\le M_A\!\left(p+\frac{1}{1-\beta_A}\right)=\frac{M_A(1+\frac12 \ln(\kappa_A))}{1-\beta_A}.
\]
By the definition of $p$, we have $p\le \frac{\frac{1}{2}\ln(\kappa_A)}{1-\beta_A}$. This completes the estimation of $\sum_{i=0}^{\infty}\norm{A^{k-i}-A_\infty}_2$. For the summation of the squared norm, we can use the same $p$ and obtain the upper bound without modifying the previous steps:
\[
\sum_{i=0}^k\norm{A^{k-i}-A_\infty}_2^2\le M_A^2\!\left(p+\frac{1}{1-\beta_A}\right)=\frac{M_A^2(1+\frac12 \ln(\kappa_A))}{1-\beta_A}.
\]
This finishes the proof of \eqref{eq: rolling sum lemma-2}.

Finally, to obtain \eqref{eq:rolling sum A}, we use Jensen's inequality. For any positive numbers $a_0,\ldots,a_k$ satisfying $\sum_{i=0}^{k} a_i=1$, we have
\begin{align}\label{eq: rolling sum lemma-4}
     &\norm{\sum_{i=0}^k(A^{k-i}-A_\infty)\Delta^{(i)}}_F^2=\norm{\sum_{i=0}^k a_{k-i}\cdot a_{k-i}^{-1}(A^{k-i}-A_\infty)\Delta^{(i)}}_F^2\\
    \le& \sum_{i=0}^k a_{k-i}\norm{a_{k-i}^{-1}(A^{k-i}-A_\infty)\Delta^{(i)}}_F^2\le \sum_{i=0}^k a_{k-i}^{-1}\norm{A^{k-i}-A_\infty}_2^2\norm{\Delta^{(i)}}_F^2.\nonumber
\end{align}
By choosing $a_{k-i}=(\sum_{i=0}^k\norm{A^{k-i}-A_\infty}_2)^{-1}\norm{A^{k-i}-A_\infty}_2$ in \eqref{eq: rolling sum lemma-4}, we obtain that
\begin{align}\label{eq: rolling sum lemma-5}
    \norm{\sum_{i=0}^k(A^{k-i}-A_\infty)\Delta^{(i)}}_F^2\le \sum_{i=0}^k\norm{A^{k-i}-A_\infty}_2\cdot\sum_{i=0}^k\norm{A^{k-i}-A_\infty}_2\norm{\Delta^{(i)}}_F^2.
\end{align}
By summing up \eqref{eq: rolling sum lemma-5} from $k=0$ to $T$, we obtain that
\begin{align*}
    &\sum_{k=0}^T\norm{\sum_{i=0}^k(A^{k-i}-A_\infty)\Delta^{(i)}}_F^2 \le s_A\sum_{k=0}^T\sum_{i=0}^k\norm{A^{k-i}-A_\infty}_2\norm{\Delta^{(i)}}_F^2\\
    \le& s_A \sum_{i=0}^T\!\left(\sum_{k=i}^T\norm{A^{k-i}-A_\infty}_2\right)\norm{\Delta^{(i)}}_F^2\le s_A^2\sum_{i=0}^T\norm{\Delta^{(i)}}_F^2,
\end{align*}
which finishes the proof of this lemma.
\end{proof}

\begin{lemma}\label{lem:noise separation}
Given $l\ge 0$ and matrices $M^{(0)}, M^{(1)},\ldots, M^{(l)}\in \mathbb{R}^{n\times n}$. For any $k\ge 0$, If $H=\sum_{i=0}^{l} M^{(i)}\vg^{(k+i)}$ is a linear combination of stochastic gradients, then 
\begin{align*}
    \EE_k[\norm{H}_F^2]\le n\sigma^2 \sum_{i=0}^l\norm{M^{(i)}}_2^2+\EE_k[\norm{\sum_{i=0}^{l} M^{(i)}\nabla f(\vx^{(k+i)})}_F^2].
\end{align*}
\end{lemma}
\begin{proof}
    This is a direct corollary from the unbiasedness and linearly independent gradients in Assumption~\ref{ass:gradient oracle}.
\end{proof}

\section{Proofs of Main Theorem}\label{app:sec:proofs}

\subsection{Proof of Lemma~\ref{lem:m-descent}}\label{app:subsec:lem 4.3}
For simplicity, we define $\Delta_y^{(k)}:=(I-B_\infty)\vy^{(k)}$ and $\overline{\nabla f}^{(k)}:=n^{-1}\one^\top \nabla f (\vx^{(k)})$. According to Assumption~\ref{ass:gradient oracle}, we have
\[\EE_k[\bar{g}^{(k+i)}]=\EE_k[\EE_{k+i}[\bar{g}^{(k+i)}]]=\EE_k[\overline{\nabla f}^{(k+i)}],\quad \forall i\ge 0.\]

Note that 
\[\pi_A^\top\sum_{i=0}^{m-1} \vy^{(k+i)} = c\sum_{i=0}^{m-1}\bar{g}^{(k+i)}+\pi_A^\top \sum_{i=0}^{m-1}\Delta_y^{(k+i)},\]
the descent lemma on $\hat{x}^{(k+m)}=\hat{x}^{(k)}-\alpha \pi_A^\top\sum_{i=k}^{k+m-1}\vy^{(i)}$ can be written as follows:
\begin{align*}
    f(\hat{x}^{(k+m)})- f(\hat{x}^{(k)})&\le -c\alpha \ip{\sum_{i=0}^{m-1}\bar{g}^{(k+i)},\nabla f(\hat{x}^{(k)})}-\alpha \ip{\pi_A^\top \sum_{i=0}^{m-1}\Delta_y^{(k+i)},\nabla f(\hat{x}^{(k)})}\\
    &\quad+c^2\alpha^2L \norm{\sum_{i=0}^{m-1}\bar{g}^{(k+i)}}^2+\alpha^2L\norm{\pi_A^\top\sum_{i=0}^{m-1}\Delta_y^{(k+i)}}^2.
\end{align*}
Take expectation conditioned on $\cF_k$ and we obtain that
\begin{align}\label{eq:des-1}
    &\EE_k[f(\hat{x}^{(k+m)})]-f(\hat{x}^{(k)})\le -c\alpha \sum_{i=0}^{m-1}\EE_k\left[\ip{\overline{\nabla f}^{(k+i)},\nabla f(\hat{x}^{(k)})}\right]\\
    &\quad-\alpha \EE_k\left[\ip{\pi_A^\top \sum_{i=0}^{m-1}\Delta_y^{(k+i)},\nabla f(\hat{x}^{(k)})}\right]+c^2\alpha^2L \EE_k[\norm{\sum_{i=0}^{m-1}\bar{g}^{(k+i)}}^2]\nonumber\\
    &\quad+\alpha^2L\EE_k[\norm{\pi_A^\top\sum_{i=0}^{m-1}\Delta_y^{(k+i)}}^2].\nonumber
\end{align}

The first inner product term in \eqref{eq:des-1} can be bounded by:
\begin{align*}
    -\ip{\overline{\nabla f}^{(k+i)},\nabla f(\hat{x}^{(k)})} &=  -\frac{1}{2}\norm{\overline{\nabla f}^{(k+i)}}^2-\frac{1}{2}\norm{\nabla f(\hat{x}^{(k)})}^2+\frac{1}{2}\norm{\overline{\nabla f}^{(k+i)}-\nabla f(\hat{x}^{(k)})}^2\\
    &\le -\frac{1}{2}\norm{\overline{\nabla f}^{(k+i)}}^2-\frac{1}{2}\norm{\nabla f(\hat{x}^{(k)})}^2+\frac{L^2}{2}\norm{\vx^{(k+i)}-\hat{\vx}^{(k)}}_F^2,
\end{align*}
where the last inequality comes from Cauchy-Schwarz inequality and $L$-smoothness. The first quartic term in \eqref{eq:des-1} can be bounded by: 
\[\EE_k[\norm{\sum_{i=0}^{m-1}\bar{g}^{(k+i)}}^2]\le  n^{-1}m\sigma^2+\EE_k[\norm{\sum_{i=0}^{m-1}\overline{\nabla f}^{(k+i)}}^2],\]
where we use the fact that $\bar{g}^{(k+i)}$ are conditionally unbiased estimates and
\[\EE_k\left[\ip{\bar{g}^{(k+i)}-\overline{\nabla f}^{(k+i)},\bar{g}^{(k+j)}-\overline{\nabla f}^{(k+j)}}\right]=0.\]

Plug the estimation for the inner product into \eqref{eq:des-1}, we obtain that 
\begin{align}\label{eq:des-2}
    &\frac{\alpha cm}{2}\EE_k[\norm{\nabla f(\hat{x}^{(k)})}^2]+\frac{\alpha c}{2}\sum_{i=0}^{m-1}\EE_k[\norm{\overline{\nabla f}^{(k+i)}}^2]\le \EE_k[f(\hat{x}^{(k)})-f(\hat{x}^{(k+m)})]\\
    &+\frac{c\alpha L^2}{2}\sum_{i=0}^{m-1}\EE_k[\norm{\vx^{(k+i)}-\hat{\vx}^{(k)}}_F^2]+\frac{\alpha^2c^2mL}{n}\sigma^2+c^2\alpha^2L\EE_k[\norm{\sum_{i=0}^{m-1}\overline{\nabla f}^{(k+i)}}^2]\nonumber\\
    &-\alpha \EE\left[\ip{\pi_A^\top \sum_{i=0}^{m-1}\Delta_y^{(k+i)},\nabla f(\hat{x}^{(k)})}\right]+\alpha^2L\EE_k[\norm{\pi_A^\top\sum_{i=0}^{m-1}\Delta_y^{(k+i)}}^2]\nonumber.
\end{align}

Next, we try to simplify the last quadratic term in \eqref{eq:des-2} . Notice that 
\[\sum_{i=0}^{m-1}\Delta_y^{(k+i)}=\left(\sum_{j=0}^{m-1}(B^j-B_\infty)\right)\Delta_y^{(k)}+\sum_{j=1}^{m-1}(B^{m-j-1}-B_\infty)(\vg^{(k+j)}-\vg^{(k)}),\]
the quadratic term can be bounded as:
\begin{align}\label{eq:quadratic-1}
    &\EE_k[\norm{\pi_A^\top \sum_{i=0}^{m-1}\Delta_y^{(k+i)}}^2]\le \EE_k[\norm{ \sum_{i=0}^{m-1}\Delta_y^{(k+i)}}_F^2]\\
    &\le 2 \norm{\sum_{j=0}^{m-1}(B^j-B_\infty)}_2^2\EE_k[\norm{\Delta_y^{(k)}}_F^2]+2\EE_k[\norm{\sum_{j=1}^{m-1}(B^{m-j-1}-B_\infty)(\vg^{(k+j)}-\vg^{(k)})}_F^2]\nonumber.
\end{align}
The first term here can be estimated using the fact that \(\norm{\sum_{j=0}^{m-1}(B^j-B_\infty)}_2^2\le s_B^2\). For the second term in \eqref{eq:quadratic-1}, according to Assumption~\ref{ass:gradient oracle}, we have
\begin{align}
    &\quad\EE_k[\norm{\sum_{j=1}^{m-1}(B^{m-j-1}-B_\infty)(\vg^{(k+j)}-\vg^{(k)})}_F^2]\\
    &= \sum_{j=1}^{m-1}\EE_k[\norm{(B^{m-j-1}-B_\infty)(\vg^{(k+j)}-\nabla f(\vx^{(k+j)}))}_F^2] \nonumber\\
    &\quad+\EE_k[\norm{(\sum_{j=1}^{m-1}(B^{m-j-1}-B_\infty))(\vg^{(k)}-\nabla f(\vx^{(k)}))}_F^2]\nonumber\\
    &\quad+\EE_k[\norm{\sum_{j=1}^{m-1}(B^{m-j-1}-B_\infty)(\nabla f(\vx^{(k+j)})-\nabla f(\vx^{(k)}))}_F^2]\nonumber\\
    &\le n\sigma^2\sum_{j=1}^{m-1}\norm{B^{m-j-1}-B_\infty}_2^2+n\sigma^2\norm{\sum_{j=1}^{m-1}(B^{m-j-1}-B_\infty)}_2^2\nonumber\\
    &\quad+\EE_k[\norm{\sum_{j=1}^{m-1}(B^{m-j-1}-B_\infty)(\nabla f(\vx^{(k+j)})-\nabla f(\vx^{(k)}))}_F^2]\nonumber\\
    &\le 2ns_B^2\sigma^2+\EE_k[\norm{\sum_{j=1}^{m-1}(B^{m-j-1}-B_\infty)(\nabla f(\vx^{(k+j)})-\nabla f(\vx^{(k)}))}_F^2]\nonumber
\end{align}

Similarly, using the fact that $\hat{x}^{(k)}$ is known under filtration $\cF_k$, we have
\begin{align}\label{eq:ip-1}
    &\quad-\alpha \EE_k\left[\ip{\pi_A^\top \sum_{i=0}^{m-1}\Delta_y^{(k+i)},\nabla f(\hat{x}^{(k)})}\right]\\
    &=-\alpha \ip{\EE_k[\pi_A^\top\sum_{j=1}^{m-1}(B^{m-j-1}-B_\infty)(\nabla f(\vx^{(k+j)})-\nabla f(\vx^{(k)}))],\nabla f(\hat{x}^{(k)})}\nonumber\\
    &\quad+\alpha\ip{\pi_A^\top \left(\sum_{i=0}^{m-1}(B^{m-i-1}-B_\infty)\right)\Delta_y^{(k)},\nabla f(\hat{x}^{(k)})}\nonumber\\
    &\overset{\text{Cauchy}}{\le} \frac{2\alpha}{cm}\EE_k[\norm{\pi_A^\top \sum_{j=1}^{m-1}(B^{m-j-1}-B_\infty)(\nabla f(\vx^{(k+j)})-\nabla f(\vx^{(k)}))}^2]\\
    &\quad + \frac{\alpha cm}{8}\norm{\nabla f(\hat{x}^{(k)})}^2\nonumber +\frac{2\alpha}{cm}\norm{\pi_A^\top\left(\sum_{i=0}^{m-1}(B^{m-i-1}-B_\infty)\right)\Delta_y^{(k)}}_F^2  +\frac{\alpha cm}{8}\norm{\nabla f(\hat{x}^{(k)})}^2\nonumber\\
    &\le \frac{2\alpha}{cm}\EE_k[\norm{\sum_{j=1}^{m-1}(B^{m-j-1}-B_\infty)(\nabla f(\vx^{(k+j)})-\nabla f(\vx^{(k)}))}_F^2]+ \frac{\alpha cm}{4}\norm{\nabla f(\hat{x}^{(k)})}^2\nonumber\\
    &\quad +\frac{2\alpha s_B^2}{cm}\norm{\Delta_y^{(k)}}_F^2 \nonumber,
\end{align}
where the last inequality uses $\norm{\pi_A}\le 1$ and $\norm{\sum_{j=1}^{m-1}(B^{m-j-1}-B_\infty)}_2\le s_B$.
Plug \eqref{eq:quadratic-1} and \eqref{eq:ip-1} into \eqref{eq:des-2}, we obtain that
\begin{align}\label{eq:des-3}
    &\quad\frac{\alpha cm}{4}\EE_k[\norm{\nabla f(\hat{x}^{(k)})}^2]+\frac{\alpha c}{2}\sum_{i=0}^{m-1}\EE_k[\norm{\overline{\nabla f}^{(k+i)}}^2]-c^2\alpha^2L\EE_k[\norm{\sum_{i=0}^{m-1}\overline{\nabla f}^{(k+i)}}^2]\\
    &\le \EE_k[f(\hat{x}^{(k)})-f(\hat{x}^{(k+m)})]+\frac{c\alpha L^2}{2}\sum_{i=0}^{m-1}\EE[\norm{\vx^{(k+i)}-\hat{\vx}^{(k)}}_F^2]\nonumber\\
    &+\frac{2\alpha^2c^2mL}{n}\sigma^2+2\alpha^2Lns_B^2\sigma^2+(2\alpha^2Ls_B^2+\frac{2\alpha s_B^2}{cm})\EE[\norm{\Delta_y^{(k)}}_F^2]\nonumber\\
    &+(2\alpha^2L+\frac{2\alpha}{cm })\EE[\norm{\sum_{j=1}^{m-1}(B^{m-j-1}-B_\infty)(\nabla f(\vx^{(k+j)})-\nabla f(\vx^{(k)}))}_F^2]\nonumber
\end{align}

For \(c^2\alpha^2L\EE[\norm{\sum_{i=0}^{m-1}\overline{\nabla f}^{(k+i)}}^2]\), we have the following estimate:
\begin{align}\label{eq:des-4}
    &\quad\EE[\norm{\sum_{i=0}^{m-1}\overline{\nabla f}^{(k+i)}}^2]\le m\sum_{i=0}^{m-1}\EE[\norm{\overline{\nabla f}^{(k+i)}}^2].
\end{align}

Take \eqref{eq:des-4} into \eqref{eq:des-3}, and suppose $\alpha \le \frac{1}{4cmL}$, we obtain that
\begin{align}\label{eq:des-5}
    &\quad\frac{\alpha cm}{4}\EE[\norm{\nabla f(\hat{x}^{(k)})}^2]+\frac{\alpha c}{4}\sum_{i=0}^{m-1}\EE[\norm{\overline{\nabla f}^{(k+i)}}^2]\\
    &\le \EE[f(\hat{x}^{(k)})-f(\hat{x}^{(k+m)})]+\frac{c\alpha L^2}{2}\sum_{i=0}^{m-1}\EE[\norm{\vx^{(k+i)}-\hat{\vx}^{(k)}}_F^2]\nonumber\\
    &+\frac{2\alpha^2c^2mL}{n}\sigma^2+2\alpha^2Lns_B^2\sigma^2+\frac{3\alpha s_B^2}{cm}\EE[\norm{\Delta_y^{(k)}}_F^2]\nonumber\\
    &+\frac{3\alpha}{cm}\EE[\norm{\sum_{j=1}^{m-1}(B^{m-j-1}-B_\infty)(\nabla f(\vx^{(k+j)})-\nabla f(\vx^{(k)}))}_F^2]\nonumber.
\end{align}
We can further apply Cauchy's inequality on the last term in \eqref{eq:des-5}, i.e.,
\begin{align*}
    &\quad\norm{\sum_{j=1}^{m-1}(B^{m-j-1}-B_\infty)(\nabla f(\vx^{(k+j)})-\nabla f(\vx^{(k)}))}_F^2\\
    &\le \sum_{j=1}^{m-1}\norm{B^{m-j-1}-B_\infty}_2^2\sum_{j=1}^{m-1}\norm{\nabla f(\vx^{(k+j)})-\nabla f(\vx^{(k)})}_F^2\\
    &\le s_B\sum_{j=1}^{m-1}\norm{\nabla f(\vx^{(k+j)})-\nabla f(\vx^{(k)})}_F^2\le s_BL^2 \sum_{j=1}^{m-1}\norm{\vx^{(k+j)}-\vx^{(k)}}_F^2.
\end{align*}
Finally, when $m\ge 6c^{-2}s_B$, we have $\frac{3\alpha s_B L^2}{cm}\le \frac{c\alpha L^2}{2}$, and \eqref{eq:des-5} can be simplified to 
\begin{align*}
    &\quad\frac{\alpha cm}{4}\EE[\norm{\nabla f(\hat{x}^{(k)})}^2]+\frac{\alpha c}{4}\sum_{i=0}^{m-1}\EE[\norm{\overline{\nabla f}^{(k+i)}}^2]\\
    &\le \EE[f(\hat{x}^{(k)})-f(\hat{x}^{(k+m)})]+\frac{c\alpha L^2}{2}\sum_{i=0}^{m-1}\EE[\norm{\vx^{(k+i)}-\hat{\vx}^{(k)}}_F^2]\nonumber\\
    &+\frac{2\alpha^2c^2mL}{n}\sigma^2+2\alpha^2Lns_B^2\sigma^2+\frac{3\alpha s_B^2}{cm}\EE[\norm{\Delta_y^{(k)}}_F^2]\nonumber\\
    &+\frac{c\alpha L^2}{2}\sum_{j=1}^{m-1} \EE_k[\norm{\vx^{(k+j)}-\vx^{(k)}}_F^2]\nonumber,
\end{align*}
By dividing $\frac{\alpha cm}{4}$ on both sides of this inequality, we finish the proof. 

\subsection{Proof of Lemma~\ref{lem:consensus error}}\label{app:subsec:lem 4.4}
We first provide the complete definition of the constant \revision{$\tilde{C}$} used in the statement. We define $C_m$ as the maximum of all lower bounds imposed on $m$ to ensure the validity of subsequent inequalities and simplifications: 
\begin{align}\label{eq::c_m_simplified}
C_m = \max \Big\{ 
    % Terms mainly related to matrix A
    & n^{-1}s_A^2, 20c^{-2}ns_A s_{A^m},\;c^{-2}n^2s_B^2,  30c^{-2}s_B^2s_{B^m}\nonumber\\
    & c^{-2}n^2(8s_A^2s_B+2ns_B+2s_As_B+\sqrt{n}s_B), 20c^{-2}n^3s_A s_{A^m} s_B s_{B^m} 
\Big\}.
\end{align}
\revision{
where $s_{A^m}$ and $s_{B^m}$ represent the matrix measure of $A^m$ and $B^m$ (see \eqref{eq:s_A definition}, \eqref{eq:s_B definition} and \eqref{eq:s_B^m_def_in_appendix}) with the upper bounds
\[
s_{A^m}\le s_A,\qquad s_{B^m}\le s_B,\qquad \forall m\ge 1.
\] 
Substituting these bounds into \eqref{eq::c_m_simplified} yields an $m$-independent upper bound $\tilde{C}$:
\begin{align}\label{eq::tilde_c_m_simplified}
C_m\leq \tilde{C} := \max \Big\{ 
    & n^{-1}s_A^2, 20c^{-2}ns_A^2,\;c^{-2}n^2s_B^2,  30c^{-2}s_B^3\nonumber\\
    & c^{-2}n^2(8s_A^2s_B+2ns_B+2s_As_B+\sqrt{n}s_B), 20c^{-2}n^3s_A^2s_B^2
\Big\}.
\end{align}
}
We define $\delta_B^{(0)}=I-B_\infty$, $\delta_B^{(j)}=B^{j}-B^{j-1},j\ge 1$ and the following constants:
\begin{align*}
N_7&:=\sum_{t=0}^\infty \norm{\sum_{j=0}^t A^{t-j}\delta_B^{(j)}}_2^2 \leq 8s_A^2s_B+2n\norm{\pi_A}^2s_B\leq 8s_A^2s_B+2ns_B,\\
    N_8&:=\sum_{t=0}^\infty \norm{\sum_{j=0}^t A^{t-j}\delta_B^{(j)}}_2 \leq 2s_As_B+\sqrt{n}\norm{\pi_A}s_B\leq 2s_As_B+\sqrt{n}s_B.
\end{align*}
For all arguments used in this section, we suppose Proposition~\ref{prop:s_A,s_B} and Assumption~\ref{ass:gradient oracle},~\ref{ass:smooth} hold. We first declare that, when $\alpha \le \frac{1}{6s_BL}$, $m\ge \max\{c^{-2}n^2(N_7+N_8^2),n^{-1}s_A^2, n^{-1}s_A, c^{-2}n^2s_A^2\}$, the following inequalities hold $\forall k,m \ge 1$.
    \begin{align}
    \Delta_1^{(k)} &\le 40mn\norm{\Delta_x^{(k)}}_F^2+ 40\alpha^2ms_B\norm{\Delta_y^{(k)}}_F^2+20c^2\alpha^2 m^2n^{-1}\sigma^2\label{ieq:delta 1}\\
    &\quad+20c^2\alpha^2m^2\sum_{j=0}^{m-1}\EE_k[\norm{\overline{\nabla f}^{(k+j)}}^2],\nonumber\\
    \Delta_2^{(k)}&\le 40mn\norm{\Delta_x^{(k)}}_F^2+ 40\alpha^2ms_B\norm{\Delta_y^{(k)}}_F^2+20c^2\alpha^2 m^2n^{-1}\sigma^2\label{ieq:delta 2}\\
    &\quad+20c^2\alpha^2m^2\sum_{j=0}^{m-1}\EE_k[\norm{\overline{\nabla f}^{(k+j)}}^2],\nonumber
\end{align}
Note that
\begin{align*}
    &\quad\vx^{(k+i)}- \hat{\vx}^{(k)}- (A^i-A_\infty)\Delta_x^{(k)}=-\alpha \sum_{j=0}^{i-1}A^{i-j-1}\Delta_y^{(k+j)}-\alpha \sum_{j=0}^{i-1}A^{i-j-1}B_\infty\vy^{(k+j)}\\
    &=-\alpha \sum_{j=0}^{i-1}A^{i-j-1}(B^j-B_\infty) \Delta_y^{(k)} - \alpha \sum_{l=1}^{i-1}\sum_{j=0}^{i-l-1}A^{i-j-l-1}\delta_B^{(j)}(\vg^{(k+l)}-\vg^{(k)}) \\
    &\quad\quad-\alpha \sum_{j=0}^{i-1}(A^{i-j-1}-A_\infty)n\pi_B\bar{g}^{(k+j)}-\alpha c\sum_{j=0}^{i-1}\bar{g}^{(k+j)}.
\end{align*}

% TODO (consistency fix, see proof_review report v3): restoring \mathbf{1} above makes the
% c\mathbf{1}\bar g term contribute \|\mathbf{1}\bar g\|_F^2 = n\|\bar g\|^2, so every \bar g-originated
% coefficient in the expansion below and in \eqref{ieq:delta 1},\eqref{ieq:delta 2} gains a factor n:
%   line ~176:  5c^2\alpha^2 m^2 \sum\|\overline{\nabla f}\|^2  ->  5c^2\alpha^2 n m^2 \sum\|...\|^2
%               5c^2\alpha^2 m^2 n^{-1}\sigma^2                ->  5c^2\alpha^2 m^2 \sigma^2
%   \eqref{ieq:delta 1},\eqref{ieq:delta 2}: same n-rescaling of the c^2 terms.
% This extra n must be cancelled by: (i) the missing 1/n in the L-smoothness step (eq:des-1, line ~59:
%   L^2/2 -> L^2/(2n)), propagated to eq:des-2/3/5 and the lemma statement (2L^2/m\,\Delta_1 -> 2L^2/(mn)\,\Delta_1);
% and (ii) keeping the TIGHT \Delta_2 coefficient 12 s_B L^2/(c^2 m^2) (do not loosen to 2L^2/m at line ~128).
% The downstream constants C_{\sigma,1},C_{f,1},C_{\Delta,1} must then be re-derived; leading-order rate is unchanged.

Therefore, we can apply Cauchy's inequality and sum it up from $i=1$ to $m-1$:
\begin{align*}
    &\quad\sum_{i=1}^{m-1}\EE_k[\norm{\vx^{(k+i)}-\hat{\vx}^{(k)}}_F^2]\\
    &\le 5\sum_{i=1}^{m-1}\norm{A^i-A_\infty}_2^2\norm{\Delta_x^{(k)}}_F^2+5\alpha^2 \sum_{i=1}^{m-1}\norm{\sum_{j=0}^{i-1}A^{i-j-1}(B^j-B_\infty)}_2^2\norm{\Delta_y^{(k)}}_F^2\\
    &\quad+5\alpha^2 \sum_{i=1}^{m-1}(\sum_{l=1}^{i-1}\norm{\sum_{j=0}^{i-l-1}A^{i-j-l-1}\delta_B^{(j)}}_2^2+\norm{\sum_{l=1}^{i-1}\sum_{j=0}^{i-l-1}A^{i-j-l-1}\delta_B^{(j)}}_2^2 )n\sigma^2\\
    &\quad+5\alpha^2\sum_{i=1}^{m-1}\EE_k[\norm{\sum_{j=l}^{i-1}A^{i-j-1}\delta_B^{(j-l)}(\nabla f(\vx^{(k+l)})-\nabla f(\vx^{(k)}))}_F^2]\\
    &\quad+5\alpha^2n^2s_A^2\sum_{i=1}^{m-1}\norm{\overline{\nabla f}^{(k+j)}}+5\alpha^2 n\sum_{i=1}^{m-1}\sum_{j=0}^{i-1}\norm{A^{i-j-1}-A_\infty}_2^2 \\
    &\quad+5c^2\alpha^2m^2 \sum_{j=0}^{m-1}\norm{\overline{\nabla f}^{(k+j)}}^2+5c^2\alpha^2m^2n^{-1}\sigma^2\\
    &\le 5s_A\norm{\Delta_x^{(k)}}_F^2 + 10\alpha^2(s_A^2s_B+mns_B)\norm{\Delta_y^{(k)}}_F^2+5\alpha^2 m (N_7+N_8^2+c^2mn^{-2})n\sigma^2\\
    &\quad+ 5\alpha^2 N_8^2 L^2 \Delta_2^{(k)} + 5\alpha^2(n^2s_A^2+c^2m^2)\norm{\overline{\nabla f}^{(k+j)}}^2.
\end{align*}
where in the first inequality, the $\sigma^2$ terms are separated using Lemma~\ref{lem:noise separation}, and the second inequality comes from rolling sum.
For the decomposition of $\vx^{(k+i)}-\vx^{(k)}$, the only difference is that we replace $(A^i-A_\infty)\Delta_x^{(k)}$ with $(A^i-I)\Delta_x^{(k)}$. This gives
\begin{align*}
    \quad\Delta_2^{(k)}\le &5\sum_{i=1}^{m}\norm{A^i-I}_2^2\norm{\Delta_x^{(k)}}_F^2+ 10\alpha^2(s_A^2+mn)s_B\norm{\Delta_y^{(k)}}_F^2\\
    &+5\alpha^2 m (N_7+N_8^2+c^2mn^{-2})n\sigma^2+ 5\alpha^2 N_8^2 L^2 \Delta_2^{(k)} \\
    &+ 5\alpha^2(n^2s_A^2+c^2m^2)\norm{\overline{\nabla f}^{(k+j)}}^2.
\end{align*}
Easy to prove that $\sum_{i=1}^{m}\norm{A^i-I}_2^2\le 2s_A+2mn$. When $\alpha \le \frac{1}{5N_8L}$, we have
\begin{align}\label{ieq(app):delta 2}
    \Delta_2^{(k)}&\le 20(s_A+mn)\norm{\Delta_x^{(k)}}_F^2+ 20\alpha^2(s_A^2+mn)s_B\norm{\Delta_y^{(k)}}_F^2\\
    &\quad+10\alpha^2 m (N_7+N_8^2+c^2mn^{-2})n\sigma^2+10\alpha^2(n^2s_A^2+c^2m^2)\norm{\overline{\nabla f}^{(k+j)}}^2.\nonumber
\end{align}
Plug \eqref{ieq(app):delta 2} into the estimate for $\Delta_1^{(k)}$, it is not difficult to verify that $\Delta_1^{(k)}$ could use the same upper bound as \eqref{ieq(app):delta 2}. Next, we provide some inequalities for the consensus error $\Delta_x^{(k)}$ and $\Delta_y^{(k)}$, $k\in[m^*]$. For $k\in [m^*]$, we have     
\begin{align}\label{eq:rolling y}
    \Delta_y^{(k+m)}=(B^m-B_\infty)\Delta_y^{(k)}+\sum_{i=0}^{m-1}(B^{m-1-i}-B_\infty)(\vg^{(k+i+1)}-\vg^{(k+i)})
\end{align}
We define 
\(
\Delta_3^{(k)}=\sum_{i=0}^{m-1}(B^{m-1-i}-B_\infty)(\vg^{(k+i+1)}-\vg^{(k+i)})\), which can be written as 
\begin{align*}
    \Delta_3^{(k)}&=(I-B_\infty)\vg^{(k+m)}+(B-I)\vg^{(k+m-1}+(B^2-B)\vg^{(k+m-2)}\\
    &\quad+\ldots+(B^{m-1}-B^{m-2})\vg^{(k+1)}-(B^{m-1}-B_\infty)\vg^{(k)},
\end{align*}
according to Lemma~\ref{lem:noise separation}, we have \(\EE_k[\norm{\Delta_3^{(k)}}_F^2]\le n\sigma^2 q_B+ h_{m,k}\),
where
\begin{align*}
    h_{m,k}&:=\sum_{j=1}^m\delta_B^{(m-j)}\nabla f(\vx^{(k+j)})-(B^{m-1}-B_\infty)\nabla f(\vx^{(k)}),\\
    q_B&=\sum_{j=1}^m\norm{\delta_B^{(m-j)}}_2^2+\norm{B^{m-1}-B_\infty}_2^2\le 2s_B.
\end{align*}
For $h_{m,k}$, we can rearrange it as
\begin{align*}
    h_{m,k}&=\sum_{j=1}^m\delta_B^{(m-j)}(\nabla f(\vx^{(k+j)})-\nabla f(\vx^{(k)})).
\end{align*}
By applying the Jensen's inequality, we have
\begin{align*}
    \EE_k[\norm{h_{m,k}}_F^2]&\le \sum_{t=0}^{m}\norm{\delta_B^{(t)}}_2^2 \sum_{j=1}^m\EE_k[\norm{\nabla f(\vx^{(k+j)})-\nabla f(\vx^{(k)})}_F^2]\\
    &\le 2s_B L^2\sum_{j=1}^m \EE_k[\norm{\vx^{(k+j)}-\vx^{(k)}}_F^2]   = 2s_BL^2\Delta_2^{(k)}.
\end{align*}
Combining the estimates above, we reach that 
\[\EE_k[\norm{\Delta_3^{(k)}}_F^2]\le 2ns_B\sigma^2 +2s_BL^2\Delta_2^{(k)}. \]
To proceed on, we iterate \eqref{eq:rolling y} and obtain that
\begin{align*}
    \Delta_y^{(k+m)} = (B^{m+k}-B_\infty)\vg^{(0)}+ \sum_{j=0}^{k/m}(B^{k-jm}-B_\infty)\Delta_3^{(jm)}.
\end{align*}
%The first term can be bounded in the same way as in \eqref{}. For the second term, we can apply a similar method to our Lemma~\ref{lem:sum of y}.
Using the rolling sum lemma, we have
\begin{align*}
    &\sum_{k\in[m^*]}\EE[\norm{\Delta_y^{(k+m)}}_F^2] \le 2\sum_{k\in [m^*]} \EE[\norm{(B^k-B_\infty)\vg^{(0)}}_F^2]+ 2s_{B^m}^2\sum_{k\in[m^*]}\EE[\EE_k[\norm{\Delta_3^{(k)}}_F^2]]\\
    &\quad\le s_{B^m}(n\sigma^2+2nL\Delta)+4ns_Bs_{B^m}^2K\sigma^2+4s_Bs_{B^m}^2L^2\sum_{k\in[m^*]}\EE[\Delta_2^{(k)}]. 
\end{align*}
Note that 
\[\EE[\norm{\Delta_y^{(0)}}_F^2] = \EE[\norm{(I-B_\infty)\vg^{(0)}}_F^2]\le 2(n\sigma^2+2nL\Delta),\]
we combine coefficients and reach that 
\begin{align*}
    \sum_{k\in[m^*]}\norm{\Delta_y^{(k)}}_F^2 &\le 7s_Bs_{B^m}^2Kn\sigma^2+6s_{B^m}nL\Delta +4s_Bs_{B^m}^2L^2\sum_{k\in[m^*]}\Delta_2^{(k)}.
\end{align*}
Replace the $\Delta_2$ term with the estimate in \eqref{ieq:delta 2}, we reach that
\begin{align*}
    &(1-160\alpha^2mL^2 s_B^2s_{B^m}^2)\sum_{k\in[m^*]}\EE[\norm{\Delta_y^{(k)}}_F^2] \\
    \le& (7s_Bs_{B^m}^2+80c^2m^2n^{-1}\alpha^2L^2s_Bs_{B^m}^2)Kn\sigma^2+6s_{B^m}nL\Delta \\&+160s_Bs_{B^m}^2L^2mn\sum_{k\in[m^*]}\EE[\norm{\Delta_x^{(k)}}_F^2]+80c^2m^2s_Bs_{B^m}^2\alpha^2L^2\sum_{t=0}^{T-1}\EE[\norm{\overline{\nabla f}^{(t)}}^2]
\end{align*}
When $\alpha \le \frac{1}{10cm\sqrt{s_Bs_{B^m}^2}L}$ and $m\ge 10c^{-2}s_Bs_{B^m}$, we have $1-160\alpha^2mL^2 s_B^2s_{B^m}^2\ge 0.5$. and the coefficient of $Kn\sigma^2$ becomes smaller than $8s_Bs_{B^m}^2$. Thus,
\begin{align}\label{ieq(app):delta y-1}
    &\sum_{k\in[m^*]}\EE[\norm{\Delta_y^{(k)}}_F^2]\le 16s_Bs_{B^m}^2Kn\sigma^2+12s_{B^m}nL\Delta \\
    &\quad+320s_Bs_{B^m}^2L^2mn\sum_{k\in[m^*]}\EE[\norm{\Delta_x^{(k)}}_F^2]+160c^2m^2s_Bs_{B^m}^2\alpha^2L^2\sum_{t=0}^{T-1}\EE[\norm{\overline{\nabla f}^{(t)}}^2].\nonumber
\end{align}
By taking this back to \eqref{ieq:delta 2}, we have
\begin{align}\label{ieq:delta 2-1}
    \sum_{k\in[m^*]}\EE[\Delta_2^{(k)}] &\le 40mn(1+320\alpha^2L^2 ms_B^2s_{B^m}^2)\sum_{k\in[m^*]}\EE[\norm{\Delta_x^{(k)}}_F^2] \\
    &\quad+ 480\alpha^2ms_{B^m}nL\Delta + (640 ms_B^2s_{B^m}^2 +20c^2m^2n^{-2})\alpha^2Kn\sigma^2\nonumber\\
    &\quad+20c^2\alpha^2m^2(1+160\alpha^2 L^2ms_B^2s_{B^m}^2)\sum_{t=0}^{T-1}\EE[\norm{\overline{\nabla f}^{(t)}}^2].\nonumber\\
    &\le 50mn\sum_{k\in[m^*]}\EE[\norm{\Delta_x^{(k)}}_F^2]+480\alpha^2ms_{B^m}nL\Delta\nonumber\\
    &\quad+ 21c^2m^2n^{-1}\alpha^2K\sigma^2+21c^2\alpha^2m^2\EE[\norm{\overline{\nabla f}^{(t)}}^2].\nonumber
\end{align}
where the second inequality uses $\alpha \le \frac{1}{10cmL}$ and $m\ge 4c^{-2}s_B^2s_{B^m}^2n^2$.

For $k\in[m^*]$, we have
\begin{align}\label{eq:rolling x}
    \Delta_x^{(k+m)} = (A^m-A_\infty)\Delta_x^{(k)}-\alpha \sum_{i=0}^{m-1}(A^{m-i}-A_\infty)\vy^{(k+i)}.
\end{align}
For $k\in [m^*]$, define 
\(\Delta_4^{(k)}:=\sum_{i=0}^{m-1}(A^{m-i}-A_\infty)\vy^{(k+i)}\), 
then we can iterate \eqref{eq:rolling x}:
\[\Delta_x^{(k+m)} =-\alpha \sum_{j=0}^{k/m}(A^{k-jm}-A_\infty)\Delta_4^{(jm)},\quad k\in[m^*].\]
Using the rolling sum lemma, we know that for $k\in[m^*]$, 
\[\quad\sum_{k\in[m^*]}\EE[\norm{\Delta_x^{(k+m)}}_F^2]\le \alpha^2s_{A^m}^2\sum_{k\in[m^*]}\EE[\EE_k[\norm{\Delta_4^{(k)}}_F^2]].\]
To further analyze $\EE_k[\Delta_4^{(k)}]$, $k\in [m^*]$, we expand it as
\begin{align*}
    \Delta_4^{(k)}&= \sum_{i=0}^{m-1}(A^{m-i}-A_\infty)\Delta_y^{(k+i)}+ n \sum_{i=0}^{m-1}(A^{m-i}-A_\infty)\pi_B\bar{g}^{(k+i)}\\
    &=\sum_{i=0}^{m-1}(A^{m-i}-A_\infty)(B^i-B_\infty)\Delta_y^{(k)} + n \sum_{i=0}^{m-1}(A^{m-i}-A_\infty)\pi_B\bar{g}^{(k+i)}\\
    &\quad + \sum_{j=i}^{m-1}(A^{m-i}-A_\infty)\delta_B^{(j-i)}(\vg^{(k+j)}-\vg^{(k)}).
\end{align*}
Note that 
\begin{align*}
    \norm{\sum_{i=0}^{m-1}(A^{m-i}-A_\infty)(B^i-B_\infty)}_2^2 &\le s_As_B, \,\,
    \sum_{i=0}^{m-1}\norm{(A^{m-i}-A_\infty)\delta_B^{(j-i)}}_2\le 2s_As_B,\\
    \sum_{i=0}^{m-1}\norm{(A^{m-i}-A_\infty)\delta_B^{(j-i)}}_2^2&\le 2s_As_B,\,\,
    \norm{\sum_{i=0}^{m-1}(A^{m-i}-A_\infty)\delta_B^{(j-i)}}_2^2\le 2s_As_B.
\end{align*}
Therefore, according to Lemma~\ref{lem:noise separation}, we can separate noiseless and noisy terms in $\sum_{i=0}^{m-1}\pi_B\bar{g}^{(k+i)}$ and $\sum_{j=i}^{m-1}(A^{m-i}-A_\infty)(B^{j-i}-B^{j-i+1})(\vg^{(k+j)}-\vg^{(k)})$, then apply Jensen's inequality. This gives
\begin{align*}
    &\EE_k[\norm{\Delta_4^{(k)}}_F^2]\le 3s_As_B\norm{\Delta_y^{(k)}}_F^2+3n^2s_A\sum_{i=0}^{m-1}\EE_k[\norm{\overline{\nabla f}^{(k+i)}}_F^2]+3ns_A\sigma^2\\
    &\quad+6s_As_B\sum_{j=i}^{m-1}\EE_k[\norm{\nabla f(\vx^{(k+j)})-\nabla f(\vx^{(k)})}_F^2]+12s_As_Bn\sigma^2\\
    &\le 3s_As_B\norm{\Delta_y^{(k)}}_F^2+3n^2s_A\sum_{i=0}^{m-1}\EE_k[\norm{\overline{\nabla f}^{(k+i)}}_F^2]+6s_As_BL^2\Delta_2^{(k)}+15s_As_Bn\sigma^2.
\end{align*}
Plug in our estimate on $\Delta_4^{(k)}$, we have
\begin{align*}
    &\quad\sum_{k\in[m^*]}\EE[\norm{\Delta_x^{(k+m)}}_F^2]\le \alpha^2s_{A^m}^2\sum_{k\in[m^*]}\EE[\EE_k[\norm{\Delta_4^{(k)}}_F^2]]\\
    &\le 15\alpha^2ns_As_Bs_{A^m}^2K\sigma^2 + 3\alpha^2n^2s_As_{A^m}^2\sum_{t=0}^{T-1}\EE[\norm{\overline{\nabla f}^{(t)}}_F^2]\\
    &\quad+ 6\alpha^2L^2s_As_Bs_{A^m}^2\sum_{k\in[m^*]}\EE[\Delta_2^{(k)}]+ 3\alpha^2s_As_Bs_{A^m}^2 \sum_{k\in[m^*]}\EE\norm{\Delta_y^{(k)}}_F^2
\end{align*}
Note that we start from consensual $x$, so $\Delta_x^{(0)}=0$. Plug in \eqref{ieq:delta 2-1} and \eqref{ieq(app):delta y-1} we have
\begin{align*}
    &\quad(1-480\alpha^2mL^2s_As_Bs_{A^m}^2n-960\alpha^2L^2ms_As_B^2s_{A^m}^2s_{B^m}^2n)\sum_{k\in[m^*]}\EE[\norm{\Delta_x^{(k)}}_F^2]\\
    &\le (15+126c^2\alpha^2L^2m^2n^{-2}+48s_Bs_{B^m}^2)\alpha^2s_As_Bs_{A^m}^2nK\sigma^2 \\
    &\quad+ (3+(126+480s_Bs_{B^m}^2)c^2\alpha^2L^2 m^2)\alpha^2n^2s_As_{A^m}^2\sum_{t=0}^{T-1}\EE[\norm{\overline{\nabla f}^{(t)}}_F^2]\\
    &\quad+ (36+2880\alpha^2L^2) s_As_Bs_{A^m}^2s_{B^m}  \alpha^2nL\Delta
\end{align*}
Using $\alpha \le \min\{\frac{1}{10\sqrt{s_Bs_{B^m}^2}cmL},\frac{1}{6L}\}$ and $m\ge 20 c^{-2}ns_As_Bs_{A^m}^2$, the coefficients can be simplified to 
\begin{align}\label{ieq(app):delta x-1}
    &\sum_{k\in[m^*]}\EE[\norm{\Delta_x^{(k)}}_F^2]\le 64\alpha^2s_As_B^2s_{B^m}^2s_{A^m}^2nK\sigma^2\\
    &\quad+7\alpha^2n^2s_As_{A^m}^2\sum_{t=0}^{T-1}\EE[\norm{\overline{\nabla f}^{(t)}}_F^2]+416\alpha^2s_As_Bs_{A^m}^2s_{B^m}  nL\Delta\nonumber
\end{align}
Taking~\eqref{ieq(app):delta x-1} back to~\eqref{ieq(app):delta y-1}, we obtain that
\begin{align*}
    &\sum_{k\in[m^*]}\EE[\norm{\Delta_y^{(k)}}_F^2]\le  (1+1280\alpha^2L^2mn^2s_As_B^2s_{B^m}^2s_{A^m}^2)16s_Bs_{B^m}^2Kn\sigma^2\\
    &\quad+(12+133120\alpha^2L^2s_As_B^2s_{B^m}s_{A^m}^2)s_{B^m}nL\Delta \\
    &\quad+160(c^2m^2s_Bs_{B^m}^2+140\alpha^2L^2n^2s_As_{A^m}^2s_Bs_{B^m}^2)\alpha^2L^2\sum_{t=0}^{T-1}\EE[\norm{\overline{\nabla f}^{(t)}}_F^2].\nonumber
\end{align*}
When $\alpha \le \frac{1}{10cm\sqrt{s_Bs_{B^m}^2}L}$ and $m\ge 10c^{-2}n^2s_As_{A^m}^2s_B$, we have
\begin{align}\label{ieq(app):delta y-2}
    \sum_{k\in[m^*]}\EE[\norm{\Delta_y^{(k)}}_F^2]&\le 32s_Bs_{B^m}^2Kn\sigma^2+26s_{B^m}nL\Delta\\
    &\quad+170 c^2m^2\alpha^2L^2s_Bs_{B^m}^2\sum_{t=0}^{T-1}\EE[\norm{\overline{\nabla f}^{(t)}}_F^2].\nonumber
\end{align}
Plug \eqref{ieq(app):delta x-1} into \eqref{ieq:delta 2-1}, we obtain that
\begin{align}\label{ieq:delta 2-2}
    &\sum_{k\in [m^*]}\EE[\Delta_2^{(k)}] \le (21c^2m^2n^{-2}+3200mns_As_B^2s_{A^m}^2s_{B^m}^2)\alpha^2nK\sigma^2\\
    &\quad+ 21280 s_As_Bs_{A^m}^2s_{B^m}\alpha^2mn^2L\Delta+(21c^2m+350n^3s_As_{A^m}^2)\alpha^2m\sum_{t=0}^{T-1}\EE[\norm{\overline{\nabla f}^{(t)}}_F^2].\nonumber\\
    &\le 181c^2m^2n^{-1}\alpha^2K\sigma^2 + 21280 s_As_Bs_{A^m}^2s_{B^m}\alpha^2mn^2L\Delta +38\alpha^2c^2m^2 \sum_{t=0}^{T-1}\EE[\norm{\overline{\nabla f}^{(t)}}_F^2],\nonumber
\end{align}
where the second inequality uses $m\ge 20c^{-2}n^3s_As_{A^m}^2s_B^2s_{B^m}^2$. Now we are ready to estimate the summation of error terms and prove Lemma~\ref{lem:consensus error}. 
Remember that $\Delta_1^{(k)}$ and $\Delta_2^{(k)}$ share the same upper bound, we utilize \eqref{ieq:delta 2-2}, \eqref{ieq(app):delta y-2} and obtain that
\begin{align*}
    &\quad\frac{2 L^2}{m}\sum_{k\in [m^*]}\EE[\Delta_1^{(k)}+\Delta_2^{(k)}]+ \frac{12 s_B^2}{c^2m^2}\sum_{k\in[m^*]}\EE[\norm{\Delta_y^{(k)}}_F^2]\\
    &\le C_{\sigma,1}  K\sigma^2+\frac{1}{m}C_{f,1}\sum_{t=0}^{T-1}\EE[\norm{\overline{\nabla f}^{(t)}}_F^2]+ C_{\Delta,1}\frac{ L\Delta}{m^2} \nonumber, 
\end{align*}
where
\begin{align*}
    C_{\sigma,1}&=724c^2mn^{-1}\alpha^2L^2+\frac{384s_B^3s_{B^m}^2}{m^2c^2}\\
    C_{f,1}&=152c^2\alpha^2L^2m^2+2040m\alpha^2L^2s_B^3s_{B^m}^2 \\
    C_{\Delta,1}&=85120 s_As_Bs_{A^m}^2s_{B^m}\alpha^2L^2m^2n^2+\frac{312s_B^2s_{B^m}n}{c^2}.
\end{align*}
When $\alpha \le \frac{1}{10cm\sqrt{s_Bs_{B^m}^2}L}$ and $m\ge c^{-2}30s_B^{1.5}$, we have 
\begin{equation}\label{eq:coefficients}
    C_{\sigma,1}\le \frac{73 c\alpha L}{n}+\frac{384s_B^3s_{B^m}^2}{m^2c^2}, \quad C_{f,1}\le 1, \quad C_{\Delta,1}\le \frac{10s_As_{A^m}^2n^2+312s_B^2s_{B^m}n}{c^2}.
\end{equation}
This finishes the proof of Lemma~\ref{lem:consensus error}. Finally, we leave a brief discussion on $s_{B^m}$. Recall that
\begin{align}\label{eq:s_B^m_def_in_appendix}
s_{B^m}= \max\{\sum_{i=0}^\infty \norm{B^{im}-B_\infty}_2, \sum_{i=0}^\infty \norm{B^{im}-B_\infty}_2^2\}
\end{align}
when $m$ increases, both $\sum_{i=1}^\infty \norm{B^{im}-B_\infty}_2$ and $\sum_{i=1}^\infty \norm{B^{im}-B_\infty}_2^2$ converge to zero exponentially fast, and $\norm{I-B_\infty}_2\ge 1$. Thus, $s_{B^m}$ converges to $\norm{I-B_\infty}_2^2$ exponentially fast as $m$ grows. \revision{The same argument applied to $A$ gives $s_{A^m}\to\norm{I-A_\infty}_2^2$ exponentially fast as $m$ grows.}

\subsection{Proof of Theorem~\ref{thm:main}}\label{app:subsec:main}
Sum up Lemma~\ref{lem:m-descent} for $k\in[m^*]$, then plug in Lemma~\ref{lem:consensus error}, we obtain that
\begin{align}\label{eq:dx-3}
    &\quad\sum_{k\in[m^*]}\EE[\norm{\nabla f(\hat{x}^{(k)})}^2]+\frac{1}{m}(1-C_{f,1})\sum_{t=0}^{T-1}\EE[\norm{\overline{\nabla f}^{(t)}}^2]\\
    &\le \frac{4(f(x^{(0)})-f(\hat{x}^{T}))}{\alpha cm}+\left(\frac{8\alpha c L}{n}+\frac{8\alpha Lns_B^2}{cm}+ C_{\sigma,1}\right)K\sigma^2+C_{\Delta,1}\frac{ L\Delta}{m^2}\nonumber.
\end{align}
Divide both sides by $K$ and utilize $C_{f,1}\le 1$, $f(x^{(0)})-f(\hat{x}^{T})\le \Delta$, $m\ge c^{-2}n^{2}s_B^2$ and $C_{\sigma,1}\le \frac{73 c\alpha L}{n}+\frac{384s_B^3s_{B^m}^2}{m^2c^2}$, $T=mK$, we have
\begin{align}\label{ieq(app):main-1}
    &\quad\frac{1}{K}\sum_{k\in [m^*]}\EE[\norm{\nabla f(\hat{x}^{(k)})}^2]\le \frac{4\Delta}{c\alpha T}  + \frac{100 c\alpha L}{n}\sigma^2+ \frac{384s_B^3s_{B^m}^2}{m^2c^2}\sigma^2 + C_{\Delta,1}\frac{L\Delta}{mT}.
\end{align}
Inequality~\eqref{ieq(app):main-1} holds when $\alpha \le \min\{\frac{1}{10cm\sqrt{s_Bs_{B^m}^2}},\frac{1}{6L}\}$ and $m\ge C_m$. 
Definition of $C_{\Delta,1}$ is given in \eqref{eq:coefficients}.
By choosing $m$ and $\alpha$ according to Theorem~\ref{thm:main}, we obtain the standard linear speedup result.

%$m\ge c^{-2}\max\{30s_B^{1.5}, 20n^3s_As_{A^m}^2s_B, etc..\}$. 

% To obtain a standard linear speedup result, we set
% \begin{align*}
%     m&=\lceil \max\{4c^{-1} s_B^{1.5}s_{B^m} \left(\frac{nT\sigma^2}{L\Delta}\right)^{0.25}, C_m\}  \rceil, \\
%     \alpha&=\min\left\{\frac{\sqrt{n\Delta}\sigma}{10c\sqrt{LT}}, \frac{1}{10cm\sqrt{s_Bs_{B^m}^2}L}, \frac{1}{6L} \right\},
% \end{align*}
% giving that 
% \begin{align*}
%     &\quad\min_{t\in\{0,1,\ldots,T-1\}}\EE[\norm{\nabla f(\hat{x}^{(t)}}^2]\le \frac{1}{K}\sum_{k\in [m^*]}\EE[\norm{\nabla f(\hat{x}^{(k)})}^2]\\
%     &\le \frac{44\sigma\sqrt{L\Delta}}{\sqrt{nT}} +\frac{L\Delta (m^{-1}C_{\Delta,1}+10m\sqrt{s_Bs_{B^m}^2}+c^{-1}) }{T}\\
%     &= \frac{44\sigma\sqrt{L\Delta}}{\sqrt{nT}} + \cO\left(\frac{L\Delta}{T}\right)^{\frac{3}{4}} + \cO\left(\frac{L\Delta}{T}\right).
% \end{align*}

\section{Experiment Details}\label{appendix:experiment}

In this section, we supplement experiment setup details.
\subsection{Weight Matrix Design}\label{app:subsec:weight matrix}
%We generated mixing matrices on six network topologies: the exponential, grid, and ring graphs are directed, while the random, geometric, and nearest neighbor graphs are undirected.
% We minimize the objective $f(x)=\frac{1}{n}\sum_{i=1}^n f_i(x)$ where each $f_i$ combines logistic loss with non-convex regularization. Data generation follows three steps: (1) sample optimal parameters $x_{\text{opt}} \sim \mathcal{N}(0,I_d)$ and create feature matrix $H \in \mathbb{R}^{L_{\text{total}}\times d}$; (2) compute binary labels via randomized thresholding; (3) partition data equally across $n$ nodes with $L_i = L_{\text{total}}/n$ samples each.
We use the following procedure to generate a row/column-stochastic weight matrix. For a fixed number of nodes $n$, we begin by initializing a 0-1 adjacency matrix $W$ as the identity matrix $I$ to account for self-loops on each node. Subsequently, based on the specific network topology definition, we set $W_{i,j}=1$ if there is a directed edge from node $i$ to node $j$. After constructing this initial 0-1 matrix $W$, we then randomly assign integer weights from $1$ to $9$ to all entries where $W_{i,j}=1$. Finally, we perform either row normalization or column normalization on this weighted matrix to obtain a row-stochastic matrix $A$ or a column-stochastic matrix $B$, respectively.

\subsection{Synthetic data}\label{app:subsec:synthetic}
For experiments conducted on the synthetic dataset, we minimize the objective:
\begin{align}
    f(x)=\frac{1}{n}\sum_{i=1}^n f_i(x),\quad x\in \mathbb{R}^{d}, \label{eq::obj_func}
\end{align}
where
\begin{equation*}
    f_{i}(x)= \underbrace{\frac{1}{L_i}\sum_{l=1}^{L_i} \ln\left(1+\exp(-y_{i,l}h_{i,l}^\top x)\right)}_{\text{Logistic Loss}} +\underbrace{\rho\sum_{j=1}^d\frac{[x]_j^2}{1+[x]_j^2}}_{\text{Non-convex Regularization}},\quad L_i=L_{\text{Local}}/n. 
\end{equation*}

\textbf{Data Synthesis Process.} We generate the synthetic dataset through the following three-step procedure. \textbf{(1) Global Data Generation:} We first generate optimal parameters $x_{\text{opt}} \sim \mathcal{N}(0,I_d)$ and create a global feature matrix $H \in \mathbb{R}^{L_{\text{total}}\times d}$ with entries $h_{l,j} \sim \mathcal{N}(0,1)$. Labels $Y\in \mathbb{R}^{L_{\text{total}}}$ with $y_l \in \{-1,+1\}$ are computed via randomized thresholding:
\begin{equation}
    y_l = \begin{cases}
        1 & \text{if } 1/z_l > 1 + \exp(-h_l^\top x_{\text{opt}}) \\
        -1 & \text{otherwise}
    \end{cases},\quad z_l \sim \mathcal{U}(0,1)\nonumber
\end{equation}
\textbf{(2) Data Distribution:} We partition $H$ and $Y$ equally across $n$ nodes using $H^{(i)} = H[iM:(i+1)M,:]$ and $Y^{(i)} = Y[iM:(i+1)M]$ where $M = L_{\text{total}}/n$, requiring $L_{\text{total}}$ to be divisible by $n$. \textbf{(3) Local Initial Points:} Each node $i$ starts from $x_i^\star = x_{\text{opt}} + \varepsilon_i$ where $\varepsilon_i \sim \mathcal{N}(0,\sigma_h^2I_d)$ with $\sigma_h=10$.

\textbf{Gradient Computation.} At each iteration, node $i$ independently computes its stochastic gradient by randomly sampling a minibatch of size $B$ from its local dataset of size $L_i = L_{\text{total}}/n$:
\begin{equation}
    \nabla f_{i,B}(x) = \underbrace{-\frac{1}{B}\sum_{b=1}^B \frac{y_{i,b}h_{i,b}}{1+\exp(y_{i,b}h_{i,b}^\top x)}}_{\text{Logistic Loss Gradient}} + \underbrace{\rho\sum_{j=1}^d \frac{2[x]_j}{(1+[x]_j^2)^2}}_{\text{Non-convex Regularization}}\nonumber
\end{equation}
where the minibatch $\{h_{i,b}, y_{i,b}\}_{b=1}^B$ is drawn uniformly without replacement from the $L_i$ local samples. The gradient at each node $i$ is computed on its local parameter $x_i \in \mathbb{R}^{d}$.

\textbf{Implementation Details.} We use global dataset size $L_{\text{total}}=204800$ with batch size $B = 200$, dimension $d=10$, and regularization parameter $\rho=0.01$. Node configurations include $n \in \{4,6,8,9,12,16,18,24,25\}$ with least common multiple (LCM) of $3,600$. Note that $L_{\text{total}}=1,440,000=3,600\cdot200\cdot2=\text{LCM}\cdot B \cdot 2$.

\textbf{Evaluation Metric.} We track the gradient norm $\|n^{-1}\one^\top \nabla f(\vx^{(k)})\|$ where
\begin{align*}
    n^{-1}\one^\top \nabla f(\vx^{(k)}) = \frac{1}{n}\sum_{i=1}^n \nabla f_{i,M}(\bm{x}_i^{(k)}), \quad M=L_{\text{Total}}/n.
\end{align*}
Results represent averages over 20 independent trials with fixed random seed 42 for reproducibility.

\subsection{Neural Network Experiments}\label{appendix::exp_mnist}
We evaluate our \pushpull method on two benchmark tasks: MNIST classification using three-layer fully connected networks and CIFAR-10 classification using ResNet-18.

\textbf{Data Distribution.} To simulate heterogeneous data distributions across nodes, we employ Dirichlet sampling. For each node $k \in \{1, \dots, n\}$, we sample a class proportion vector $\mathbf{q}^k = (q_{k,0}, q_{k,1}, \dots, q_{k,9}) \sim \text{Dirichlet}(\alpha \cdot \mathbf{1}^{10})$, where $\mathbf{1}^{10}$ is a 10-dimensional vector of ones. Each $q_{k,c}$ represents the proportion of samples from class $c$ assigned to node $k$, satisfying $q_{k,c}\ge 0$ and $\sum_{c=0}^{9} q_{k,c}=1$. The target number of samples of class $c$ for node $k$ is $N_{k,c}^{\text{target}} = q_{k,c}\cdot M/n$, resulting in approximately $60,000/n$ samples per node for MNIST and $50,000/n$ for CIFAR-10. We set $\alpha = 0.9$ to create heterogeneous distributions, as illustrated in Figure~\ref{fig::heat_map} for 16 nodes.

\textbf{Evaluation Metrics.} We measure convergence using the normalized gradient norm computed after each batch. Specifically, we average the gradients from all $n$ nodes and normalize by the square root of the total parameter count:
\begin{align}
    \left\| \frac{1}{n}\sum_{i=1}^n \mathbf{g}_{i,B_j} \right\|_2\Big/ \sqrt{P}, \label{eq::avg_grad_norm}
\end{align}
where $\mathbf{g}_{i,B_j}$ is the gradient computed on the $j$-th batch of node $i$, and $P$ is the number of parameters. For MNIST with three-layer FCNs, $P=1,863,690$; for CIFAR-10 with ResNet-18, $P=11,173,962$.

\textbf{Experimental Setup.} We test our method on various network topologies including directed exponential and grid graphs, as well as undirected random graphs. Training uses learning rate $lr=0.005$ and batch size $B=128$ for MNIST, $B=128$ for CIFAR-10, consistent across all topologies and node configurations. Performance is compared against single-node baseline training ($n=1$). Figures~\ref{fig::linear_speedup_synthetic} and \ref{fig::linear_speedup_mnist} demonstrate that Push-Pull achieves linear speedup on both synthetic data and real datasets across directed and undirected topologies. Figure~\ref{fig::acc_cifar10} shows the validation accuracy comparison for CIFAR-10.

\begin{figure}[ht]
    \centering
    \includegraphics[width=0.85\linewidth]{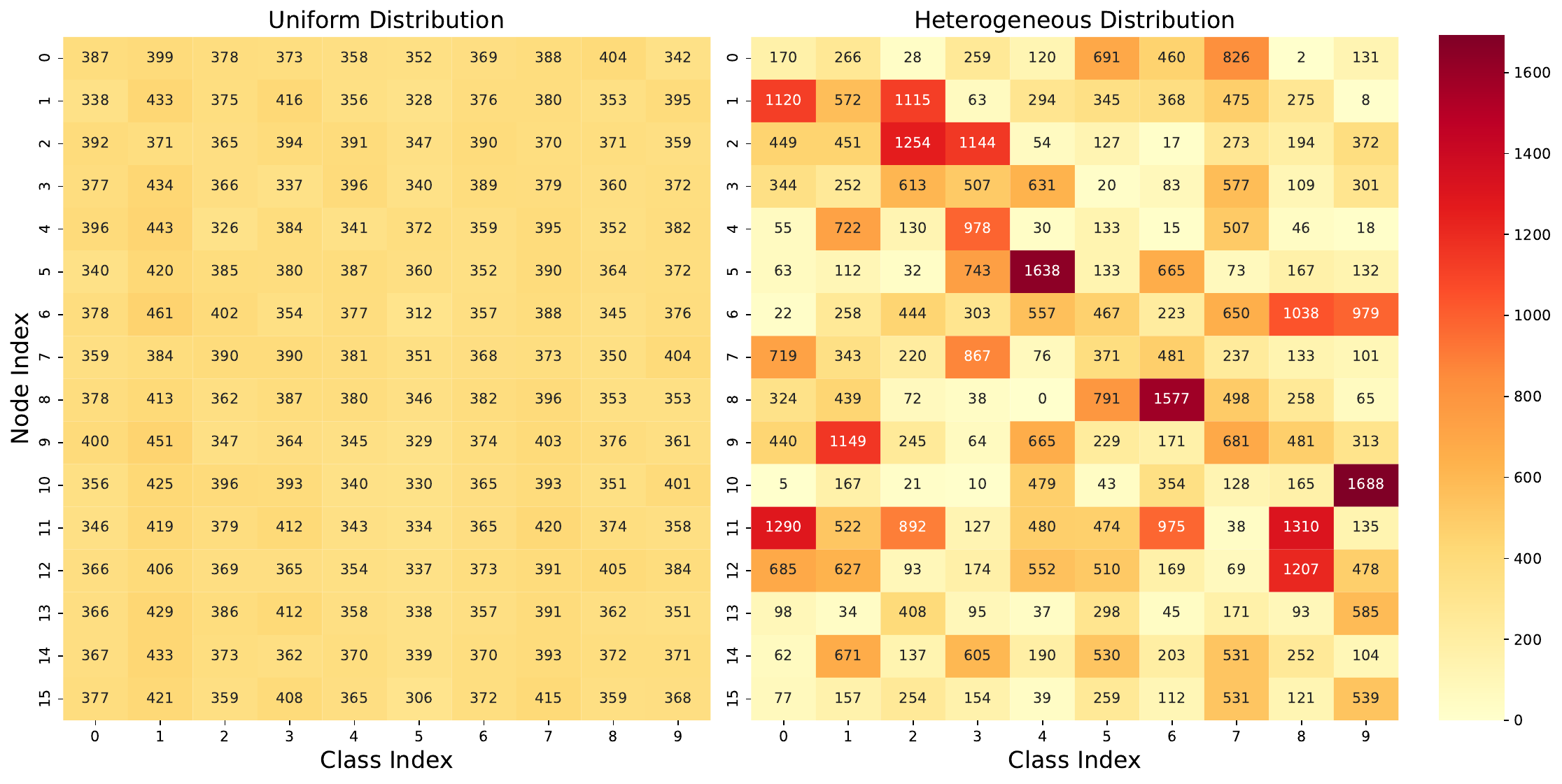}
    \caption{Comparison of uniform (left) and heterogeneous (right) data distributions across $16$ nodes ($\alpha = 0.9$).}
    \label{fig::heat_map}
\end{figure}

\section{Network Topology Visualizations}\label{app:sec:topologies}

Figure~\ref{fig::topologies} visualizes the six network topologies used in our experiments.
Panels~(a)--(d) show directed graphs with arrows indicating the direction of information flow;
panels~(e)--(f) show undirected graphs where each edge denotes bidirectional communication.

\begin{figure}[ht]
\centering
\begin{minipage}{0.31\linewidth}\centering
\includegraphics[width=\linewidth]{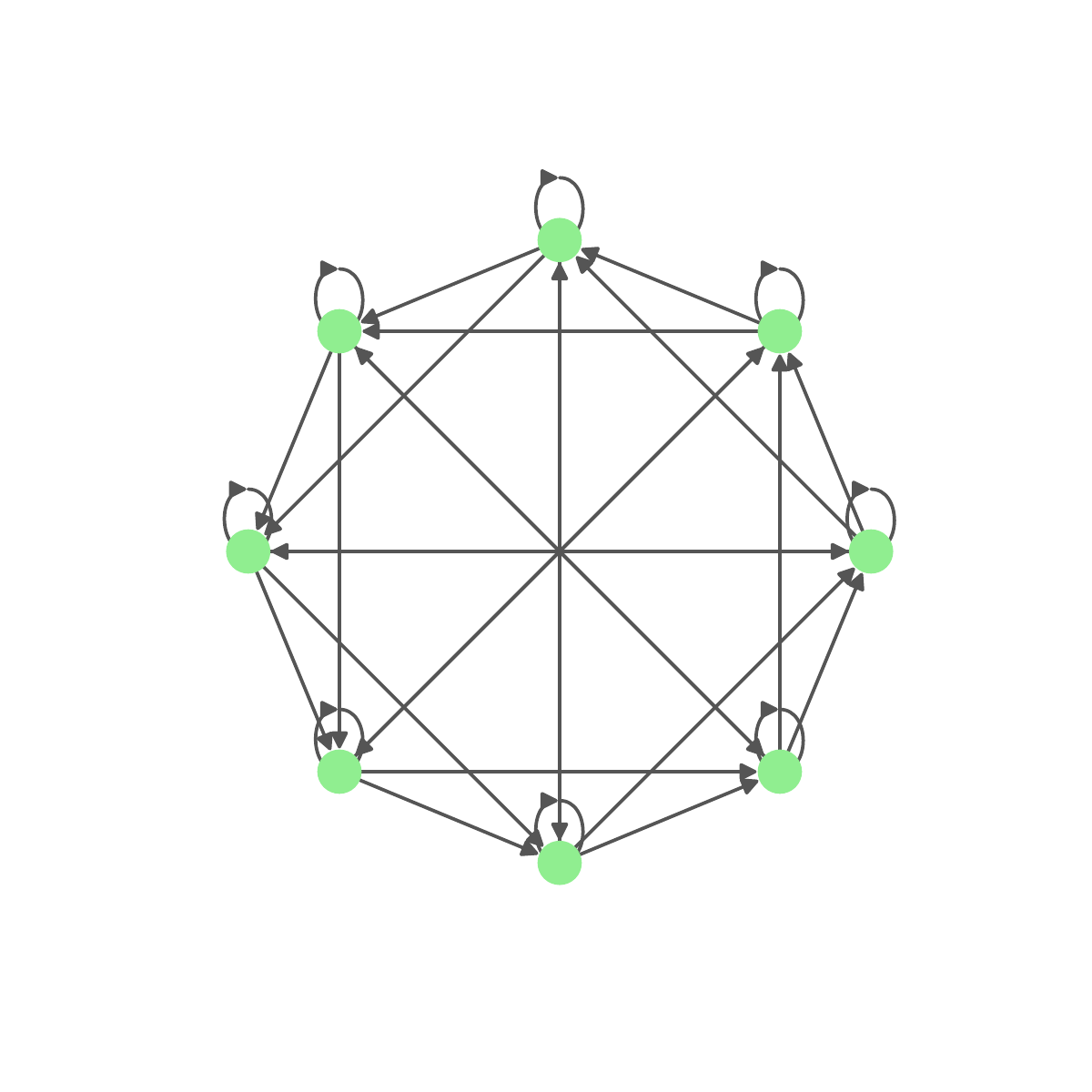}\\
\footnotesize (a) Exponential, $n=8$
\end{minipage}\hfill
\begin{minipage}{0.31\linewidth}\centering
\includegraphics[width=\linewidth]{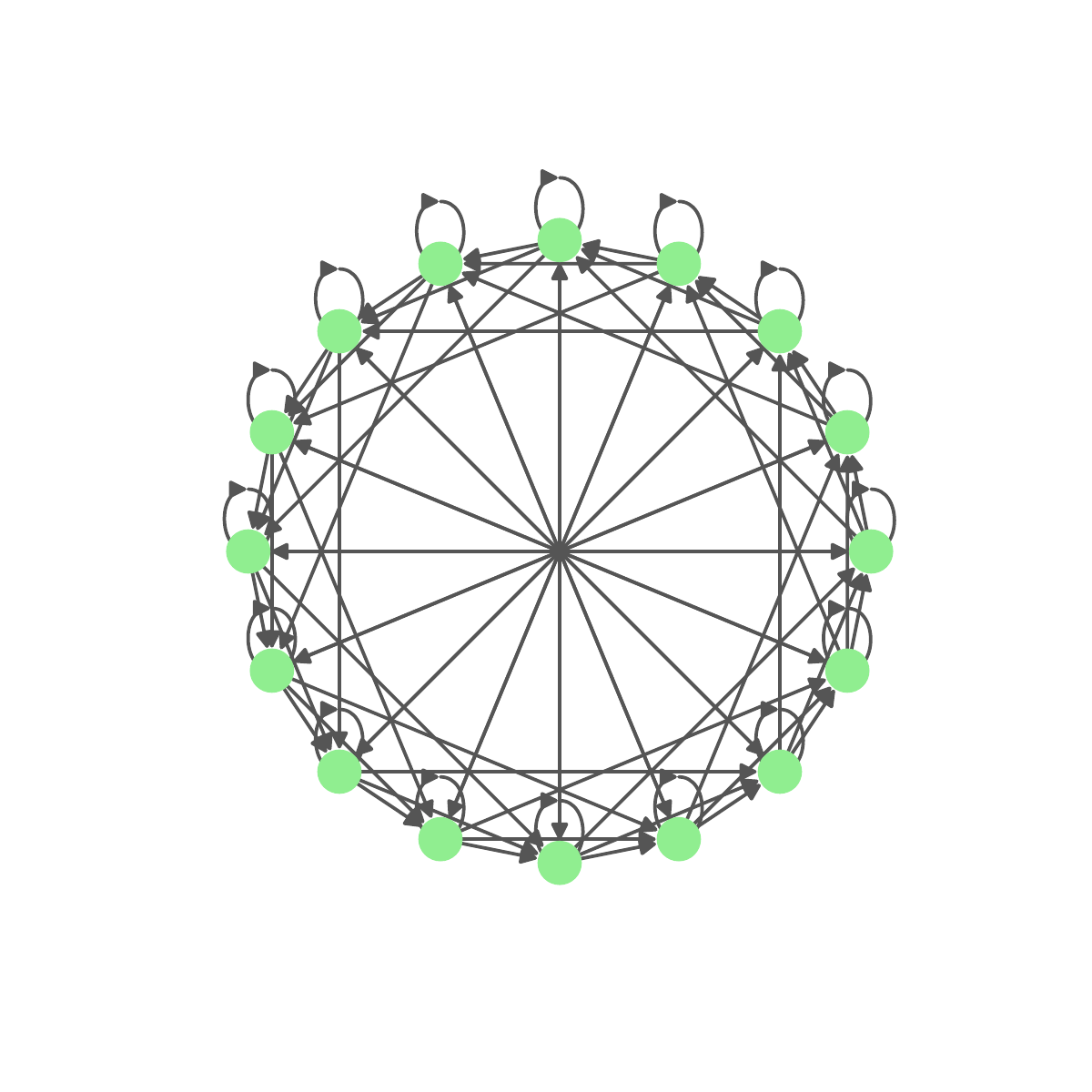}\\
\footnotesize (b) Exponential, $n=16$
\end{minipage}\hfill
\begin{minipage}{0.31\linewidth}\centering
\includegraphics[width=\linewidth]{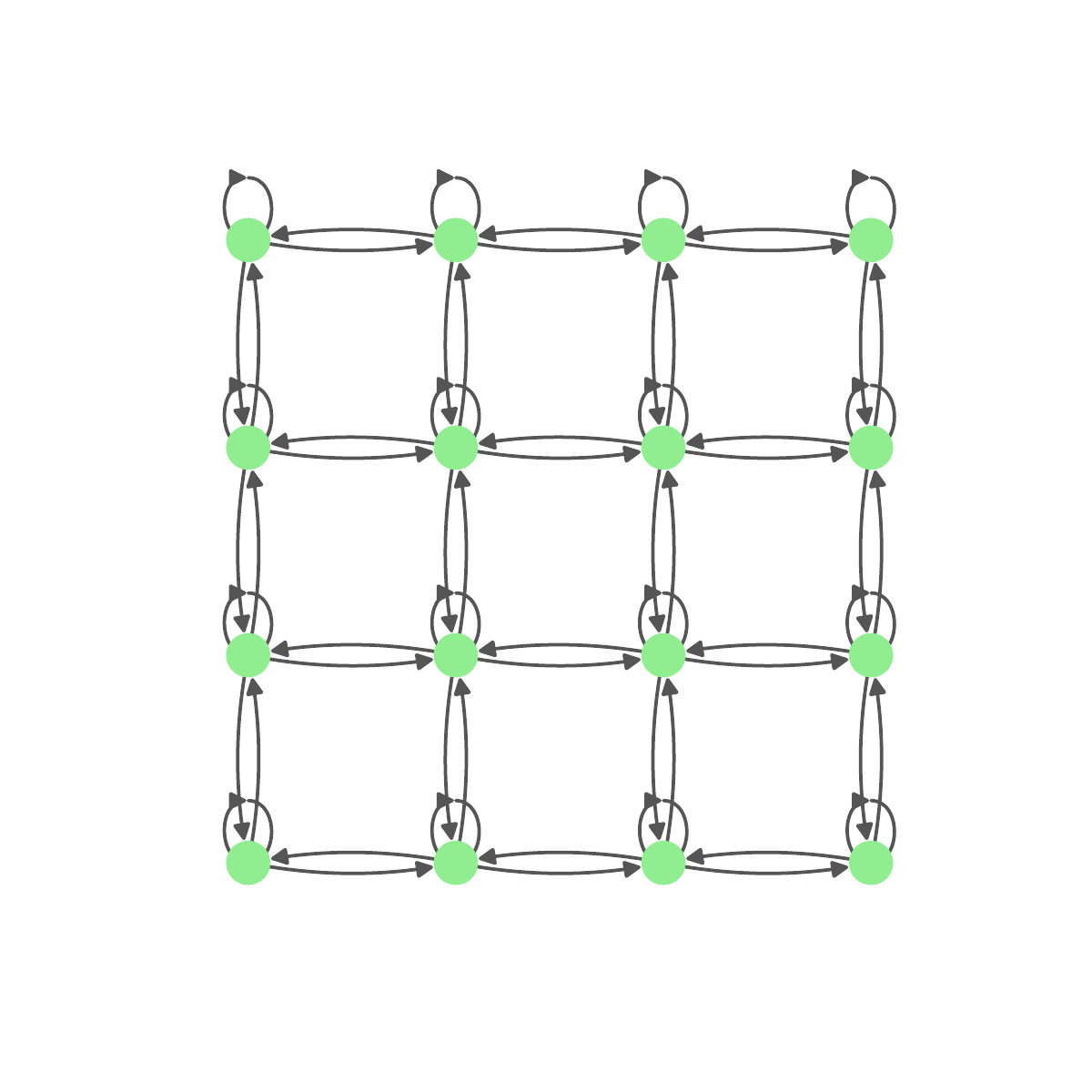}\\
\footnotesize (c) Grid, $n=16$
\end{minipage}

\vspace{2mm}

\begin{minipage}{0.31\linewidth}\centering
\includegraphics[width=\linewidth]{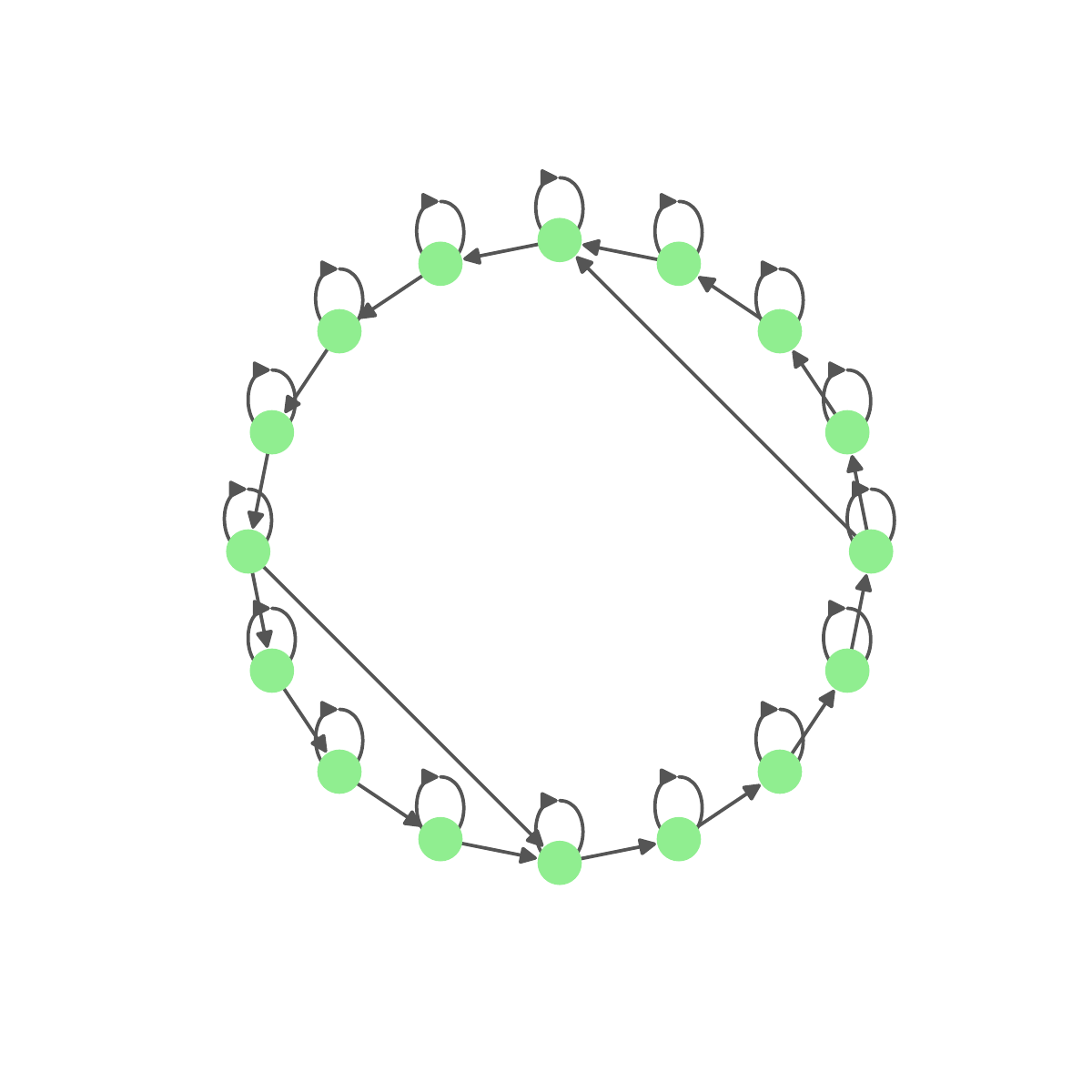}\\
\footnotesize (d) Ring, $n=16$
\end{minipage}\hfill
\begin{minipage}{0.31\linewidth}\centering
\includegraphics[width=\linewidth]{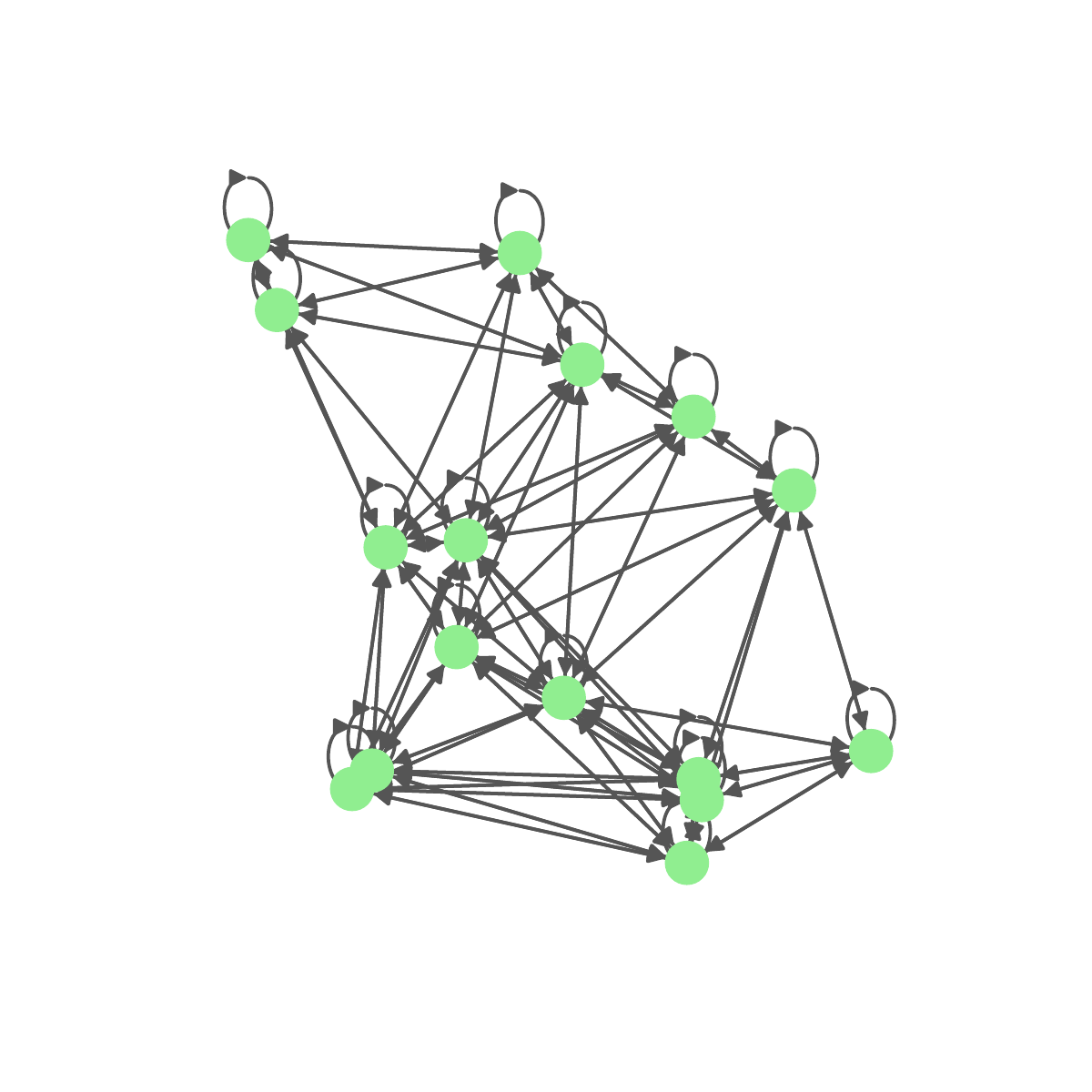}\\
\footnotesize (e) Geometric, $n=16$
\end{minipage}\hfill
\begin{minipage}{0.31\linewidth}\centering
\includegraphics[width=\linewidth]{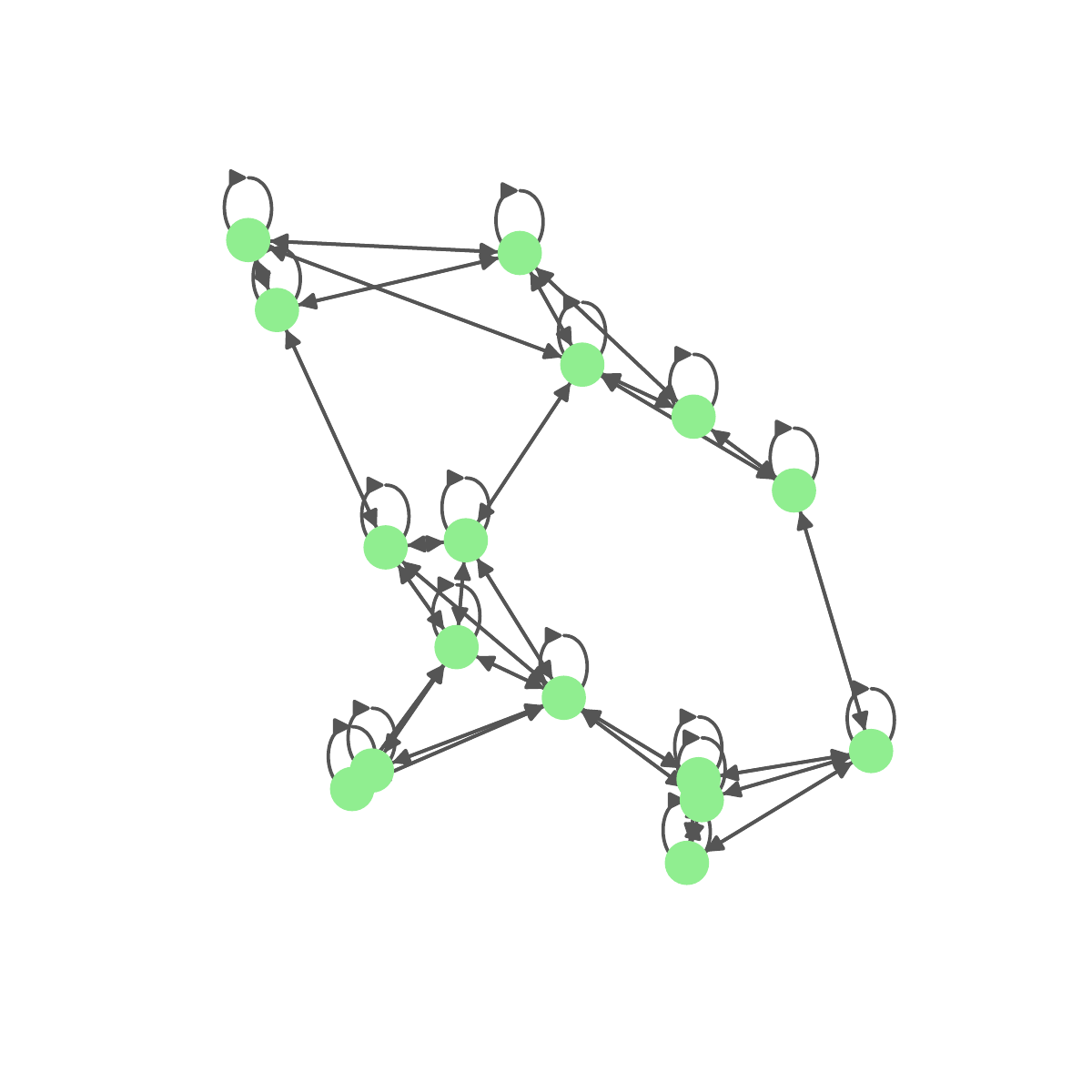}\\
\footnotesize (f) Nearest-neighbor, $n=16$
\end{minipage}
\caption{Visualization of the network topologies used in our experiments.
Panels~(a)--(d) display the directed graphs (exponential, grid, ring) with arrows indicating
the direction of communication; panels~(e)--(f) display the undirected graphs (geometric,
nearest-neighbor) where each edge denotes bidirectional communication.
The exponential graph is shown at both $n=8$ and $n=16$ to illustrate how the topology
scales with the number of nodes.}
\label{fig::topologies}
\end{figure}

\end{document}